\documentclass[a4paper, 12pt]{article}

\usepackage{xcolor}
\usepackage{bbm}
\usepackage{amsmath,amssymb,mathrsfs}
\usepackage{amsthm} 
\usepackage{amsbsy} 
\usepackage{graphicx}
\usepackage{ulem}  

\usepackage{algorithm}
\usepackage{algorithmic}
\usepackage{multirow}

\newtheorem*{Claim}{Claim}

\newtheorem{Lemma}{Lemma}[section]
\newtheorem{Proposition}{Proposition}[section]
\newtheorem{Theorem}[Lemma]{Theorem}
\newtheorem{Corollary}[Lemma]{Corollary}

\theoremstyle{definition}

\DeclareMathOperator{\R}{\mathbb{R}}
\DeclareMathOperator{\C}{\mathbb{C}}

\newcommand{\dfix}[2]{\mathfrak{#1}_{#2}}
\newcommand{\dproj}[1]{\mathrm{proj}\big[ #1 \big]} 

\newcommand{\dep}[2]{\mathrm{EP}[#1]\left(#2\right)} 
\newcommand{\dEP}[2]{\mathrm{EP}[#1]\left(\oplus #2\right)} 
\newcommand{\drm}[2]{\mathrm{#1} \hspace{0.6mm} #2} 
\newcommand{\dAM}[1]{\mathbf{A}\left( #1 \right)} 
\newcommand{\dChMat}[1]{\mathbf{R}_{ #1 }}

\newcommand{\dbipart}[2]{[\mathfrak{#1}_{#2'},\mathfrak{#1}_{#2''}]}
\newcommand{\dubipart}[3]{[\mathfrak{#1}_{#2},\mathfrak{#1}_{#3}]}

\title{\bf\upshape On the Automorphism Group of a Graph }
\date{}

\author{ Wenxue Du \\
{\small  \it School of Mathematical Science, Anhui University, Hefei,
230039, China} \\
{\footnotesize E-mail:  wenxuedu@gmail.com }
}

\begin{document}
\maketitle

\begin{abstract} 
An automorphism of a graph $G$ with $n$ vertices is a bijective map $\phi$ from $V(G)$ to itself such that $\phi(v_i)\phi(v_j)\in E(G)$ $\Leftrightarrow$ $v_i v_j\in E(G)$ for any two vertices $v_i$ and $v_j$ of $G$. Denote by $\mathfrak{G}$ the group consisting of all automorphisms of $G$. As well-known, the structure of the action of $\mathfrak{G}$ on $V(G)$ is represented definitely by its block systems. On the other hand for each permutation $\sigma$ on $[n]$, there is a natural action on any vector $\pmb{v}=(v_1,v_2,\ldots,v_n)^t\in \R^n$ such that $\sigma\pmb{v}=(v_{\sigma^{-1}1},v_{\sigma^{-1}2},\ldots,v_{\sigma^{-1} n})^t$. Accordingly, we actually have a permutation representation of $\mathfrak{G}$ in $\R^n$.  In this paper, we establish the some connections between block systems of $\mathfrak{G}$ and its irreducible representations, and by virtue of that we finally devise an algorithm outputting a generating set and all block systems of $\mathfrak{G}$ within time $n^{C \log n}$ for some constant $C$.
\end{abstract}

\vspace{3mm}

\section{Introduction}

Let $G$ and $H$ be two simple graphs. A bijective map $\phi:V(G)\rightarrow V(H)$ is called an {\it isomorphism} between $G$ and $H$ if $\phi(v_i)\phi(v_j)\in E(H)$ $\Leftrightarrow$ $v_i v_j\in E(G)$ for any two vertices $v_i$ and $v_j$ of $G$. If there is such an isomorphism between $G$ and $H$, we say that $G$ and $H$ are isomorphic, denoted by $G \cong H$.  Naturally, for two given graphs we are interested in whether or not they are isomorphic, that is the the problem of {\it Graph Isomorphism} (GI). 

One of striking facts about GI is the following established by Whitney in 1930s.

\begin{Theorem}\label{Thm-WhitneyIsomorphism} 
Two connected graphs are isomorphic if and only if their line graphs are isomorphic, with a single exception: $K_3$ and $K_{1,3}$, which are not isomorphic but both have $K_3$ as their line graph.
\end{Theorem}

Clearly, the relation above offers a reduction of GI from general graphs to a special class of graphs - line graphs, which accounts only for a small fraction of all graphs. This fact suggests that GI may not be very hard. In fact, GI is well solved from practical point of view and there are a number of efficient algorithms available \cite{McP}. Even from worst-case point of view, GI may not be as hard as NP-complete problems. As a matter of fact GI is not NP-complete unless the polynomial hierarchy collapses to its second level \cite{BHZ,Schonig}. On the other hand, however, there is no efficient algorithm for general graphs in worst-case analysis, while for restricted graph classes we have efficient algorithms, for graphs with bounded degree \cite{Luks} and for graphs with bounded eigenvalue multiplicity \cite{BaGrMu} for instance. L. Babai \cite{Babai} recently declared an algorithm resolving GI for all graphs within time $\exp\big\{ (\log n)^{O(1)} \big\}$ in worst-case analysis.

In the case that two graphs considered are the same, a bijection $\phi$ is called a {\it permutation} of the vertex set $V$, and in the case that $\phi$ preserves all adjacency relations among vertices, it is called an {\it automorphism} of the graph $G$. It is plain to see that all permutations of $V$ can  form a group under composition of maps, that is the {\it symmetric group on $V$} and  denoted by $\drm{Sym}{V}$. A {\it permutation group} is a subgroup of some symmetric group. Obviously, all automorphisms of $G$ form a permutation group under composition of maps, which is denoted by $\mathrm{Aut} \hspace{0.6mm} G$, or $\mathfrak{G}$ for short. 

The problem of finding a generating set of $\mathfrak{G}$ is the {\it Automorphism Group Problem} (AG), which has a close relation to GI. In fact, GI can be reduced to AG, by dealing with AG with respect to a new graph $G$ constructed by combining two originally given graphs $G_1$ and $G_2$ such that $V(G) = V(G_1) \cup V(G_2)$ and two vertices of $G$ are adjacent if and only if they are adjacent in $G_1$ or in $G_2$. Apparently, if we can resolve AG for $G$ then we can determine whether those two graphs $G_1$ and $G_2$ are isomorphic or not. 

However, only by a generating set of $\mathfrak{G}$ we cannot see those key features possessed by $\mathfrak{G}$, so we need to reveal more information about $\mathfrak{G}$. In order to analyze the structure of $\mathfrak{G}$, one of effective ways is to investigate the action of $\mathfrak{G}$ on an object. In our case, $\mathfrak{G}$ has a natural action on the vertex set $V$, which means for any vertex $v$ we can get some vertex in $V$ by $\sigma v$, where $\sigma$ belongs to $\mathfrak{G}$. Clearly there are two possibilities: 
\begin{equation}\label{Equ-ActionGonV}
\mbox{either } \sigma v = v \mbox{ or } \sigma v \neq v.
\end{equation}
In this way, we can obtain a subset $\{ \sigma v : \sigma \in \mathfrak{G} \}$ of $V$, which is called an {\it orbit} of $\mathfrak{G}$.
 
It is interesting that there might be some subsets of $V$ possessing that property (\ref{Equ-ActionGonV}). Let $B$ be a non-empty subset of some orbit $T$ of a $\mathfrak{G}$, which is called a {\it block} for $\mathfrak{G}$ if for any $\gamma\in\mathfrak{G}$, either $\gamma B = B$ or $\gamma B \cap B = \emptyset$. Obviously, any element $y$ of $T$ and the orbit $T$ itself are blocks for $\mathfrak{G}$. If the group $\mathfrak{G}$ has only two such kinds of blocks in $T$ we say the action of $\mathfrak{G}$ on $T$ is {\it primitive}, otherwise {\it imprimitive.} Apparently, the family of subsets $\{ \gamma B \mid \gamma \in \mathfrak{G} \}$ forms a partition of $T$, that is call the {\it system of blocks containing $B$}. As well-known orbits and block systems are vitally important in characterizing the structure of a permutation group.

Suppose the vertex set $V$ of the graph $G$ is $[n] = \{1,2,\ldots,n\}$ and  $\pmb{v} = (v_1,v_2,\ldots,v_n)^t$ is a vector in $\R^n$. Each permutation $\sigma$ in $\drm{Sym}{[n]}$ can naturally  act on $\pmb{v}$ such that 
\begin{equation}\label{Equ-TheRepresent}
\sigma\pmb{v}=(v_{\sigma^{-1}1},v_{\sigma^{-1}2},\ldots,v_{\sigma^{-1} n})^t.
\end{equation} 
Accordingly, any permutation $\sigma$ in $\drm{Sym}{[n]}$ can be regarded, through the action on vectors, as a linear operator on $\R^n$, which is denoted by $\mathcal{T}_{\sigma}$. In terms of linear representation of a group, it is indeed a permutation representation of $\drm{Sym}{[n]}$ in $\R^n$, and clearly every permutation group possesses such a representation.

Recall that a non-trivial subspace $U$ of $\R^n$ is said to be {\it $\mathcal{T}$-invariant} if $\mathcal{T} U \subseteq U$, where $\mathcal{T}$ is a linear operator on $\R^n$. Suppose $\mathfrak{G}$ is a permutation group in $\drm{Sym}{[n]}$. A subspace $U$ is said to be {\it $\mathfrak{G}$-invariant} if $U$ is $\mathcal{T}_{\sigma}$-invariant for all $\sigma \in \mathfrak{G}$. We are particularly interested in those minimally $\mathfrak{G}$-invariant subspaces which turns out to be truly useful in representing the action of $\mathfrak{G}$ on $[n]$. If a $\mathfrak{G}$-invariant subspace contains no proper subspace being also $\mathfrak{G}$-invariant, we say the subspace {\it irreducible}. Let $W$ be an irreducible subspace for $\mathfrak{G}$. Then the permutation representation of $\mathfrak{G}$ could be restricted to $W$, so $W$ is called an {\it irreducible representation} (IR) of $\mathfrak{G}$. The permutation representation of $\drm{Sym}{[n]}$ in $\R^n$ possesses only two IRs: $\drm{span}{\{\pmb{1}\}}$, the subspace spanned by the vector $\pmb{1}$, and its orthogonal complement in $\R^n$. In general however IRs of a permutation group can be tremendously colorful. 

The {\it adjacency matrix} of a simple graph $G$ on $n$ vertices is a $n\times n$ (0,1)-matrix where an entry $a_{ij}$ of the matrix is equal to 1 if and only if the two vertices $v_i$ and $v_j$ are adjacent. We denote the adjacency matrix of $G$ by $\mathbf{A}(G)$, or $\mathbf{A}$ for short, in the case that one can easily identify the corresponding graph from the context.

\vspace{3mm}
The problem of finding a generating set of $\mathfrak{G}$ and of determining all block systems and a decomposition of all eigenspaces of $\mathbf{A}(G)$ into IRs of $\mathfrak{G}$ is called the {\it structure problem of an automorphism group} (SAG)

A function $f(n)$ is called {\it feeble-exponential bounded} if for large enough $n$ there exists a constant $C$ so that $f(n) \leq n^{C \log n}$. We use {\it feeble-exponential time} to refer to feeble-exponentially bounded time. Our main result is the following. 

\begin{Theorem}\label{MainTheorem}
The SAG can be solved in feeble-exponential time.   
\end{Theorem}

As one can see in the last section, our algorithm can not only cope with SAG for simple graphs but also for  non-simple ones, in feeble-exponential time,  with some weight function on $V(G)\cup E(G)$, as long as the adjacency matrix of the graph is symmetric.

There are two key problems in resolving SAG:
\begin{enumerate}

\item how to determine whether or not two vertices considered are symmetric, {\it i.e.,} in the same orbit of $\mathfrak{G}$, and in the case of being symmetric to figure out one automorphism moving one of two vertices to another;

\item how to find out one block system of $\mathfrak{G}$ and a decomposition of an eigenspace of $\mathbf{A}(G)$ into IRs of $\mathfrak{G}$.

\end{enumerate}

\subsection{Automorphisms between two given vertices}

Let us begin with an algebraic description of automorphisms of $G$. We call a (0,1)-square matrix a {\it  permutation matrix} if in each row and each column there is exactly one entry which is equal to 1. It is easy to check that the matrix $\mathbf{P}_{\sigma}$ of the operator $\mathcal{T}_{\sigma}$ with respect to the standard basis $\pmb{e}_1,\ldots,\pmb{e}_n$ is a permutation matrix, where each $\pmb{e}_i$ $(i=1,\ldots,n)$ has exactly one non-trivial entry on $i$th coordinate which is equal to 1, and all other entries of $\pmb{e}_i$ are equal to 0. One moment's reflection shows that
\begin{equation}\label{Equation00} 
\mbox{a permutation }\sigma\mbox{ of } [n] \mbox{ is an automorphism of }G\mbox{ if and only if }\mathbf{P}_{\sigma}^{-1}\mathbf{A}\mathbf{P}_{\sigma}=\mathbf{A}.
\end{equation}

Evidently, the adjacency matrix $\mathbf{A}(G)$ is symmetric and thus $\mathbf{A}(G)$ can be viewed as the matrix of a self-adjoint operator $\mathcal{T}_G$ on $\R^n$ with respect to an ordered basis $\pmb{b}_1,\ldots,\pmb{b}_n$, {\it i.e.,} $\mathcal{T}_G(\pmb{b}_1,\ldots,\pmb{b}_n)=(\pmb{b}_1,\ldots,\pmb{b}_n)\mathbf{A}$.  Accordingly,
$$\mathcal{T}_G \pmb{b}_i=\sum_{~\hspace{5mm} v_j\sim v_i} \pmb{b}_j,~~~i=1,\ldots,n,$$
where the symbol $v_j\sim v_i$ indicates that the two vertices $v_j$ and $v_i$ are adjacent in $G$, and  $1\le i, j \le n$. In other words $\mathcal{T}_G$ provides  the adjacency information about the graph $G$ and thus the standard basis $\pmb{e}_1,\ldots,\pmb{e}_n$ would be appropriate for $\mathbf{A}(G)$, since in that case, one can find out in virtue of $\mathcal{T}_G \pmb{e}_i$ the neighbors of the vertex $v_i$, $i=1,\ldots,n$.

We now can formulate another way of describing automorphisms of $G$ via the eigenspaces of $\mathcal{T}_G$.

\begin{Lemma}\label{Lem-AutomorphismAndOperator} Let $G$ be a graph with the vertex set $[n]$ and let $\sigma$ be a permutation in $\drm{Sym}{[n]}$. Then $\sigma$ is an automorphism of $G$ if and only if every eigenspace of $\mathcal{T}_G$ is $\mathcal{T}_{\sigma}$-invariant. \end{Lemma}

Recall that the $n$-dimensional vector space $\R^n$ is endowed with the {\it  inner product} $\langle\cdot,\cdot\rangle$ such that $\langle\pmb{u},\pmb{v}\rangle=\pmb{v}^t\pmb{u}=\sum_{i=1}^n u_i\cdot v_i$ for any vectors $\pmb{u}=(u_1,\ldots,u_n)^t$ and $\pmb{v}=(v_1,\ldots,v_n)^t$ in $\R^n$. Two vectors $\pmb{u}$ and $\pmb{v}$ are said to be {\it  orthogonal} if $\langle\pmb{u},\pmb{v}\rangle=0$.
Since the matrix $\mathbf{A}(G)$ of $\mathcal{T}_G$ is symmetric, the operator $\mathcal{T}_G$ is self-adjoint. In accordance with the real spectral theorem (see \cite{Axler} for example), there is an orthonormal basis of $\R^n$ consisting of eigenvectors of $\mathcal{T}_G$.

\begin{proof}[\bf Proof]  As we have pointed out before, $\mathbf{A}$ and $\mathbf{P}_{\sigma}$ are the matrices, respectively, of two linear operators $\mathcal{T}_G$ and $\mathcal{T}_{\sigma}$ with respect to the standard basis $\pmb{e}_1,\ldots,\pmb{e}_n$. Consequently, $\mathcal{T}_G$ and $\mathcal{T}_{\sigma}$ can be replaced, respectively, with $\mathbf{A}$ and $\mathbf{P}_{\sigma}$ in the statements of the lemma.

We begin with the necessity of the assertion. In accordance with the relation (\ref{Equation00}), $\sigma$ is an automorphism of $G$ if and only if $\mathbf{P}_{\sigma}^t\mathbf{A}\mathbf{P}_{\sigma}=\mathbf{A}$, so for any eigenvector $\pmb{v}$ of $\mathbf{A}$ associated to some eigenvalue $\lambda$, $$\mathbf{P}_{\sigma}^t\mathbf{A}\mathbf{P}_{\sigma}\pmb{v}=\mathbf{A}\pmb{v}=\lambda\pmb{v}.$$
Consequently, $\mathbf{A}\mathbf{P}_{\sigma}\pmb{v}=\lambda\mathbf{P}_{\sigma}\pmb{v}$, which means $\mathbf{P}_{\sigma}\pmb{v}$ is also an eigenvector of $\mathbf{A}$ associated to the eigenvalue $\lambda$, and thus every eigenspace of $\mathbf{A}$ is $\mathbf{P}_{\sigma}$-invariant.

Conversely, let us select an orthonormal basis $\pmb{x}_1,\ldots,\pmb{x}_n$ of $\R^n$, consisting of eigenvectors of $\mathbf{A}$ such that $\mathbf{A}\pmb{x}_i = \lambda_i\pmb{x}_i$, $i=1,\ldots,n$. Since every eigenspace of $\mathbf{A}$ is $\mathbf{P}_{\sigma}$-invariant, for every $\pmb{x}_i$ we have
$$
\mathbf{A}\mathbf{P}_{\sigma}\pmb{x}_i
 = \lambda_i\mathbf{P}_{\sigma}\pmb{x}_i
 = \mathbf{P}_{\sigma} \lambda_i\pmb{x}_i
 =\mathbf{P}_{\sigma}\mathbf{A}\pmb{x}_i.
$$
Consequently, for an arbitrary vector $\pmb{v}=\sum_{i=1}^na_i\pmb{x}_i$ in $\R^n$,
$$
\mathbf{P}_{\sigma}\mathbf{A}\pmb{v}=\mathbf{P}_{\sigma}\mathbf{A}\sum_{i=1}^na_i\pmb{x}_i
=\sum_{i=1}^na_i\mathbf{P}_{\sigma}\mathbf{A}\pmb{x}_i=
\sum_{i=1}^na_i\mathbf{A}\mathbf{P}_{\sigma}\pmb{x}_i
=\mathbf{A}\mathbf{P}_{\sigma}\sum_{i=1}^na_i\pmb{x}_i
=\mathbf{A}\mathbf{P}_{\sigma}\pmb{v}.
$$
As a result, $\mathbf{P}_{\sigma}\mathbf{A}=\mathbf{A}\mathbf{P}_{\sigma}$, and thus the permutation $\sigma$ belongs to $\mathfrak{G}$.
\end{proof}

According to Lemma \ref{Lem-AutomorphismAndOperator}, we can describe automorphisms of $G$ and so the group $\mathfrak{G}$ in terms of eigenspaces of $\mathbf{A}(G)$. Let $U$ be a non-trivial subspace in $\R^n$. Set 
$$
\drm{Aut}{U} = \{ \sigma \in \drm{Sym}{[n]} : \sigma U \subseteq U \}.
$$
Then
\begin{equation}\label{Equ-AutG-EigSpace}
\drm{Aut}{G} = \bigcap_{\lambda \hspace{0.5mm} \in  \hspace{0.5mm} \drm{spec}{\mathbf{A}}} \drm{Aut}{V_{\lambda}}.
\end{equation}
For convenience, we denote the right hand side of the equation above by $\drm{Aut}{\oplus V_{\lambda}}$.

It is plain to see that we cannot determine $\drm{Aut}{\oplus V_{\lambda}}$ by checking permutations in $\drm{Sym}{[n]}$ one by one, for there are $n!$ permutations there. As a matter of fact, it is the distribution of the orthogonal projections of the standard basis ({\it dist}. of OPSB) onto those subspaces that reveals symmetries contained in the subspaces, so we modify the statement in the lemma \ref{Lem-AutomorphismAndOperator} in the following form. Recall that a linear operator $\mathcal{T}$ on $\R^n$ is said to be an {\it isometry} if $\| \mathcal{T} \pmb{v} \| = \| \pmb{v} \|$ for any vector $\pmb{v}$ in $\R^n$. It is easy to check that a permutation on $[n]$ is an isometry on $\R^n$. 

\begin{Lemma}\label{ProjOperatorCommutative} 
Let $\mathcal{T}$ be an isometry on $\mathbb{R}^n $, and let $U$ be a subspace of  $\mathbb{R}^n $. Then the following statements are equivalent. 
\begin{enumerate}
    \item $U$ is $\mathcal{T}$-invariant.
    \item $\mathcal{T}\circ\dproj{U} =\dproj{U}\circ\mathcal{T}$, where $\dproj{U}$ is the orthogonal projection onto the subspace $U$. 
    \item There exists a basis $\pmb{b}_1,\ldots,\pmb{b}_n$ of $\R^n$ so that $\mathcal{T}\circ\dproj{U}(\pmb{b}_i) =\dproj{U}\circ\mathcal{T}(\pmb{b}_i)$, $i=1,\ldots,n$. 
\end{enumerate}
\end{Lemma}

\begin{proof}[\bf Proof]  We first verify that {\it i)}$\Rightarrow${\it ii)}. Let $\pmb{v}$ be a vector of $\R^n$. Then there exist uniquely $\pmb{u}\in U$ and $\pmb{u}'\in U^{\bot}$ so that $\pmb{v}=\pmb{u}+\pmb{u}'$. Consequently,   $\mathcal{T}\circ\dproj{U}(\pmb{v})=\mathcal{T}\circ\dproj{U}(\pmb{u}+\pmb{u}')=\mathcal{T}(\pmb{u})=\dproj{U}(\mathcal{T}(\pmb{u})+\mathcal{T}(\pmb{u}'))=\dproj{U}\circ\mathcal{T}(\pmb{v})$ since $\mathcal{T}$ is an isometry and $U$ is an $\mathcal{T}$-invariant subspace.

Clearly, the 2nd statement can imply the 3rd one. So now we turn to the last part and show that the 3rd statement implies the 1st one.

Let us first recall a fact that 
\begin{equation}\label{AFact} \mathcal{T}U=U \mbox{ if and only if }\mathcal{T}(\pmb{u}) = \dproj{U}\circ\mathcal{T}(\pmb{u}), ~\forall \pmb{u}\in U. \end{equation}  
Since $\pmb{b}_1,\ldots,\pmb{b}_n $ is a basis of $\mathbb{R}^n$, for any vector $\pmb{u}\in U$, $\pmb{u}=\sum_{i=1}^n u_i\pmb{b}_i$ where $u_i\in\R$ and $i=1,\ldots,n$. In accordance with the 3rd statement, we have
\begin{align*} \dproj{U}\circ\mathcal{T}(\pmb{u}) 
&= \sum_{i=1}^n u_i\cdot\dproj{U}\circ\mathcal{T}(\pmb{b}_i) 
= \sum_{i=1}^n u_i\cdot\mathcal{T}\circ\dproj{U}(\pmb{b}_i) \\
&= \mathcal{T}\circ\dproj{U}\left(\sum_{i=1}^n u_i\pmb{b}_i\right)
= \mathcal{T}\left(\pmb{u}\right). \end{align*}
\end{proof}
 
In order to determine whether or not two chosen vertices are symmetric in $G$, we actually need a more general relation below, which can be proved in the way similar to proving Lemma \ref{ProjOperatorCommutative}.   
 
\begin{Lemma}\label{ProjOperatorCommutative-2} Let $\mathcal{T}$ be an isometry on $\mathbb{R}^n $, and let $U$ and $W$ be two subspaces of  $\mathbb{R}^n $. Then the following statements are equivalent. 
\begin{enumerate}
    \item $W = \mathcal{T} U$.
    \item $\mathcal{T}\circ\dproj{U} =\dproj{W}\circ\mathcal{T}$.  
    \item There exists a basis $\pmb{b}_1,\ldots,\pmb{b}_n$ so that $\mathcal{T}\circ\dproj{U}(\pmb{b}_i) =\dproj{W}\circ\mathcal{T}(\pmb{b}_i)$, $i=1,\ldots,n$. 
\end{enumerate}
\end{Lemma}

\vspace{3mm}
To figure out a generating set of $\mathfrak{G}$, there are two important targets: 
\begin{enumerate}

\item[1)] decomposing eigenspaces of $\mathbf{A}$ into IRs of $\mathfrak{G}$;

\item[2)] partitioning the vertex set of $G$ into orbits of $\mathfrak{G}$.

\end{enumerate}

In order to achieve the 2nd target, we need a powerful apparatus - equitable partition (EP). So let us first present some of basic properties of equitable partitions. Suppose $\Pi$ is a partition of $V(G)$ with a group of cells $C_1,\ldots,C_t$, which is said to be {\it equitable} if for any vertex $u$ in $C_i$, the number of neighbors of $u$ in $C_j$ is a constant $b_{ij}$, $(1\leq i,j \leq t)$, which is independent of the vertex $u$. A moment's reflection would show that the partition of $V(G)$ comprised of orbits of a subgroup of $\mathfrak{G}$ is equitable.

It is interesting that one can construct a new graph $G / \Pi$ from $G$ and its equitable partition $\Pi$, which is called the {\it quotient graph} of $G$ over $\Pi$. The vertex set of $G / \Pi$ consists of cells of $\Pi$ and there are $b_{ij}$ arcs from the $i$th to the $j$th vertices of $V(G  / \Pi)$. 

For each cell $C$ of a partition, one can build a vector $\pmb{R}_{C}$ to indicate $C$, that is called the {\it characteristic vector} of the cell $C$ of which the $k$th coordinate is 1 if $k$ belongs to $C$ otherwise it is 0 ($k=1,\ldots,n$). In virtue of characteristic vectors, we can define the {\it characteristic matrix} $\mathbf{R}$ of the partition as $(\pmb{R}_1 \pmb{R}_2 \cdots \pmb{R}_t)$, where $\pmb{R}_i$ is the characteristic vector of the $i$th cell. As we can see below, equitable partitions have a simple feature in linear algebra terms. 

\begin{Lemma}[Godsil and Royle \cite{GodRoy}]\label{Lemma-Book-BasicPropertyofEP}
If $\Pi$ is an equitable partition of $G$ then $\mathbf{A}(G)\cdot\mathbf{R} = \mathbf{R}\cdot\mathbf{A}(G / \Pi)$ and $\mathbf{A}(G / \Pi) = (\mathbf{R}^T \mathbf{R})^{-1} \cdot \mathbf{R}^T \mathbf{A}(G) \mathbf{R}.$
\end{Lemma}

\begin{Lemma}[Godsil and Royle \cite{GodRoy}]\label{Lemma-Book-CharacterizationToEP}
A partition $\Pi$ of $V(G)$ is equitable if and only if the column space of $\mathbf{R}$ is $\mathbf{A}$-invariant.
\end{Lemma}

By means of the lemmas above, it is not difficult to prove the following result, which not only reveals a connection between eigenvalues of $\mathbf{A}(G)$ and that of $\mathbf{A}(G / \Pi)$ but actually provides a way of decomposing eigenspaces of $\mathbf{A}(G)$.
\begin{Theorem}[Godsil and Royle \cite{GodRoy}]\label{Lemma-Book-CharacteristicPolynomialOfEP}
If $\Pi$ is an equitable partition of $G$ then the characteristic polynomial of $\mathbf{A}(G / \Pi)$ divides the characteristic polynomial of $\mathbf{A}(G)$.
\end{Theorem}

Now let us turn to a notable connection between the eigenvectors of $\mathbf{A}(G)$ and that of $\mathbf{A}(G / \Pi)$. Let $\pmb{x}_{\lambda}$ be an eigenvector of $\mathbf{A}(G / \Pi)$ corresponding to the eigenvalue $\lambda$. Then the vector $\mathbf{R} \pmb{x}_{\lambda}$ cannot vanish and in fact
$$
\mathbf{A}(G) \cdot \mathbf{R} \pmb{x}_{\lambda} 
= \mathbf{R}\mathbf{A}(G / \Pi) \cdot \pmb{x}_{\lambda}
= \lambda \mathbf{R} \pmb{x}_{\lambda}.
$$
Hence $\mathbf{R} \pmb{x}_{\lambda}$ is an eigenvector of $\mathbf{A}(G)$. In this situation, we say that the eigenvector $\pmb{x}_{\lambda}$ of $\mathbf{A}(G / \Pi)$ {\it ``lifts''} to an eigenvector of $\mathbf{A}(G)$.

Since the column space $U_\Pi$ of $\mathbf{R}$ is $\mathbf{A}$-invariant due to Lemma \ref{Lemma-Book-CharacterizationToEP}, $U_\Pi$ must have a basis comprised of eigenvectors of $\mathbf{A}(G)$. As a result, each of these eigenvectors is constant on the cells of $\Pi$. In other words, if $x$ and $y$ are two vertices of $G$  belonging to the same cell of $\Pi$ and $V_{\lambda}$ is an eigenspace of $\mathbf{A}(G)$, then 
$$
\langle \pmb{e}_x,\dproj{V_\lambda}(\pmb{R}_j) \rangle = \langle \pmb{e}_y,\dproj{V_\lambda}(\pmb{R}_j) \rangle, ~\forall~\lambda\in\mathrm{spec}~\mathbf{A}(G) \mbox{\it ~and } j\in [t].\footnote{As we shall see in Lemma \ref{Lemma-EquitablePartProj}, this relation  is also sufficient for being equitable.}
$$
Hence, the eigenvectors of $\mathbf{A}$ could be divided into two classes: those that are constant on every cell of $\Pi$ and those that sum to zero on each cell of $\Pi$. Accordingly, we could use eigenspaces of $\mathbf{A}(G/\Pi)$ to split eigenspaces of $\mathbf{A}(G)$, which is one of two major tools we employ to decompose eigenspaces of $\mathbf{A}(G)$. 

It is well-known that given a partition of $V(G)$ there is a unique coarsest equitable partition finer than the original one, and there are a number of efficient algorithms to find the coarsest EP for a given partition. As a matter of fact, the lemma \ref{Lemma-EquitablePartProj} provides an efficient way of figuring out that kind of EPs.

\vspace{3mm}
Now let us illustrate how our algorithm works with  determining whether two vertices belong to one orbit of the automorphism group of the Petersen graph (see Figure 1). We here try to determine if vertices 1 and 7 are in the same orbit, and to figure out an automorphism moving 1 to 7. In order to realize our goal, we need to make use of the geometric information about the permutation representation of $\mathfrak{G}$ and of its subgroups in $\R^{10}$. The key is to work out a group of EPs of $G$, which enable us to split eigenspaces of $\mathbf{A}$ and therefore to obtain IRs of the stabilizers $\mathfrak{G}_1$ and $\mathfrak{G}_7$, where $\mathfrak{G}_v$ is the point stabilizer of $v$ defined as $\{ \xi \in \mathfrak{G} : \xi v = v \}$.  

\begin{figure}[H]
\begin{center}
\includegraphics[angle=2,totalheight=6cm]{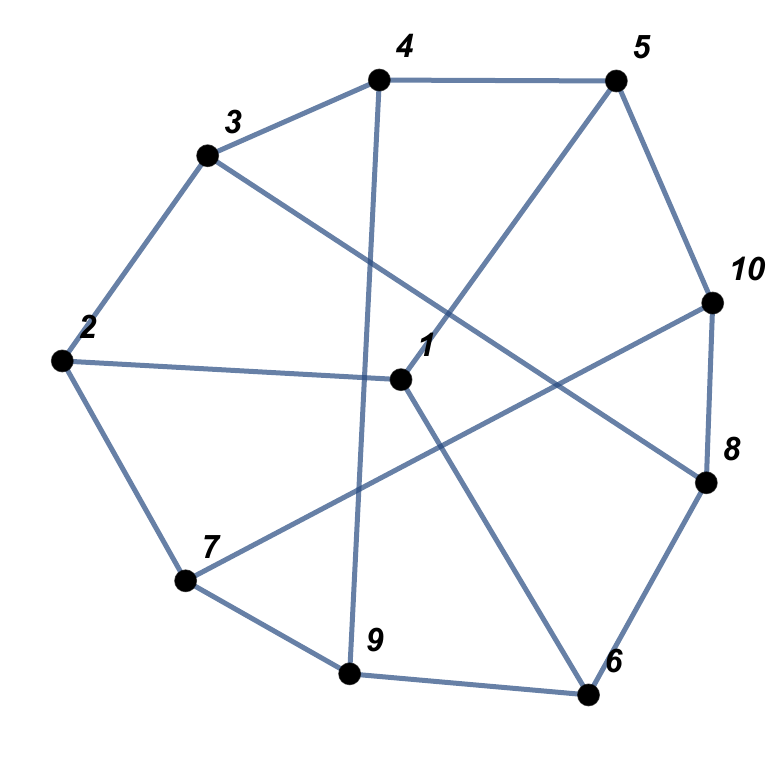} 
\end{center}
\caption[subsection]{The Petersen Graph}
\end{figure}

First of all, let us gather the evidence of being symmetric for 1 and 7 by means of the {\it dist.} of OPSB onto eigenspaces of the adjacency matrix $\mathbf{A}$. As we shall see, it is the {\it dist.} of OPSB onto subspaces relevant that reveals symmetries among vertices. However in the case that the vertex set is of huge order and eigenspaces involved are of really large dimension, the {\it dist.} of OPSB is a real mess, and hence we employ EPs to group vertices and to decompose eigenspaces so we can ultimately clarify symmetries of the graph.

One can easily compute the eigenvalues and corresponding eigenspaces of the graph $G$.\footnote{ In order to decide whether or not two eigenvalues or two vectors are the same, we need high precision arithmetic - eigenvalues and coordinates of eigenvectors may have to be calculated to $n^K$ digits accuracy for some integer $K$. We refer readers to \cite{GolubLoan} for more information. In this paper, we use {\it Mathematica} 10 to compute all data and to draw most of figures.} Actually, $\mathbf{A}$ possesses three eigenvalues 3, -2 and 1 of multiplicity 1, 4 and 5, respectively, and the eigenspace corresponding to 3 shows us nothing about the structure of $\mathfrak{G}$ because it is spanned by the vector $\pmb{1}$. Consequently, we examine the rest of two eigenspaces. 

It is easy to check that all lengths of the OPSB onto the eigenspace $V_{-2}$ and onto $V_1$ respectively are the same, so in order to see angles among the OPSB onto those eigenspaces it suffices to see the OPSB onto $V_{-2}$ and $V_1$. They are displayed by the following matrices, in which the $j$th column is the orthogonal projection of $\pmb{e}_j$ onto the eigenspace relevant.

{ \scriptsize
$$
 \begin{array}{c}
\begin{pmatrix} 
 \begin{array}{rrrrrrrrrr}
  0.400000  & -0.266667  & 0.0666667  & 0.0666667  & -0.266667  & -0.266667  & 
  0.0666667  & 0.0666667  & 0.0666667  & 0.0666667 \\
   -0.266667  & 0.400000  & -0.266667  & 0.0666667  & 0.0666667  & 
  0.0666667  & -0.266667  & 0.0666667  & 0.0666667  & 0.0666667 \\
   0.0666667  & -0.266667  & 0.400000  & -0.266667  & 0.0666667  & 0.0666667  & 
  0.0666667  & -0.266667  & 0.0666667  & 0.0666667 \\
   0.0666667  & 0.0666667  & -0.266667  & 0.400000  & -0.266667  & 0.0666667  & 
  0.0666667  & 0.0666667  & -0.266667  & 0.0666667 \\
   -0.266667  & 0.0666667  & 0.0666667  & -0.266667  & 0.400000  & 0.0666667  & 
  0.0666667  & 0.0666667  & 0.0666667  & -0.266667 \\
   -0.266667  & 0.0666667  & 0.0666667  & 0.0666667  & 0.0666667  & 0.400000  & 
  0.0666667  & -0.266667  & -0.266667  & 0.0666667 \\
   0.0666667  & -0.266667  & 0.0666667  & 0.0666667  & 0.0666667  & 0.0666667  & 
  0.400000  & 0.0666667  & -0.266667  & -0.266667 \\
   0.0666667  & 0.0666667  & -0.266667  & 0.0666667  & 0.0666667  & -0.266667  & 
  0.0666667  & 0.400000  & 0.0666667  & -0.266667 \\
   0.0666667  & 0.0666667  & 0.0666667  & -0.266667  & 
  0.0666667  & -0.266667  & -0.266667  & 0.0666667  & 0.400000  & 0.0666667 \\
   0.0666667  & 0.0666667  & 0.0666667  & 0.0666667  & -0.266667  & 
  0.0666667  & -0.266667  & -0.266667  & 0.0666667  & 0.400000 
  \end{array}
\end{pmatrix} \\
  ~ \\ 
  { \text{ \footnotesize The OPSB matrix corresponding to }\displaystyle  V_{-2} }
  \end{array}
$$ }

{ \scriptsize
$$
\begin{array}{c}
 \begin{pmatrix} 
  \begin{array}{rrrrrrrrrr}
 0.500000 & 0.166667 & -0.166667 & -0.166667 & 0.166667 & 
  0.166667 & -0.166667 & -0.166667 & -0.166667 & -0.166667 \\
   0.166667 & 0.500000 & 0.166667 & -0.166667 & -0.166667 & -0.166667 & 
  0.166667 & -0.166667 & -0.166667 & -0.166667 \\
   -0.166667 & 0.166667 & 0.500000 & 
  0.166667 & -0.166667 & -0.166667 & -0.166667 & 
  0.166667 & -0.166667 & -0.166667 \\
   -0.166667 & -0.166667 & 0.166667 & 0.500000 & 
  0.166667 & -0.166667 & -0.166667 & -0.166667 & 0.166667 & -0.166667 \\
   0.166667 & -0.166667 & -0.166667 & 0.166667 & 
  0.500000 & -0.166667 & -0.166667 & -0.166667 & -0.166667 & 0.166667 \\
   0.166667 & -0.166667 & -0.166667 & -0.166667 & -0.166667 & 
  0.500000 & -0.166667 & 0.166667 & 0.166667 & -0.166667 \\
   -0.166667 & 0.166667 & -0.166667 & -0.166667 & -0.166667 & -0.166667 & 
  0.500000 & -0.166667 & 0.166667 & 0.166667 \\
   -0.166667 & -0.166667 & 0.166667 & -0.166667 & -0.166667 & 
  0.166667 & -0.166667 & 0.500000 & -0.166667 & 0.166667 \\
   -0.166667 & -0.166667 & -0.166667 & 0.166667 & -0.166667 & 0.166667 & 
  0.166667 & -0.166667 & 0.500000 & -0.166667 \\
   -0.166667 & -0.166667 & -0.166667 & -0.166667 & 0.166667 & -0.166667 & 
  0.166667 & 0.166667 & -0.166667 & 0.500000
   \end{array}
  \end{pmatrix}\\
  ~ \\ 
  { \text{ \footnotesize The OPSB matrix corresponding to }\displaystyle  V_{1} }
 \end{array}
$$ }

Given two vertices $u$ and $v$, one useful necessary condition for being symmetric is that there exists a permutation $\sigma$ such that $\sigma \hspace{0.6mm} \dproj{V_\lambda}( \pmb{e}_u ) = \dproj{V_\lambda}( \pmb{e}_v )$, $\forall \lambda\in\mathrm{spec} \hspace{0.6mm} \mathbf{A}$, so in order to determine orbits of $\mathfrak{G}$ we first focus on partitions of $V(G)$ erected by projections relevant to the vertex considered.

Let us begin with a partition of $V$ relevant to the vertex 1. In virtue of the projections of $\pmb{e}_1$ onto some eigenspaces $\oplus_{\lambda \in \Lambda} V_{\lambda}$ where $\Lambda \subseteq \drm{spec}{\mathbf{A}(G)}$, one can see a natural binary relation among vertices: two vertices $x$ and $y$ are said to be {\it projection-related with respect to vertex 1 and $\Lambda$} if for any  
$\lambda \in \Lambda$,
\begin{equation}\label{TheBinaryRelation} 
\| \dproj{V_{\lambda}}( \pmb{e}_x ) \| = \| \dproj{V_{\lambda}}( \pmb{e}_y ) \|
\mbox{ and }
\langle \pmb{e}_x,\dproj{V_{\lambda}}( \pmb{e}_1 )
\rangle = \langle \pmb{e}_y,\dproj{V_{\lambda}}( \pmb{e}_1 )
\rangle. 
\end{equation}
Obviously, the relation above is an equivalence relation, so we have a partition of the vertex set $[10]$. We take here the subset $\Lambda$ to be $\drm{spec}{\mathbf{A}(G)}$, and hence the partition is $\{ \{1\},\{2,5,6\},\{3,4,7,8,9,10\} \}$. It is quite clear that one can define similar relations for other vertices and build similar partitions of $[10]$. For the vertex 7, the partition is $\{ \{7\}, \{2,9,10\}, \{1,3,4,5,6,8\} \}$. Let $\dEP{1}{V_\lambda}$ and $\dEP{7}{V_\lambda}$ be the coarsest EPs, respectively, of those two partitions above, which are actually the same as those two partitions, for they are themselves equitable, but in general $\dEP{v}{ V_\lambda }$ would be finer than the partition erected by the relation (\ref{TheBinaryRelation}).   

In virtue of Lemma \ref{Lemma-EquitablePartProj}, one can readily see that two cells $C_1$ and $C_2$ of an equitable partition can be distinguished by two sets of projections $\{ \dproj{V_{\lambda}}( \pmb{R}_{C_1} ) : \lambda \in \drm{spec}{\dAM{G}} \}$ and $\{ \dproj{V_{\lambda}}( \pmb{R}_{C_2} ) : \lambda \in \drm{spec}{\dAM{G}} \}$. As a result, we can associate two EPs of a graph by means of those projections. Suppose $u$ and $v$ are two vertices of $G$, and $C_u$ and $C_v$ are two cells of $\dEP{u}{ V_\lambda }$ and $\dEP{v}{ V_\lambda }$, respectively, which are said to be {\it in the same type} if for every eigenspace $V_{\lambda}$ of $\mathbf{A}$, 
\begin{equation}\label{Equ-AssociateTwoCells}
\begin{minipage}{12.6cm}
$ \langle \pmb{e}_x,\dproj{V_{\lambda}}( \pmb{R}_{C_u} )
\rangle = \langle \pmb{e}_y,\dproj{V_{\lambda}}( \pmb{R}_{C_v} )
\rangle$ $\forall x \in C_u$ and $y \in C_v,$

and
$\exists \hspace{0.6mm} \sigma\in\mathrm{Sym}\hspace{0.6mm} V ~ s.t., $ $\sigma \hspace{0.6mm} C_u = C_v$ and 
$\sigma\dproj{V_\lambda}( \pmb{R}_{C_u} ) 
= 
\dproj{V_\lambda}( \pmb{R}_{C_v} ),
$
\end{minipage}
\end{equation} 
where $\pmb{R}_{C_u}$ is the characteristic vector of the cell $C_u$. Accordingly, we can define type for equitable partitions $\dEP{u}{ V_\lambda }$ and $\dEP{v}{ V_\lambda }$, which are said to be {\it in the same type}, denoted by $\dEP{u}{ V_\lambda } \asymp \dEP{v}{ V_\lambda }$, if there is a bijection $\psi$ from the set of cells of $\dEP{u}{ V_\lambda }$ to that of $\dEP{v}{ V_\lambda }$ such that $C$ and $\psi( C )$ are in the same type for any cell $C$ of $\dEP{u}{ V_\lambda }$.

In our example, $\dEP{1}{ V_\lambda }$ and $\dEP{7}{ V_\lambda }$ are in the same type and the correspondence is the following:  
$$
\{1\} \mapsto \{7\}, \{2,5,6\} \mapsto \{2,9,10\} \mbox{ and } \{3,4,7,8,9,10\} \mapsto \{1,3,4,5,6,8\}.
$$ 
It is plain to see that if two vertices $u$ and $v$ belong to the same orbit of $\mathfrak{G}$ then two partitions $\dEP{u}{ V_\lambda }$ and $\dEP{v}{ V_\lambda }$ must be in the same type, and the bijection $\psi$ between two EPs is unique.

Two vertices $u$ and $v$ of a graph $G$ are said to be {\it projection-symmetric with respect to $\oplus_{\lambda\in\Lambda} V_\lambda$}, where $\Lambda \subseteq \drm{spec}{\mathbf{A}(G)}$, 
if 
\begin{equation}\label{TheBinaryRelation-2}
\dEP{u}{_{\lambda\in\Lambda} V_\lambda } \asymp \dEP{v}{_{\lambda\in\Lambda} V_\lambda }
\mbox{ and }
G / \dEP{u}{_{\lambda\in\Lambda} V_\lambda } \cong G / \dEP{v}{_{\lambda\in\Lambda} V_\lambda }
\end{equation}
with respect to $\psi$. The isomorphism between two graphs $G / \dEP{u}{_{\lambda\in\Lambda} V_\lambda }$ and $G / \dEP{v}{_{\lambda\in\Lambda} V_\lambda }$ is defined in the way similar to simple graphs, but we here require that the bijection $\phi$ preserves the direction and weights on arcs. Note that with the help of the correspondence $\psi$ one can quickly verify whether or not $G / \dEP{u}{_{\lambda\in\Lambda} V_\lambda } \cong G / \dEP{v}{_{\lambda\in\Lambda} V_\lambda }$.  

Apparently, the relation defined by (\ref{TheBinaryRelation-2}) is an equivalence one, so it induces a partition of $V$. We denote the uniquely coarsest EP of the partition by $\dEP{\mathfrak{G}}{_{\lambda\in\Lambda} V_\lambda }$. In the case that $\Lambda = \drm{spec}{\mathbf{A}(G)}$ we denote for the convenience the equitable partition by $\dEP{\mathfrak{G}}{V_\lambda }$. It is plain to see that any orbit of $\mathfrak{G}$ must be entirely contained in some cell of $\dEP{\mathfrak{G}}{V_\lambda}$, for every eigenspace of $\mathbf{A}$ is $\mathfrak{G}$-invariant. On the other hand, according to the process above, if we have for each eigenspace $V_{\lambda}$ of $\mathbf{A}$ a group of subspaces $X_{\lambda,1},\ldots,X_{\lambda,t}$ such that $\oplus X_{\lambda,k}$ is an orthogonal decomposition of $V_{\lambda}$ ($\lambda \in \drm{spec}{\mathbf{A}}$) and every $X_{\lambda,k}$ ($k=1,\ldots,t$) is $\mathfrak{G}$-invariant, then we could build a new partition $\dEP{\mathfrak{G}}{ X_{\lambda,k} }$, which is refiner than the partition $\dEP{\mathfrak{G}}{V_\lambda}$ but coarser than the partition $\Pi^*$ comprised of orbits of $\mathfrak{G}$. 

In order to determine orbits of $\mathfrak{G}$, we need to split every eigenspace $V_{\lambda}$ further and further. As a matter of fact, if $\oplus W_{\lambda,p}$ is a decomposition of $\R^n$ into IRs of $\mathfrak{G}$, where $W_{\lambda,p} \subseteq V_{\lambda}$ and $\lambda \in \drm{spec}{\mathbf{A}}$, then $\dEP{\mathfrak{G}}{ W_{\lambda,p} }$ is the same as the partition $\Pi^*$. This is because $\Pi^*$ is itself an equitable partition and thus there exists an IR $W_{\lambda,x}$ of $\mathfrak{G}$ of dimension 1 such that $\langle \pmb{e}_u,\pmb{x} \rangle \neq \langle \pmb{e}_v,\pmb{x} \rangle$ on condition that $u$ and $v$ do not belong to the same orbit of $\mathfrak{G}$, where $\pmb{x}$ is a non-trivial vector in $W_{\lambda,x}$. 

In our example, $\dEP{\mathfrak{G}}{V_\lambda}$ possesses only one trivial cell $[10]$, but at this stage we cannot be sure that the eigenspaces of $\mathbf{A}(G)$ are irreducible for $\mathfrak{G}$, so we cannot decide whether $\mathfrak{G}$ is transitive. On the other hand, the matrix $\mathbf{A}(G/ \dEP{\mathfrak{G}}{V_\lambda})$ has only one eigenvalue 3, so we cannot employ the eigenspace of $\mathbf{A}(G/ \dEP{\mathfrak{G}}{V_\lambda})$ to split eigenspaces of $\mathbf{A}(G)$. As a result, we have to use some of the OPSB to split those eigenspaces of $\mathbf{A}(G)$ and thus turn to figuring out IRs of $\mathfrak{G}_v$ for vertices we have chosen instead of trying to decompose eigenspaces of $\mathbf{A}(G)$ into IRs of $\mathfrak{G}$. However we could ultimately obtain the decomposition of $\R^{10}$ into IRs of $\mathfrak{G}$ by means of IRs of $\mathfrak{G}_v$ due to Lemma \ref{Lemma-IRsFeatureGvNonTrivial}. 

Recall that $\dEP{1}{V_\lambda} = \{ \{1\},\{2,5,6\},\{3,4,7,8,9,10\} \}$. It is easy to check that 
$$ 
\mathrm{spec} \hspace{0.6mm} \mathbf{A}( G / \dEP{1}{V_\lambda} ) = \{ 3, -2, 1 \}.
$$ 
In accordance with the relation between eigenvectors of $\mathbf{A}( G / \dEP{1}{V_\lambda} )$ and $\mathbf{A}( G )$, those two projections $\dproj{V_{-2}}( \pmb{e}_1 )$ and $\dproj{V_{1}}( \pmb{e}_1 )$ are lifted from eigenspaces of $\mathbf{A}( G / \dEP{1}{V_\lambda} )$, so we can now split $V_{-2}$ and $V_1$ by eigenspaces of $\mathbf{A}( G / \dEP{1}{V_\lambda} )$ relevant. Note that if $U$ is an $\sigma$-invariant subspace, where $\sigma$ is a permutation, then the orthogonal complement $U^{\perp}$ is also $\sigma$-invariant. Consequently, it is sufficient, to determine if 1 and 7 are symmetric, to analyze one of eigenspaces $V_{-2}$ and $V_1$. In what follows, we focus on the subspace $V_{-2}$.

Set $X_{-2,1,0}[1] =  \mathbf{R}_{ \dEP{1}{V_{\lambda}} } V_{-2}^{ G/\dEP{1}{V_{\lambda}} }$ where $V_{-2}^{ G/\dEP{1}{V_{\lambda}} }$ is the eigenspace of $\mathbf{A}(G/\dEP{1}{V_{\lambda}})$ corresponding to the eigenvalue -2. We now can decompose $V_{-2}$ as $X_{-2,1,0}[1] \oplus X_{-2,1,1}[1]$, where $X_{-2,1,1}[1]$ is the orthogonal complement of $X_{-2,1,0}[1]$ in $V_{-2}$. For convenience, we define $X_{-2,1,1}[1]$ as
$$
V_{-2} \ominus X_{-2,1,0}[1].
$$
Evidently, $\dim X_{-2,1,1}[1] = 3$. Let us now turn to determining orbits of $\mathfrak{G}_1$ by means of the equitable partition $\dep{\mathfrak{G}_1}{ X_{-2,1,0}[1] \oplus X_{-2,1,1}[1] }$.

We first list the OPSB onto $X_{-2,1,1}[1]$ and then figure out $\dep{1;v}{ X_{-2,1,0}[1] \oplus X_{-2,1,1}[1] }$ for every vertex $v \in [10] \setminus \{1\}$ by means of the relation (\ref{TheBinaryRelation}). After that, we can figure out the partition $\dep{\mathfrak{G}_1}{ X_{-2,1,0}[1] \oplus X_{-2,1,1}[1] }$ by means of another relation (\ref{TheBinaryRelation-2}).

{ \scriptsize
$$ 
\begin{array}{c}
\begin{pmatrix} 
 \begin{array}{rrrrrrrrrr}
 0 & 0 & 0 & 0 & 0 & 0 & 0 & 0 & 0 & 0 \\
 0 & 0.222222 & -0.222222 & 0.111111 & -0.111111 & -0.111111 & -0.222222 & 
  0.111111 & 0.111111 & 0.111111 \\
 0 & -0.222222 & 0.388889 & -0.277778 & 0.111111 & 0.111111 & 
  0.0555556 & -0.277778 & 0.0555556 & 0.0555556 \\
 0 & 0.111111 & -0.277778 & 0.388889 & -0.222222 & 0.111111 & 0.0555556 & 
  0.0555556 & -0.277778 & 0.0555556 \\
 0 & -0.111111 & 0.111111 & -0.222222 & 0.222222 & -0.111111 & 0.111111 & 
  0.111111 & 0.111111 & -0.222222 \\
 0 & -0.111111 & 0.111111 & 0.111111 & -0.111111 & 0.222222 & 
  0.111111 & -0.222222 & -0.222222 & 0.111111 \\
 0 & -0.222222 & 0.0555556 & 0.0555556 & 0.111111 & 0.111111 & 0.388889 & 
  0.0555556 & -0.277778 & -0.277778 \\
 0 & 0.111111 & -0.277778 & 0.0555556 & 0.111111 & -0.222222 & 0.0555556 & 
  0.388889 & 0.0555556 & -0.277778 \\
 0 & 0.111111 & 0.0555556 & -0.277778 & 0.111111 & -0.222222 & -0.277778 & 
  0.0555556 & 0.388889 & 0.0555556 \\
 0 & 0.111111 & 0.0555556 & 0.0555556 & -0.222222 & 
  0.111111 & -0.277778 & -0.277778 & 0.0555556 & 0.388889
  \end{array}
\end{pmatrix}\\
  ~ \\ 
  { \text{ \footnotesize The OPSB matrix relevant to 1 in }\displaystyle  X_{-2,1,1}[1] }
 \end{array}
$$ }

It is easy to check that $\dep{\mathfrak{G}_1}{ X_{-2,1,0}[1] \oplus X_{-2,1,1}[1] } = \{ \{1\},\{2,5,6\},\{3,4,7,8,9,10\} \}$. It is different from $\dEP{\mathfrak{G}}{V_{\lambda}}$ above that we now could decompose $X_{-2,1,1}[1]$ further by means of one cell $\{2,5,6\}$ of $\dep{\mathfrak{G}_1}{ X_{-2,1,0}[1] \oplus X_{-2,1,1}[1] }$. To be precise 
$$
X_{-2,1,1}[1] = \drm{span}{ \{ X_{-2,1,1}[1] : \{2,5,6\} \} } \oplus X_{-2,1,3}[1],
$$ 
where 
$$
\drm{span}{ \{ X_{-2,1,1}[1] : \{2,5,6\} \} } = 
\drm{span}{ \{ \dproj{X_{-2,1,1}[1]}( \pmb{e}_x ) :  x \in \{2,5,6\} \} },
$$ 
which is of dimension 2 and denoted by $X_{-2,1,2}[1]$, and $X_{-2,1,3}[1]$ is the orthogonal complement of $X_{-2,1,2}[1]$ in $X_{-2,1,1}[1]$, which is spanned by the first row in the following matrix. Note that the subspace $X_{-2,1,2}[1]$ is invariant for $\mathfrak{G}_1$, since the set $\{2,5,6\}$ is a cell of $\mathrm{EP}[\mathfrak{G}_1]( X_{-2,1,0}[1] \oplus X_{-2,1,1}[1] )$. Consequently, the subspace $X_{-2,1,3}[1]$ is also $\mathfrak{G}_1$-invariant.

{ \scriptsize
$$ 
\begin{array}{c}
\begin{pmatrix} 
 \begin{array}{rrrrrrrrrr}
 0 & 0 & 0.408248 & -0.408248 & 0 & 0 & -0.408248 & -0.408248 & 0.408248 & 0.408248 \\
 0.408248 &	0 &	-0.408248 &	0.408248 &	-0.408248 &	-0.408248 &	0 &	0.408248 &	0 &	0
\end{array}
\end{pmatrix}\\
  ~ \\ 
  { \text{\small vectors spanning } X_{-2,1,3}[1] \mbox{ and } X_{-2,1,3}[7] }
 \end{array}
$$ }

Due to corollaries \ref{Cor-DecomposeInvariantSubspaces} and \ref{Cor-SpanInvariantSubspaces-Isomorphic}, it is often the case that we can split subspaces like $X_{-2,1,1}[1]$ by virtue of a subspace spanned by one of cells of $\dEP{\mathfrak{H}}{X_{\lambda,k}}$, where $\mathfrak{H}$ is a subgroup of $\mathfrak{G}$. As we shall see in the last section, it is one of major tools we use to split subspaces involved and then to refine equitable partitions we have.

\begin{figure}[H]
\begin{center}
\includegraphics[angle=2,totalheight=5cm]{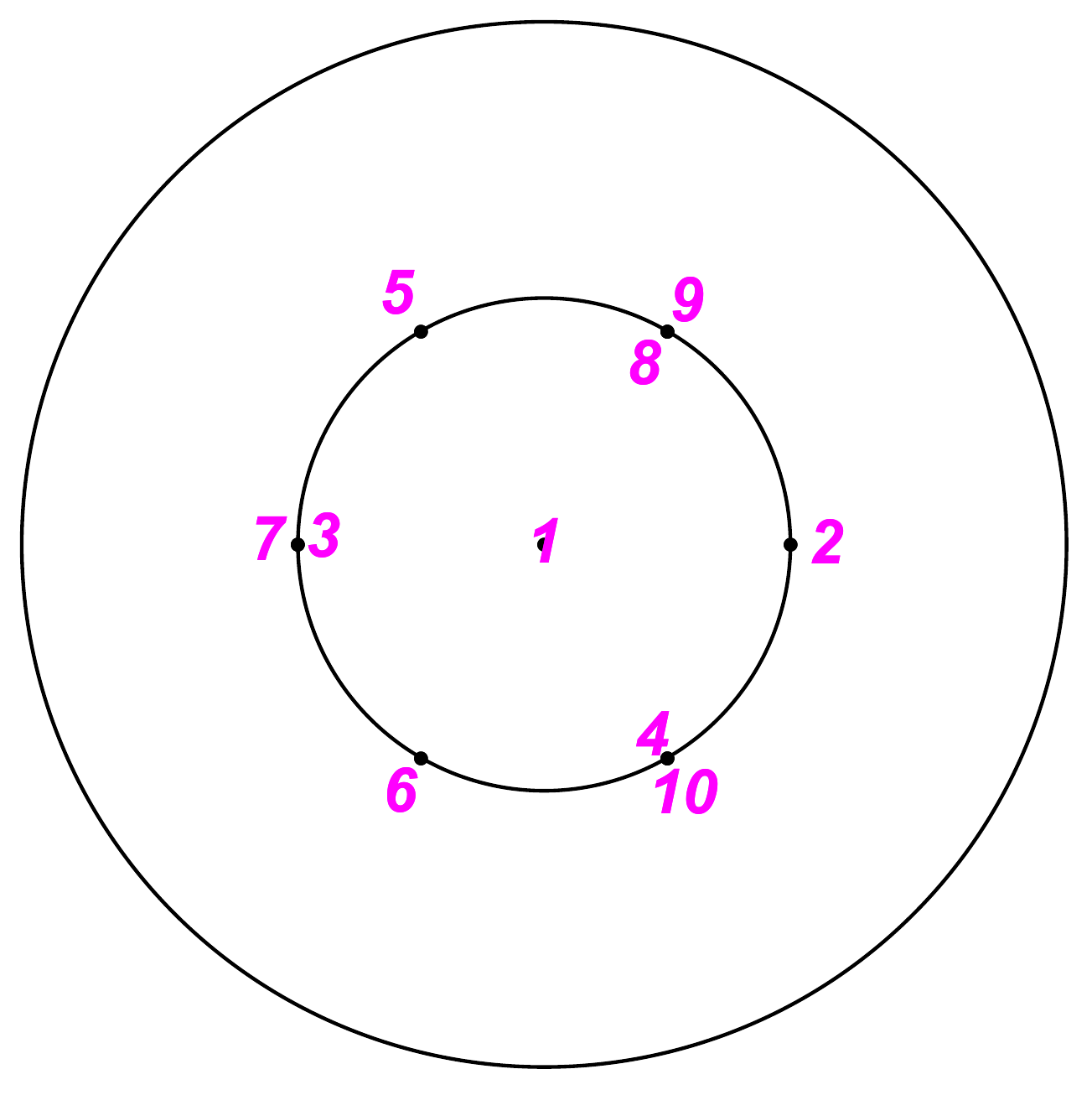} 
~~~~~~~~~ 
\includegraphics[angle=2,totalheight=5cm]{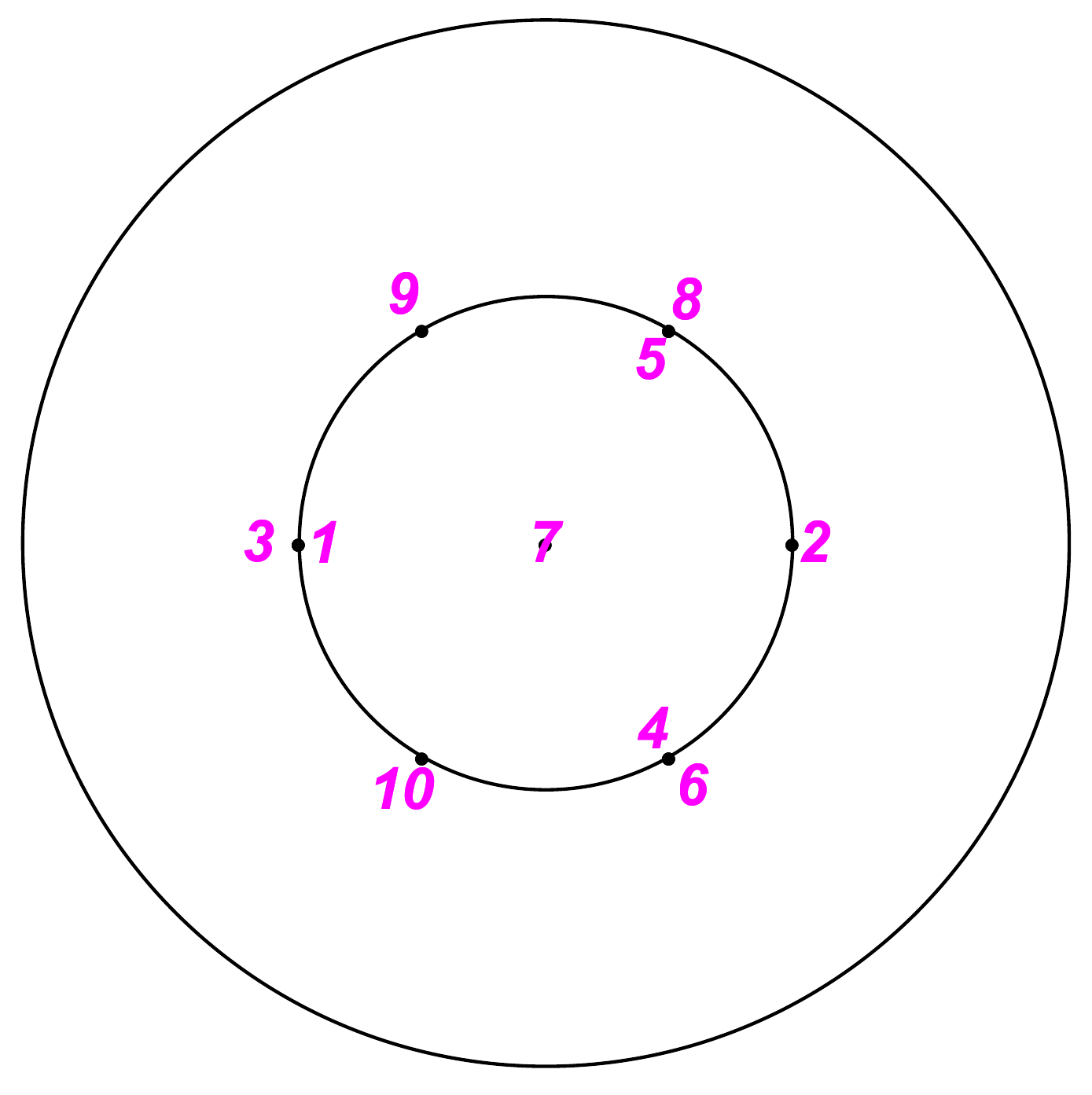}
\end{center}
\caption{The {\it dist.} of OPSB onto $X_{-2,1,2}[1]$ and $X_{-2,1,2}[7]$ }
\end{figure}

In summary, we now decompose $V_{-2}$ into an orthogonally direct sum $\oplus_{ k \in \{0,2,3\} } X_{-2,1,k}[1]$.
By analyzing the {\it dist.} of OPSB onto those three subspaces respectively, one can readily see that the partition $\{ \{1\},\{2,5,6\},\{3,4,7,8,9,10\} \}$ is comprised of orbits of $\mathfrak{G}_1$, and furthermore those three subspaces above are indeed IRs of $\mathfrak{G}_1$ contained in $V_{-2}$. According to Lemma \ref{Lemma-BlockCHARACTERISEDbyProj}, for any block system of $\mathfrak{G}$, there are IRs of $\mathfrak{G}$ displaying the system via the {\it dist.} of OPSB onto those IRs. In our example, those three IRs show us that the action of $\mathfrak{G}_1$ on $\{2,5,6\}$ is primitive but that on $\{3,4,7,8,9,10\}$ is endowed with two block systems $\{ \{3,7\},\{4,10\},\{8,9\} \}$ and $\{ \{3,9,10\},\{4,7,8\} \}$. 

\vspace{3mm}
It is obvious that if 1 and 7 belong to the same orbit of $\mathfrak{G}$ then the stabilizer $\mathfrak{G}_7$ would have three IRs in $V_{-2}$ such that the {\it dist.} of OPSB on them possesses the same structure as that on $X_{-2,1,k}[1]$, $k = 0,2,3$, respectively. Therefore, we now know how to decompose $V_{-2}$ with respect to $\mathfrak{G}_7$. The following matrix displays the OPSB onto the subspace $X_{-2,1,1}[7]$ of $V_{-2}$, which is obtained by $ V_{-2} \ominus X_{-2,1,0}[7]$ where $X_{-2,1,0}[7] = \mathbf{R}_{ \dEP{7}{V_{\lambda}} } V_{-2}^{ G/\dEP{7}{V_{\lambda}} }$, and similarly, $X_{-2,1,1}[7]$ can be decomposed into $X_{-2,1,2}[7] \oplus X_{-2,1,3}[7]$, where the subspace $X_{-2,1,2}[7] = \drm{span}{ \{ X_{-2,1,1}[7] : \{2,9,10\} \} }$, which is spanned by $\{ \dproj{X_{-2,1,1}[7]}( \pmb{e}_x ) :  x \in \{2,9,10\} \}$ and of dimension 2, and $X_{-2,1,3}[7]$ is the orthogonal complement of $X_{-2,1,2}[7]$ in $X_{-2,1,1}[7]$.

{ \scriptsize
$$ 
\begin{array}{c}
\begin{pmatrix} 
 \begin{array}{rrrrrrrrrr} 
  0.388889 & -0.222222 & 0.0555556 & 0.0555556 & -0.277778 & -0.277778 & 0 & 
  0.0555556 & 0.111111 & 0.111111 \\
  -0.222222 & 0.222222 & -0.222222 & 0.111111 & 0.111111 & 0.111111 & 0 & 
  0.111111 & -0.111111 & -0.111111 \\
  0.0555556 & -0.222222 & 0.388889 & -0.277778 & 0.0555556 & 0.0555556 & 
  0 & -0.277778 & 0.111111 & 0.111111 \\
  0.0555556 & 0.111111 & -0.277778 & 0.388889 & -0.277778 & 0.0555556 & 0 & 
  0.0555556 & -0.222222 & 0.111111 \\
  -0.277778 & 0.111111 & 0.0555556 & -0.277778 & 0.388889 & 0.0555556 & 0 & 
  0.0555556 & 0.111111 & -0.222222 \\
  -0.277778 & 0.111111 & 0.0555556 & 0.0555556 & 0.0555556 & 0.388889 & 
  0 & -0.277778 & -0.222222 & 0.111111 \\
  0 & 0 & 0 & 0 & 0 & 0 & 0 & 0 & 0 & 0 \\
  0.0555556 & 0.111111 & -0.277778 & 0.0555556 & 0.0555556 & -0.277778 & 0 & 
  0.388889 & 0.111111 & -0.222222 \\
  0.111111 & -0.111111 & 0.111111 & -0.222222 & 0.111111 & -0.222222 & 0 & 
  0.111111 & 0.222222 & -0.111111 \\
  0.111111 & -0.111111 & 0.111111 & 0.111111 & -0.222222 & 0.111111 & 
  0 & -0.222222 & -0.111111 & 0.222222
  \end{array}
\end{pmatrix}\\
  ~ \\ 
  { \text{ \footnotesize The OPSB matrix relevant to 7 in }\displaystyle  X_{-2,1,1}[7]  }
 \end{array}
$$ }

Accordingly, provided that 1 and 7 are in the same orbit of $\mathfrak{G}$, $\mathfrak{G}_7$ would have three orbits $\{7\}$, $\{2,9,10\}$, $\{1,3,4,5,6,8\}$ and its action on $\{2,9,10\}$ would be primitive and that on $\{1,3,4,5,6,8\}$ would have two block systems $\{ \{1,3\}, \{4,6\},\{5,8\} \}$ and $\{ \{1,4,8\},\{3,5,6\} \}$. In virtue of those relations and  two decompositions of the eigenspace $V_{-2}$ 
$$
\bigoplus_{ k \in \{0,2,3\} } X_{-2,1,k}[1]
\mbox{ and }
\bigoplus_{ k \in \{0,2,3\} } X_{-2,1,k}[7],
$$
one can easily obtain a permutation $\sigma$ on $V_{-2}$ such that 
\begin{align*} 
\sigma \hspace{0.6mm} \dproj{ V_{-2} }( \pmb{e}_1 )   & = \dproj{ V_{-2} }( \pmb{e}_7 ),  \\ 
\sigma \hspace{0.6mm} \dproj{ X_{-2,1,2}[1] }( \pmb{e}_2 ) & = \dproj{ X_{-2,1,2}[7] }( \pmb{e}_2 ), \\
\sigma \hspace{0.6mm} \dproj{ X_{-2,1,2}[1] }( \pmb{e}_5 ) & = \dproj{ X_{-2,1,2}[7] }( \pmb{e}_9 ), \\
\sigma \hspace{0.6mm} \dproj{ X_{-2,1,3}[1] }( \pmb{e}_3 ) & = \dproj{ X_{-2,1,3}[7] }( \pmb{e}_1 ).
\end{align*}
and thus
$$  \sigma = 
\begin{pmatrix} 
 \begin{array}{rrrrrrrrrr}
 1 & 2 & 5 & 6 & 3 & 7 & 10 & 4 & 9 & 8 \\ 
 7 & 2 & 9 & 10 & 1 & 3 & 4 & 6 & 8 & 5
\end{array}
\end{pmatrix}.
$$ 
By means of the relation (\ref{Equation00}), one can quickly check that $\sigma$ is an automorphism of the Petersen graph, so two vertices 1 and 7 are symmetric in the Petersen, and consequently two partitions $\{ \{1\},\{2,5,6\},\{3,4,7,8,9,10\} \}$ and $\{ \{7\}, \{2,9,10\}, \{1,3,4,5,6,8\} \}$ do consist respectively of orbits of $\mathfrak{G}_1$ and $\mathfrak{G}_7$. 

\vspace{3mm}
Now let us see how to characterize the structure of $\mathfrak{G}$ by virtue of what we have had. First of all, let us build a bipartite graph $\dubipart{G}{1}{7}$. The vertex set of $\dubipart{G}{1}{7}$ consists of orbits of those two stabilizers $\mathfrak{G}_{1}$ and $\mathfrak{G}_{7}$, and two vertices in the graph are adjacent if the intersection of two subsets corresponding to the vertices is not empty (see Fig. 3). Evidently, the graph $\dubipart{G}{1}{7}$ is connected, that due to Lemma \ref{LemFindBlocks-1} shows the action of $\mathfrak{G}$ on $V$ is transitive. Moreover, in accordance with Lemma \ref{Lemma-IRsFeatureGvNonTrivial}, two eigenspaces $V_{-2}$ and $V_1$ are indeed irreducible for $\mathfrak{G}$, for the subspaces $V_{-2}[1]$ and $V_{1}[1]$ are both of dimension 1, where $V_{\lambda}[1]$ is defined as $\{ \pmb{x} \in V_{\lambda} : \xi \hspace{0.5mm} \pmb{x} = \pmb{x}, \forall\xi\in\mathfrak{G}_1 \}$ and actually equal to $\mathbf{R}_{ \Pi_1^* } V_{\lambda}^{ G/ \Pi_1^* }$, where $\Pi_1^*$ is the partition consisting of orbits of $\mathfrak{G}_1$ and $\lambda = -2,1$. As a result, the action of $\mathfrak{G}$ on $[10]$ is primitive due to Lemma \ref{Lemma-BlockCHARACTERISEDbyProj}. 

\begin{figure}[H]
\begin{center}
\includegraphics[totalheight=5.2cm]{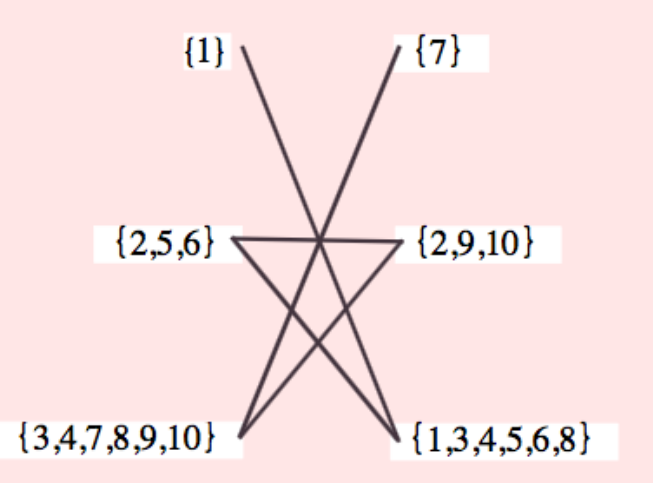}
\end{center}
\caption{The bipartite graph $\dubipart{G}{1}{7}$ }
\end{figure}

As having seen, we show two vertices 1 and 7 are symmetric in the Petersen graph by comparing the {\it dist.} of OPSB on IRs of $\mathfrak{G}_1$ and that of $\mathfrak{G}_7$ relevant. There is another way of displaying the symmetry via the OPSB matrix, that is somewhat better than the first one, for it enables us to see symmetries among vertices in subspaces of dimension larger than 3. For instance, we reorganize the OPSB matrix onto $V_{-2}$ in accordance with the order 1,2,5,6,3,7,10,4,9,8, that means we rearrange the order of entries in each column of the OPSB matrix according to the order (see the matrix below). Similarly, we can reorganize the OPSB matrix onto $V_{-2}$ in accordance with the order 7,2,9,10,1,3,4,6,8,5. It is easy to check that those two reorganized matrices are the same. As the matter of fact, it is the relation that shows 1 and 7 are indeed symmetric in the Petersen graph.

{ \scriptsize
$$ 
\begin{array}{c}
\begin{pmatrix} 
 \begin{array}{rrrrrrrrrr}
  0.400000 & -0.266667 & -0.266667 & -0.266667 & 0.0666667 & 0.0666667 & 
  0.0666667 & 0.0666667 & 0.0666667 & 0.0666667 \\
 -0.266667 & 0.400000 & 0.0666667 & 0.0666667 & -0.266667 & -0.266667 & 
  0.0666667 & 0.0666667 & 0.0666667 & 0.0666667 \\
 -0.266667 & 0.0666667 & 0.400000 & 0.0666667 & 0.0666667 & 
  0.0666667 & -0.266667 & -0.266667 & 0.0666667 & 0.0666667 \\
 -0.266667 & 0.0666667 & 0.0666667 & 0.400000 & 0.0666667 & 0.0666667 & 
  0.0666667 & 0.0666667 & -0.266667 & -0.266667 \\
 0.0666667 & -0.266667 & 0.0666667 & 0.0666667 & 0.400000 & 0.0666667 & 
  0.0666667 & -0.266667 & 0.0666667 & -0.266667 \\
 0.0666667 & -0.266667 & 0.0666667 & 0.0666667 & 0.0666667 & 
  0.400000 & -0.266667 & 0.0666667 & -0.266667 & 0.0666667 \\
 0.0666667 & 0.0666667 & -0.266667 & 0.0666667 & 0.0666667 & -0.266667 & 
  0.400000 & 0.0666667 & 0.0666667 & -0.266667 \\
 0.0666667 & 0.0666667 & -0.266667 & 0.0666667 & -0.266667 & 0.0666667 & 
  0.0666667 & 0.400000 & -0.266667 & 0.0666667 \\
 0.0666667 & 0.0666667 & 0.0666667 & -0.266667 & 0.0666667 & -0.266667 & 
  0.0666667 & -0.266667 & 0.400000 & 0.0666667 \\
 0.0666667 & 0.0666667 & 0.0666667 & -0.266667 & -0.266667 & 
  0.0666667 & -0.266667 & 0.0666667 & 0.0666667 & 0.400000
\end{array}
\end{pmatrix}\\
  ~ \\ 
  { \text{ \footnotesize Reorganizing the OPSB matrix onto } V_{-2} \text{\it ~according to two orders: }\displaystyle 1,2,5,6,3,7,10,4,9,8 \text{ and } 7,2,9,10,1,3,4,6,8,5  }
 \end{array}
$$ }

{ \scriptsize
$$ 
\begin{array}{c}
\begin{pmatrix} 
 \begin{array}{rrrrrrrrrr}
 0.500000 & 0.166667 & 0.166667 & 
  0.166667 & -0.166667 & -0.166667 & -0.166667 & -0.166667 & -0.166667 & \
-0.166667\\
 0.166667 & 0.500000 & -0.166667 & -0.166667 & 0.166667 & 
  0.166667 & -0.166667 & -0.166667 & -0.166667 & -0.166667\\
 0.166667 & -0.166667 & 0.500000 & -0.166667 & -0.166667 & -0.166667 & 
  0.166667 & 0.166667 & -0.166667 & -0.166667\\
 0.166667 & -0.166667 & -0.166667 & 
  0.500000 & -0.166667 & -0.166667 & -0.166667 & -0.166667 & 0.166667 & 
  0.166667\\
 -0.166667 & 0.166667 & -0.166667 & -0.166667 & 
  0.500000 & -0.166667 & -0.166667 & 0.166667 & -0.166667 & 0.166667\\
 -0.166667 & 0.166667 & -0.166667 & -0.166667 & -0.166667 & 0.500000 & 
  0.166667 & -0.166667 & 0.166667 & -0.166667\\
 -0.166667 & -0.166667 & 0.166667 & -0.166667 & -0.166667 & 0.166667 & 
  0.500000 & -0.166667 & -0.166667 & 0.166667\\
 -0.166667 & -0.166667 & 0.166667 & -0.166667 & 
  0.166667 & -0.166667 & -0.166667 & 0.500000 & 0.166667 & -0.166667\\
 -0.166667 & -0.166667 & -0.166667 & 0.166667 & -0.166667 & 
  0.166667 & -0.166667 & 0.166667 & 0.500000 & -0.166667\\
 -0.166667 & -0.166667 & -0.166667 & 0.166667 & 0.166667 & -0.166667 & 
  0.166667 & -0.166667 & -0.166667 & 0.500000
\end{array}
\end{pmatrix}\\
  ~ \\ 
   { \text{ \footnotesize Reorganizing the OPSB matrix onto } V_{1} \text{\it ~according to two orders: }\displaystyle 1,2,5,6,3,7,10,4,9,8 \text{ and } 7,2,9,10,1,3,4,6,8,5  }
 \end{array}
$$ }

\subsection{A block family of $\mathfrak{G}$}

As we have mentioned, a non-empty subset $B$ of some orbit of $\mathfrak{G}$ is a block for $\mathfrak{G}$ if for any permutation $\sigma \in \mathfrak{G}$, either $\sigma B = B$ or $\sigma B \cap B =\emptyset$. That feature makes $B$ impressive, for the action of $\mathfrak{G}$ on $B$ is like that on one vertex, although $B$ may contain a large number of vertices. As a result, the action of $\mathfrak{G}$ on a block system is somewhat like that on the whole set $V$.

Let us first present a fundamental characterization of blocks, which suggests that the stabilizer of a block is to some extent maximal in $\mathfrak{G}$. It is an essential feature of blocks and we shall give two other characterizations of the feature in the next section.

\begin{Lemma}\label{LemBlock}
Let $\mathfrak{H}$ be a subgroup of a permutation group $\mathfrak{G}$ that acts on a finite set $V$. If $T$ is an orbit of $\mathfrak{H}$ such that $\dfix{G}{t}\leq\mathfrak{H}$ for some $t\in T$, then $T$ is a block for $\mathfrak{G}$ and $\mathfrak{H}$ is the stabilizer of $T$, 
{\it i.e.,} $\mathfrak{H}=\dfix{G}{T}$. Conversely, if $B$ is a block for $\mathfrak{G}$ then $\dfix{G}{b}$ must be a subgroup of $\dfix{G}{B}$ for some $b\in B$. \end{Lemma}

The following is a classical result characterizing the relation between blocks and their stabilizers, that explains the reason why blocks are vitally important in finding out a generating set of $\mathfrak{G}$.

\begin{Lemma}[Dixon and Mortimer \cite{DM}] \label{LemBlock&Stabilizers} 
Let $\mathfrak{G}$ be a permutation group acting on a finite set $V$ transitively, let $\mathcal{B}$ be the set of all blocks $B$ for $\mathfrak{G}$ with $b\in B$, where $b\in V$, and let $\mathcal{S}$ be the set of all subgroups $\mathfrak{H}$ of $\mathfrak{G}$ with $\dfix{G}{b} \leq \mathfrak{H}$. Then there is a bijection $\Psi$ of $\mathcal{B}$ onto $\mathcal{S}$ defined by $\Psi(B) = \dfix{G}{B}$, and furthermore the mapping $\Psi$ is order-preserving in the sense that if $B_1$ and $B_2$ are two blocks in $\mathcal{B}$ then $B_1 \subseteq B_2$ if and only if $\Psi(B_1) \leq \Psi(B_2)$. 
\end{Lemma}

In virtue of the relation above, one can readily prove the following.

\begin{Lemma}\label{KeyThmMaximalSubgroup}  
Let $\mathfrak{G}$ be a permutation group acting transitively on a set $V$, and let $B$ be a block for $\mathfrak{G}$. Then $B$ is a maximal block if and only if $\dfix{G}{B}$ is a maximal subgroup of $\mathfrak{G}$.   
\end{Lemma}
\begin{proof}[\bf Proof] 
Since $B$ is a block for $\mathfrak{G}$, $\dfix{G}{b}$ is a subgroup of $\dfix{G}{B}$ due to Lemma \ref{LemBlock}, and thus $\dfix{G}{B}$ is one member of the set $\mathcal{S}$ consisting of all subgroups $\mathfrak{H}$ of $\mathfrak{G}$ with $\dfix{G}{b} \leq \mathfrak{H}$. By means of Lemma \ref{LemBlock&Stabilizers}, $B$ must be a maximal block for $\mathfrak{G}$.

Now we turn to the necessity of the assertion. Suppose $\dfix{G}{B}$ is not a maximal subgroup of $\mathfrak{G}$ and  there exists a subgroup $\mathfrak{H}$ of $\mathfrak{G}$ such that $\dfix{G}{B} \lneq \mathfrak{H} \lneq \mathfrak{G}$. Then $\mathfrak{H}$ contains $\dfix{G}{b}$ as a subgroup and thus it is a member of $\mathcal{S}$. Since $B$ is a maximal and  $\Psi$ is an order-preserving map, $\Psi^{-1}(\mathfrak{H})$ should be a subset of $B$, and thereby $\mathfrak{H}$ would be a subgroup of $\dfix{G}{B}$, which is in contradiction with the assumption that $\dfix{G}{B} \lneq \mathfrak{H} \lneq \mathfrak{G}$.
\end{proof}

This lemma leads immediately to the following characterization of primitive groups. 

\begin{Lemma}[Dixon and Mortimer \cite{DM}]\label{LemPrimitiveness1stNeceSuffCond}  
Let $\mathfrak{G}$ be a permutation group acting transitively on a set $V$ with at least two elements. Then $\mathfrak{G}$ is primitive if and only if each stabilizer $\mathfrak{G}_v$ is a maximal subgroup of $\mathfrak{G}$, where $v$ is an element of $V$.   
\end{Lemma}

First of all, note that in virtue of the same argument for establishing those three results above one could prove a more general version for each one of those results without the restriction that the action of $\mathfrak{G}$ is transitive. In the case of being not transitive, we focus on one orbit, say, $T$, of $G$ and naturally require that $B$ is subset of $T$.

According to Lemma \ref{LemPrimitiveness1stNeceSuffCond}, if $\mathfrak{G}$ is primitive then we only need, provided that we have the subgroup $\mathfrak{G}_v$, to figure out one permutation $\gamma$ in $\mathfrak{G}$ moving the vertex $v$ to anther vertex, for $\mathfrak{G} = \langle \mathfrak{G}_v,\gamma \rangle$ in this case. However, in order to use the lemma above, we need to determine whether $\mathfrak{G}$ is primitive, so in the next section we establish another way to characterize the primitiveness (see Theorem \ref{ThmNeceAndSuffForPrimitiveness} for details).

In the case of being imprimitive, we need a group of blocks $B_1 \subsetneq B_2 \subsetneq \cdots \subsetneq B_l$ such that $B_1$ is a minimal block for $\mathfrak{G}$, $B_i$ is maximal in $B_{i+1}$ ($i=1,\ldots,l-1$) and $B_l$ is a maximal block. Such a group is called a {\it block family} of $\mathfrak{G}$. According to Lemma \ref{LemBlock&Stabilizers}, to generate the subgroup $\mathfrak{G}_{B_{i+1}}$, we only need, provided that we have the stabilizer $\mathfrak{G}_{B_i}$, to figure out one permutation $\gamma_{i}$ in $\mathfrak{G}$ moving one vertex $b_i$ in $B_i$ to a vertex in $B_{i+1} \setminus B_i$, for $B_i$ is maximal in $B_{i+1}$, so $\mathfrak{G}_{B_{i+1}} = \langle \mathfrak{G}_{B_i},\gamma_i \rangle$. As a result, if we have the point stabilizer $\mathfrak{G}_{b_1}$, where $b_1 \in B_1$, then the subset $\{ \mathfrak{G}_{b_1}, \gamma_1, \gamma_2,\ldots, \gamma_l \}$ is a generating set of $\mathfrak{G}$, where $\gamma_l$ is a permutation in $\mathfrak{G}$ moving $B_l$ to another block.

In order to find out a generating set of $\mathfrak{G}_{b_1}$, we do the same thing on an orbit of $\mathfrak{G}_{b_1}$. It is clear that we can ultimately have a generating set of $\mathfrak{G}$.

\begin{figure}
\begin{center}\setlength{\unitlength}{0.5bp}
 \includegraphics[width=6.5cm]{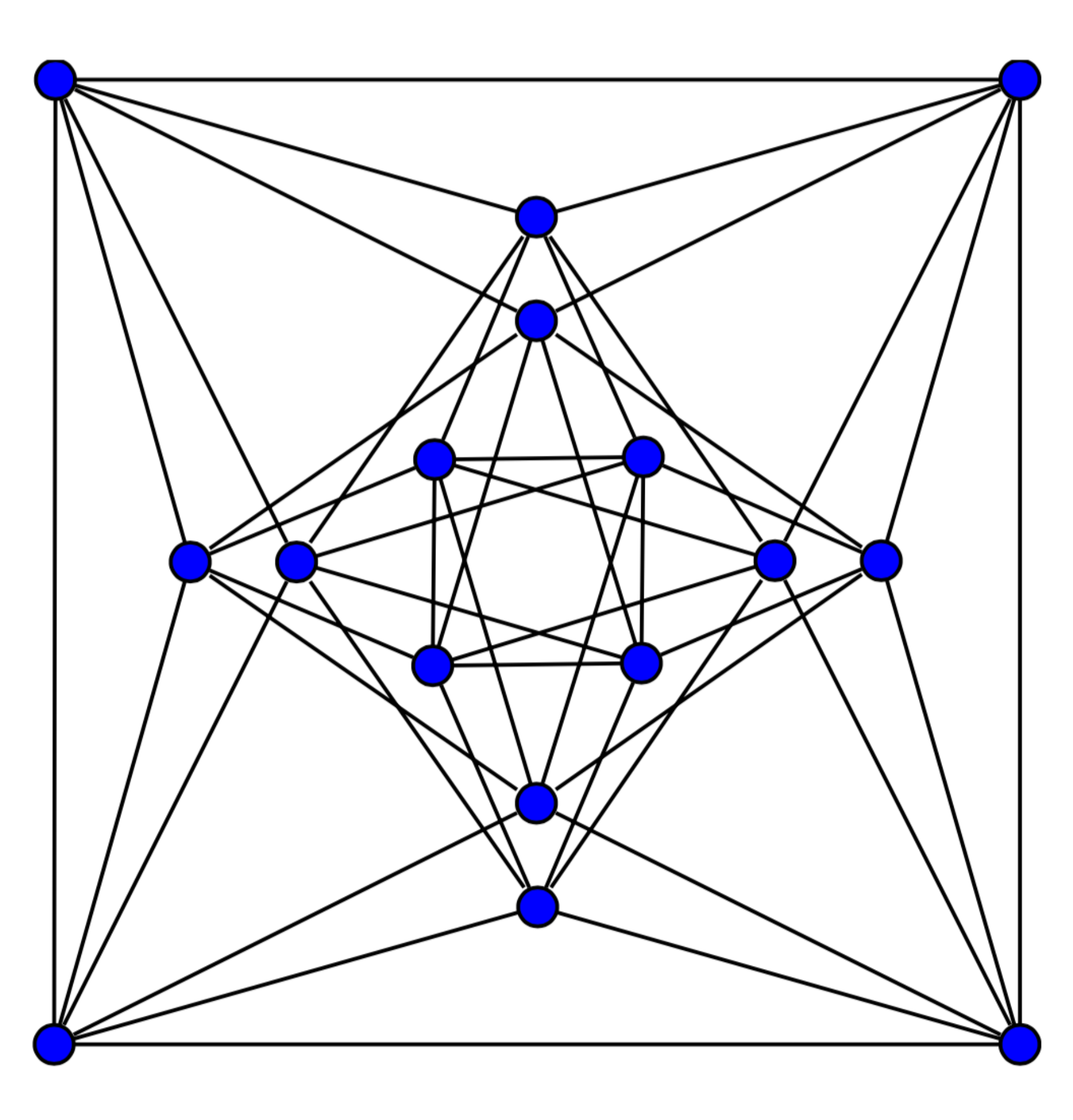}
 {\color[rgb]{1,0,0} 
 \put(-375,360){\fontsize{11}{2.73}\selectfont 1}
 \put(-373,7){\fontsize{11}{2.73}\selectfont 2}
 \put(-17,7){\fontsize{11}{2.73}\selectfont 3}
 \put(-17,360){\fontsize{11}{2.73}\selectfont 4}
 \put(-195,320){\fontsize{11}{2.73}\selectfont 5}
 \put(-330,182){\fontsize{11}{2.73}\selectfont 6}
 \put(-195,47){\fontsize{11}{2.73}\selectfont 7}
 \put(-63,182){\fontsize{11}{2.73}\selectfont 8}
 \put(-195,248){\fontsize{11}{2.73}\selectfont 9}
 \put(-265,182){\fontsize{11}{2.73}\selectfont 10}
 \put(-200,117){\fontsize{11}{2.73}\selectfont 11}
 \put(-140,182){\fontsize{11}{2.73}\selectfont 12}
 \put(-230,208){\fontsize{11}{2.73}\selectfont 13}
 \put(-230,160){\fontsize{11}{2.73}\selectfont 14}
 \put(-184,160){\fontsize{11}{2.73}\selectfont 15}
 \put(-180,208){\fontsize{11}{2.73}\selectfont 16}
 }
\end{center}
\caption{ The Shrikhande graph with four blocks $\{1,3,13,15\},\{2,4,14,16\},\{5,6,7,8\},\{9,10,11,12\}$ }\label{Fig-3}
\end{figure}
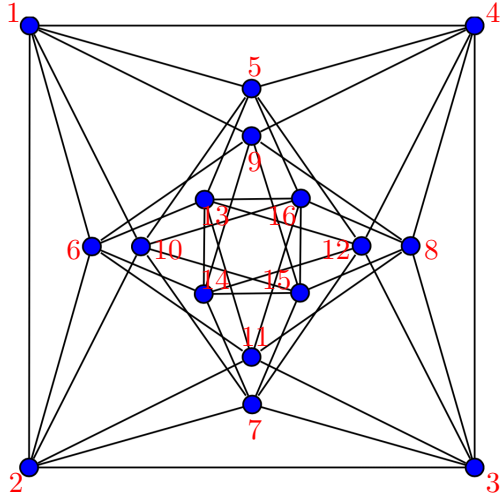 

\vspace{3mm}
Now let us show by an example how to find out a block family of $\mathfrak{G}$. The graph considered here is the Shrikhande graph shown in Fig. \ref{Fig-3}, and its adjacency matrix has 3 distinct eigenvalues 6, 2 and $-2$ of multiplicity 1, 6 and 9, respectively. According to Lemma \ref{Lem-AutomorphismAndOperator}, the automorphism group $\mathfrak{G}$ is determined by the eigenspace $V_2$, corresponding to the eigenvalue 2, {\it i.e.,} $\drm{Aut}{G} = \drm{Aut}{V_2}$. One of key ingredients we need is a partition $\Pi_v^*$ of $V$ consisting of orbits of a point stabilizer $\mathfrak{G}_v$ and we here try to determine $\Pi_1^*$. The following is the OPSB onto the eigenspace.

$$ 
\begin{array}{c}
{ \scriptsize
\left( 
 \begin{array}{rrrrrrrrrrrrrrrr}
 0.375 & 0.125 & -0.125 & 0.125 & 0.125 & 0.125 & -0.125 & -0.125 & 0.125 & 
0.125 & -0.125 & -0.125 & -0.125 & -0.125 & -0.125 & -0.125 \\ 0.125 & 
0.375 & 0.125 & -0.125 & -0.125 & 0.125 & 0.125 & -0.125 & -0.125 & 0.125 & 
0.125 & -0.125 & -0.125 & -0.125 & -0.125 & -0.125 \\ -0.125 & 0.125 & 
0.375 & 0.125 & -0.125 & -0.125 & 0.125 & 0.125 & -0.125 & -0.125 & 0.125 & 
0.125 & -0.125 & -0.125 & -0.125 & -0.125 \\ 0.125 & -0.125 & 0.125 & 0.375 & 
0.125 & -0.125 & -0.125 & 0.125 & 0.125 & -0.125 & -0.125 & 0.125 & -0.125 & 
-0.125 & -0.125 & -0.125 \\ 0.125 & -0.125 & -0.125 & 0.125 & 0.375 & 
-0.125 & -0.125 & -0.125 & -0.125 & 0.125 & -0.125 & 0.125 & 0.125 & -0.125 & 
-0.125 & 0.125 \\ 0.125 & 0.125 & -0.125 & -0.125 & -0.125 & 0.375 & -0.125 & 
-0.125 & 0.125 & -0.125 & 0.125 & -0.125 & 0.125 & 0.125 & -0.125 & -0.125 \\ 
-0.125 & 0.125 & 0.125 & -0.125 & -0.125 & -0.125 & 0.375 & -0.125 & -0.125 & 
0.125 & -0.125 & 0.125 & -0.125 & 0.125 & 0.125 & -0.125 \\ -0.125 & -0.125 & 
0.125 & 0.125 & -0.125 & -0.125 & -0.125 & 0.375 & 0.125 & -0.125 & 0.125 & 
-0.125 & -0.125 & -0.125 & 0.125 & 0.125 \\ 0.125 & -0.125 & -0.125 & 0.125 & 
-0.125 & 0.125 & -0.125 & 0.125 & 0.375 & -0.125 & -0.125 & -0.125 & -0.125 & 
0.125 & 0.125 & -0.125 \\ 0.125 & 0.125 & -0.125 & -0.125 & 0.125 & -0.125 & 
0.125 & -0.125 & -0.125 & 0.375 & -0.125 & -0.125 & -0.125 & -0.125 & 0.125 & 
0.125 \\ -0.125 & 0.125 & 0.125 & -0.125 & -0.125 & 0.125 & -0.125 & 0.125 & 
-0.125 & -0.125 & 0.375 & -0.125 & 0.125 & -0.125 & -0.125 & 0.125 \\ 
-0.125 & -0.125 & 0.125 & 0.125 & 0.125 & -0.125 & 0.125 & -0.125 & -0.125 & 
-0.125 & -0.125 & 0.375 & 0.125 & 0.125 & -0.125 & -0.125 \\ -0.125 & 
-0.125 & -0.125 & -0.125 & 0.125 & 0.125 & -0.125 & -0.125 & -0.125 & -0.125 & 
0.125 & 0.125 & 0.375 & 0.125 & -0.125 & 0.125 \\ -0.125 & -0.125 & -0.125 & 
-0.125 & -0.125 & 0.125 & 0.125 & -0.125 & 0.125 & -0.125 & -0.125 & 0.125 & 
0.125 & 0.375 & 0.125 & -0.125 \\ -0.125 & -0.125 & -0.125 & -0.125 & 
-0.125 & -0.125 & 0.125 & 0.125 & 0.125 & 0.125 & -0.125 & -0.125 & -0.125 & 
0.125 & 0.375 & 0.125 \\ -0.125 & -0.125 & -0.125 & -0.125 & 0.125 & -0.125 & 
-0.125 & 0.125 & -0.125 & 0.125 & 0.125 & -0.125 & 0.125 & -0.125 & 0.125 & 
0.375
 \end{array}
\right) }\\
  ~ \\ 
  { \text{ \footnotesize The OPSB onto the eigenspace }\textstyle  V_2 }
 \end{array}
$$ 

Apparently, it is hard to hold our desire only by examining the {\it dist}. of OPSB in the eigenspace $V_2$, so let us try to split the subspace into smaller ones. In virtue of the OPSB onto $V_2$, we can work out a family of equitable partitions $\{ \dep{v}{V_2} : v \in V \}$ by the relation (\ref{TheBinaryRelation}) and then obtain the partition $\dep{\mathfrak{G}}{V_2}$ by anther relation (\ref{TheBinaryRelation-2}). It is easy to check that $\dep{\mathfrak{G}}{V_2}$ possesses only one cell $[16]$, so the adjacency matrix of the quotient graph $G / \dep{\mathfrak{G}}{V_2}$ has only one eigenvalue 6, and therefore we cannot use the eigenspace of $\mathbf{A}( G / \dep{\mathfrak{G}}{V_2} )$ to split the eigenspace $V_2$. Consequently, let us turn to the partition $\dep{1}{V_2} = \{ \{1\},\{2,4,5,6,9,10\},\{3,7,8,11,12,13,14,15,16\} \}$.

One can easily check that $\dAM{ G / \dep{1}{V_2} }$ possesses 3 eigenvalues 6, 2 and $-2$, every one of which is of multiplicity 1, and the vector $\dproj{V_2}( \pmb{e}_1 )$ belongs to $\dChMat{\dep{1}{V_2}} V_2^{G / \dep{1}{V_2}}$, {\it i.e.,} it is lifted from the eigenspace $V_2^{G / \dep{1}{V_2}}$. Hence we can use the subspace $X_{2,1,0}[1] := \dChMat{\dep{1}{V_2}} V_2^{G / \dep{1}{V_2}}$ to split $V_2$. Set 
$$
X_{2,1,1}[1] = V_2 \ominus X_{2,1,0}[1].
$$
Clearly, $\dim X_{2,1,1}[1] = 5$.

At this stage, it is important to verify whether or not some of cells of $\dep{1}{V_2}$ can split the subspace $X_{2,1,1}[1]$ further.  
Note that
$$
\dim \drm{span}{ \{ X_{2,1,1}[1] : \{2,4,5,6,9,10\} \} } = 4
$$ 
while 
$$
\dim \drm{span}{ \{ X_{2,1,1}[1] : \{3,7,8,11,12,13,14,15,16\} \} } = 5,
$$ 
so we could use the first subspace to split $X_{2,1,1}[1]$, where $\drm{span}{ \{ X_{2,1,1}[1] : \{2,4,5,6,9,10\} \} }$ is the subspace spanned by vectors $\{ \dproj{ X_{2,1,1}[1] }( \pmb{e}_x ) : x \in \{2,4,5,6,9,10\} \}$ and denoted by $X_{2,1,2}[1]$. Set $X_{2,1,3}[1] = X_{2,1,1}[1] \ominus X_{2,1,2}[1]$, which is of dimension 1 and spanned by the following vector 
$$
(0,0,0,0,0,0,-0.408248,0.408248,0,0, 0.408248, -0.408248,0,-0.408248,0, 0.408248).
$$
Since the subspace $X_{2,1,3}[1]$ is $\mathfrak{G}_1$-invariant and the projections of $\pmb{e}_3$, $\pmb{e}_{13}$ and $\pmb{e}_{15}$ onto $X_{2,1,3}[1]$ are equal to $\pmb{0}$, the cell $\{3,7,8,11,12,13,14,15,16\}$ cannot be one orbit of $\mathfrak{G}_1$. Consequently, we now can refine the partition $\dep{1}{V_2}$ by means of the decomposition $\oplus_{k\in\{0,2,3\}} X_{2,1,k}[1]$ of $V_2$.

Once again, we first work out a family of partitions $\{ \dEP{1;v}{X_{2,1,k}[1]} : v \in V \setminus\{1\} \}$ by the relation (\ref{TheBinaryRelation}) and then obtain the partition $\dEP{\mathfrak{G}_1}{X_{2,1,k}[1]}$ by anther relation (\ref{TheBinaryRelation-2}). One can readily check that  
$$
\dEP{\mathfrak{G}_1}{X_{2,1,k}[1]} = \{ \{1\},\{2,4,5,6,9,10\},\{3,13,15\},\{7,8,11,12,14,16\} \}.
$$

It is easy to compute the spectrum of $\dAM{G / \dEP{\mathfrak{G}_1}{X_{2,1,k}[1]}}$. In fact, the matrix possesses 3 distinct eigenvalues 6, 2 and $-2$ of multiplicity 1, 1 and 2 respectively. Denote the subspace $\dChMat{\dEP{\mathfrak{G}_1}{X_{2,1,k}[1]}} V_2^{G / \dEP{\mathfrak{G}_1}{X_{2,1,k}[1]}}$ by $X_{2,2,0}[1]$. Note that $\dim X_{2,2,0}[1] = 1$, so we cannot use it to split the subspace $X_{2,1,2}[1]$, which is of dimension 4, for $X_{2,2,0}[1] = X_{2,1,0}[1]$. We have however another apparatus to split the subspace at this stage - those subspaces spanned by cells of $\dEP{\mathfrak{G}_1}{X_{2,1,k}[1]}$. As a matter of fact, 
$$
\dim \drm{span}{ \{ X_{2,1,2}[1] : \{3,13,15\} \} } = 2.
$$
As a result we now can decompose $X_{2,1,2}[1]$ into two subspacec $X_{2,2,1}[1]$ and $X_{2,2,2}[1]$, where $X_{2,2,1}[1]$ stands for $\drm{span}{ \{ X_{2,1,2}[1] : \{3,13,15\} \} }$ and $X_{2,2,2}[1]$ is the orthogonal complement of $X_{2,2,1}[1]$ in $X_{2,1,2}[1]$. By examining the {\it dist}. of OPSB onto those two subspaces (see Fig. \ref{Fig-dist-SH} for details), we can see 
$$
\Pi_1^* = \{ \{1\},\{2,4,5,6,9,10\},\{3,13,15\},\{7,8,11,12,14,16\} \}.
$$
Furthermore, one can readily see that those four subspaces $X_{2,1,0}[1]$, $X_{2,1,3}[1]$, $X_{2,2,1}[1]$ and $X_{2,2,2}[1]$ are IRs of $\mathfrak{G}_1$ in the eigenspace $V_2$.

\begin{figure}
\begin{center}
\includegraphics[angle=2,totalheight=5.5cm]{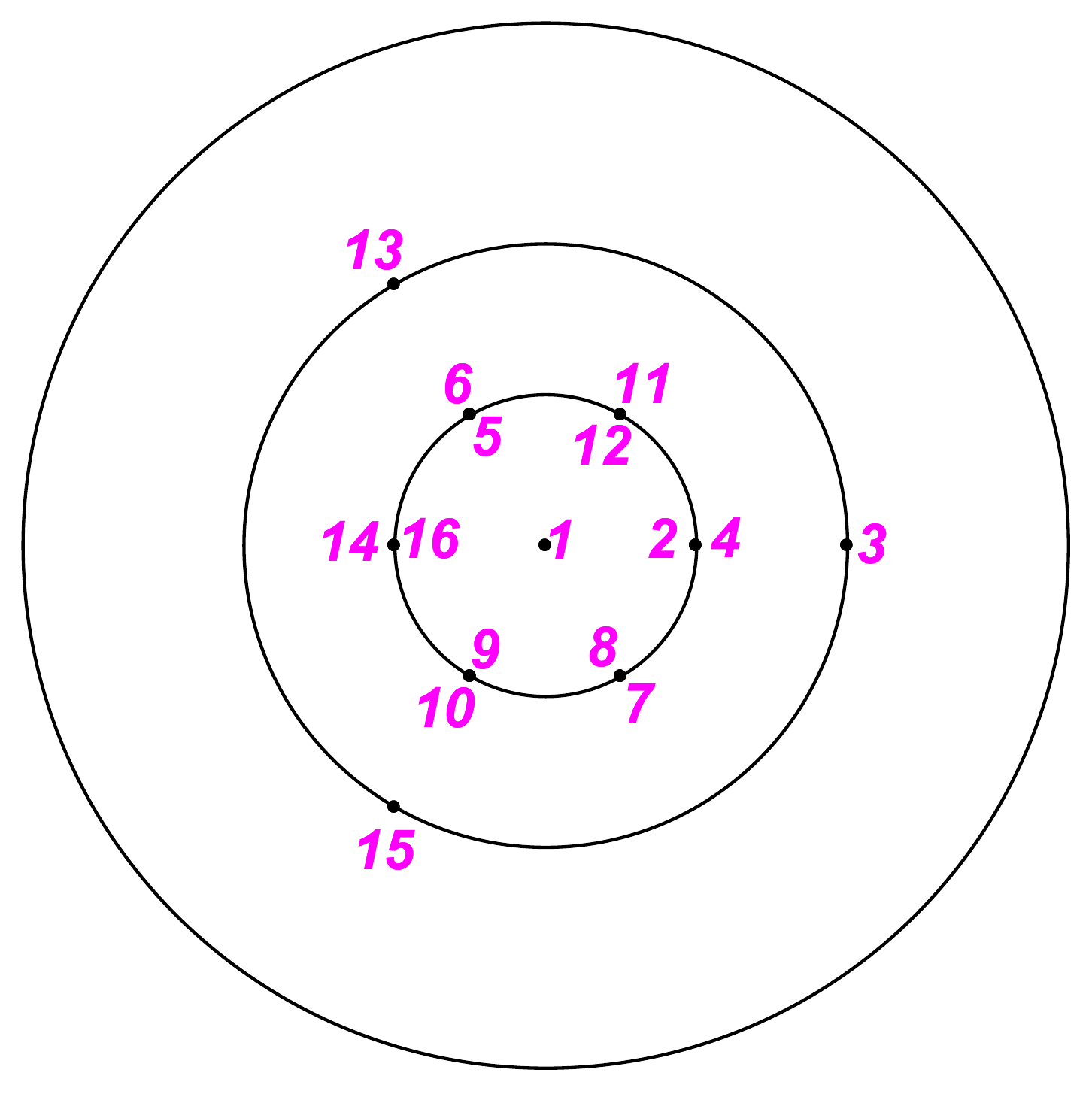} 
~~~~~~~~~ 
\includegraphics[angle=2,totalheight=5.5cm]{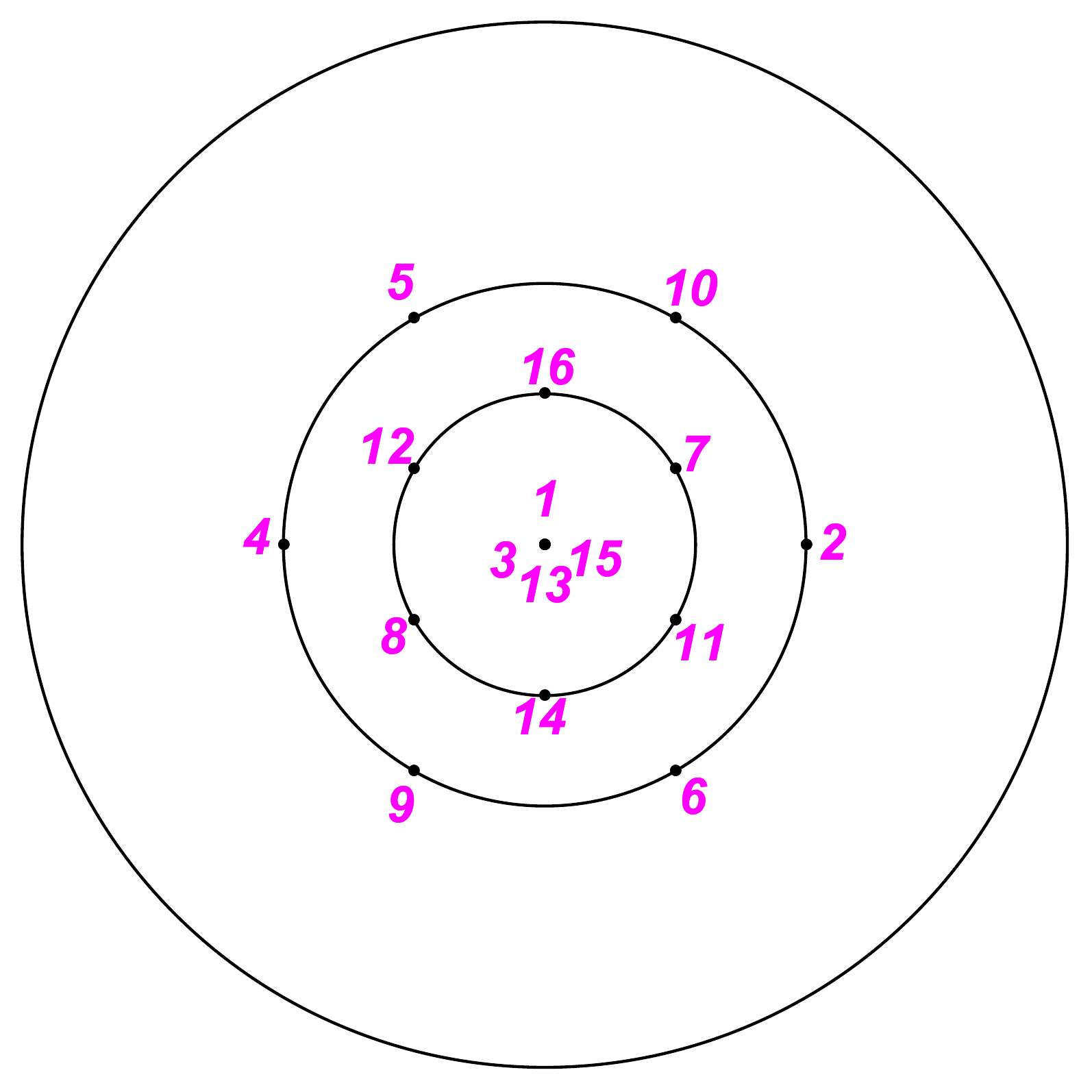}
\end{center}
\caption{The {\it dist.} of OPSB onto $X_{2,2,1}[1]$ and $X_{2,2,2}[1]$ }\label{Fig-dist-SH}
\end{figure}

Another key ingredient in finding out a block family is to figure out a group of automorphisms of $G$, by means of which we can determine the block system containing one specific block that we are interested in. Let us begin with figuring out one automorphism moving 1 to 2. As we have seen in the previous subsection, if two vertices 2 and 1 are symmetric in $G$, we can observe the same {\it dist}. of OPSB in subspaces obtained in the way same as what we did for the vertex 1, so we could see those automorphisms, one of which is shown below
$$
\sigma_{1,2} = (1 \hspace{0.8mm} 2)  (3 \hspace{0.8mm} 4)  (5 \hspace{0.8mm} 11)  (6 \hspace{0.8mm} 10)  (7 \hspace{0.8mm} 9)  (8 \hspace{0.8mm} 12)  (13 \hspace{0.8mm} 16)  (14 \hspace{0.8mm} 15).
$$
Consequently, 
$$
\Pi_2^* = \{ \{2\},\{1, 3, 6, 7, 10, 11\}, \{4, 14, 16\}, \{5, 8, 9, 12, 13, 15\} \}.
$$
Note that the graph $\dubipart{G}{1}{2}$ is connected, so the action of $\mathfrak{G}$ on $V$ is transitive due to Lemma \ref{LemFindBlocks-1}. As a result, in order to determine the orbits of a point stabilizer $\mathfrak{G}_v$, $v\in V \setminus \{1,2\}$, we do not have to decompose $V_2$ into IRs of $\mathfrak{G}_v$ but only need to find the one dimension subspace similar to $X_{2,1,3}[1]$, which shows us how to split the biggest cell in $\dep{v}{V_2}$ containing nine vertices. The following matrix displays the family $\{ \Pi_v^* : v \in V \}$.
$$
\left( 
 \begin{array}{rrrrrrrrrrrrrrrrrrr}
1 & || &  2 & 4 & 5 & 6 & 9 & 10 & || & 3 & 13 & 15 & || &  7 & 8 & 11 & 12 & 14 & 
   16\\
2 & || &  1 & 3 & 6 & 7 & 10 & 11 & || &     4 & 14 & 16 & || &  5 & 8 & 9 & 12 &
     13 & 15\\
3 & || &  2 & 4 & 7 & 8 & 11 & 12 & || &     1 & 13 & 15 & || &  5 & 6 & 9 & 10 &
     14 & 16\\

4 & || &  1 & 3 & 5 & 8 & 9 & 12 & || &       2 & 14 & 16 & || &  6 & 7 & 10 & 
    11 & 13 & 15\\

5 & || &  1 & 4 & 10 & 12 & 13 & 16 & || & 6 & 7 & 8 & || &       2 & 3 & 9 & 
    11 & 14 & 15\\

6 & || &  1 & 2 & 9 & 11 & 13 & 14 & || &  5 & 7 & 8 & || &       3 & 4 & 10 & 
    12 & 15 & 16\\

7 & || &  2 & 3 & 10 & 12 & 14 & 15 & || & 5 & 6 & 8 & || &       1 & 4 & 9 & 
    11 & 13 & 16\\

8 & || &  3 & 4 & 9 & 11 & 15 & 16 & || &  5 & 6 & 7 & || &       1 & 2 & 10 & 
    12 & 13 & 14\\
9 & || &  1 & 4 & 6 & 8 & 14 & 15 & || &     10 & 11 & 12 & || & 2 & 3 & 5 & 7 & 
    13 & 16\\

10 & || & 1 & 2 & 5 & 7 & 15 & 16 & || &     9 & 11 & 12 & || &  3 & 4 & 6 & 8 & 
    13 & 14\\

11 & || & 2 & 3 & 6 & 8 & 13 & 16 & || &     9 & 10 & 12 & || &  1 & 4 & 5 & 7 & 
    14 & 15\\

12 & || & 3 & 4 & 5 & 7 & 13 & 14 & || &     9 & 10 & 11 & || &  1 & 2 & 6 & 8 & 
    15 & 16\\
13 & || & 5 & 6 & 11 & 12 & 14 & 16 & || & 1 & 3 & 15 & || &    2 & 4 & 7 & 8 & 9 &
     10\\

14 & || & 6 & 7 & 9 & 12 & 13 & 15 & || &   2 & 4 & 16 & || &    1 & 3 & 5 & 8 & 
    10 & 11\\
15 & || & 7 & 8 & 9 & 10 & 14 & 16 & || &   1 & 3 & 13 & || &    2 & 4 & 5 & 6 & 
    11 & 12\\

16 & || & 5 & 8 & 10 & 11 & 13 & 15 & || & 2 & 4 & 14 & || &    1 & 3 & 6 & 7 & 9 & 
   12

\end{array}
\right)
$$
In accordance with Lemma \ref{LemFindBlocks-2} and Theorem \ref{ThmNeceAndSuffForPrimitiveness}, by checking  bipartite graphs $\dubipart{G}{1}{v}$, $v \in V \setminus \{1\}$, we can determine minimal blocks for $\mathfrak{G}$ containing the vertex 1. Accordingly, one can readily see that the subset $B := \{1,3,13,15\}$ is a minimal block, and thus the block system $\mathscr{B}$ containing $B$ must have four members. 

We now try to determine the system $\mathscr{B}$. First of all, note that the vertex 2 does not belong to the block $B$, so the automorphism $\sigma_{1,2}$ yields anther block $\{ 2,4,14,16 \}$ in the system $\mathscr{B}$. Clearly, there are eight vertices in $V \setminus ( B \cup \sigma_{1,2} B)$, so in order to find other members of $\mathscr{B}$, we need two more automorphisms of $G$. By virtue of the method we use in figuring out $\sigma_{1,2}$, one can readily figure out the following automorphisms:
$$
\sigma_{1,8} = 
(1 \hspace{0.8mm} 8 \hspace{0.8mm} 10 \hspace{0.8mm} 3 \hspace{0.8mm} 5 \hspace{0.8mm} 11) (2 \hspace{0.8mm} 4 \hspace{0.8mm} 16) (6 \hspace{0.8mm} 9 \hspace{0.8mm} 15 \hspace{0.8mm} 7 \hspace{0.8mm} 12 \hspace{0.8mm} 13),
$$
and 
$$
\sigma_{1,9} = (1 \hspace{0.8mm} 9 \hspace{0.8mm} 14 \hspace{0.8mm} 13 \hspace{0.8mm} 11 \hspace{0.8mm} 2) (3 \hspace{0.8mm} 10 \hspace{0.8mm} 4 \hspace{0.8mm} 15 \hspace{0.8mm} 12 \hspace{0.8mm} 16) (5 \hspace{0.8mm} 8 \hspace{0.8mm} 7).
$$
Accordingly, 
$$
\mathscr{B} = \{ \{1,3,13,15\},\{2,4,14,16\},\{5,6,7,8\},\{9,10,11,12\} \}.
$$

In order to obtain a minimal block for $\mathfrak{G}$ containing properly $B$ as one member, we construct a new bipartite graph $\dubipart{G}{B'}{B''}$, which is similar to $\dubipart{G}{v'}{v''}$. The vertex set of the graph is comprised of orbits of the action of $\mathfrak{G}_{B'}$ and of $\mathfrak{G}_{B''}$ on the system $\mathscr{B}$, and again two vertices of the graph are adjacent if the intersection of two subsets corresponding to the vertices is not empty. By means of Lemma \ref{LemFindBlocks-1}, we can easily determine the orbits of the action of $\mathfrak{G}_{B'}$ on the system $\mathscr{B}$ with the aid of the bipartite graph $\dubipart{G}{x'}{y'}$, where $x'$ and $y'$ both belong to the block $B'$.

In our case, one can readily verify that each one of six graphs $\dubipart{G}{B'}{B''}$ is connected, where $B',B'' \in \mathscr{B}$, and hence the action of $\mathfrak{G}$ on the system $\mathscr{B}$ is primitive due to Theorem \ref{ThmNeceAndSuffForPrimitiveness}. Therefore, the block $B$ itself is a block family of $\mathfrak{G}$.

\vspace{3mm}
As we have seen, only by the family $\{ \Pi_v^* : v \in V \}$ and a group of automorphisms $\{ \sigma_{1,2}, \sigma_{1,8}, \sigma_{1,9} \}$, we figure out a block family of $\mathfrak{G}$. As a matter of fact, we can also decompose only by them every eigenspace of $\dAM{G}$ into IRs of $\mathfrak{G}$.

Recall that $V_{\lambda}[v] = \{ \pmb{x} \in V_{\lambda} : \xi \hspace{0.5mm} \pmb{x} = \pmb{x}, \forall \hspace{0.6mm} \xi \in \mathfrak{G}_v \}$. Note $\dim V_2[1] = 1$, so the eigenspace $V_2$ is indeed irreducible for $\mathfrak{G}$ due to Lemma \ref{Lemma-IRsFeatureGvNonTrivial}, while $\dim V_{-2}[1] = 2$, so the eigenspace $V_{-2}$ contains two IRs of $\mathfrak{G}$ due to the same lemma.

There are two ways of decomposing $V_{-2}$ into IRs of $\mathfrak{G}$. The 1st one is to use Theorem \ref{Thm-EquitablePart&BlockSystem}, which asserts that any block system of $\mathfrak{G}$ is indeed an equitable partition. Let $\Pi_{\mathscr{B}}$ is the partition of $V$ consisting of blocks in $\mathscr{B}$. Then $\dChMat{\Pi_{\mathscr{B}}} V_{-2}^{ G / \Pi_{\mathscr{B}} }$, denoted by $X$ for short, is a subspace in $V_{-2}$ representing the system $\mathscr{B}$, which is of dimension 3. As a result, the eigenspace $V_{-2}$ can be decomposed into $X$ and its orthogonal complement in $V_{-2}$, that are the two IRs in $V_{-2}$.

The 2nd way is to use Theorem \ref{Thm-IsomorphismThmBetweenIRs}. According to the result, the intersection $\cap_{v \in \{ 1,3,13,15 \}} V_{-2}[v]$ is not trivial but spanned by the vector which belongs to the IR $X$ representing the block $\{ 1,3,13,15 \}$. In this way we can determine other three vectors in $X$ representing blocks $\{2,4,14,16\}$, $\{5,6,7,8\}$ and $\{9,10,11,12\}$ respectively. Apparently, $X$ is spanned by those four vectors, so we can single out the subspace $X$ from $V_{-2}$.

As we shall see in the following sections, we in general need only two sets of data - a partition $\Pi_{t}^*$, consisting of orbits of $\mathfrak{G}_t$, and a subset $E(t)$ of $\mathfrak{G}$ such that $\forall\hspace{0.5mm} x \in T$, $\exists| \sigma\in E(t)$ {\it s.t.,} $\sigma\hspace{0.5mm} t = x$, in order to find out a block system of $\mathfrak{G}$ contained in some orbit $T$ of $\mathfrak{G}$ and to decompose an eigenspace of $\dAM{G}$ into IRs of $\mathfrak{G}$ onto which the projections of $\pmb{e}_t$ is not trivial.

\section{Blocks and block families of $\mathfrak{G}$ }

In this section, our goal is to establish a series of results which enable us to find out blocks and block families of $\mathfrak{G}$. We start with two characterizations of blocks, which reveal the connection between the action of $\mathfrak{G}$ on a block system and that on a family of subsets in $\mathfrak{G}$: in the 1st lemma the family is comprised of the  cosets of $\mathfrak{G}_B$ while in the 2nd lemma the family consists of the conjugations of $\mathfrak{G}_B$. In the case that $\mathfrak{G}$ is not transitive, the connections enables us to see how block systems in distinct orbits of $\mathfrak{G}$ are related to one another.

\begin{Lemma} \label{LemBlockStructure}
Let $\mathfrak{G}$ be a permutation group acting on $V$ with an orbit $S$ and let $\mathfrak{H}$ be a subgroup of $\mathfrak{G}$ with an orbit $T$. If the subset $T$ is contained in $S$ and $|S|/|T| = [\mathfrak{G}:\mathfrak{H}]$, 
then $T$ is a block for $\mathfrak{G}$ and $\mathfrak{H}$ is the stabilizer of $T$, 
{\it i.e.,} $\mathfrak{H}=\dfix{G}{T}$. 
Conversely, if $B$ is a block for $\mathfrak{G}$ 
then its stabilizer $\dfix{G}{B}$ enjoys the relationship that $|S|/|B|=[\mathfrak{G}:\dfix{G}{B}]$.  \end{Lemma}

According to the relation above, if $B$ is a block then $\#\{ \sigma B : \sigma \in \mathfrak{G} \} = \#\{ \sigma \mathfrak{G}_B : \sigma \in \mathfrak{G} \}$, while in general $\#\{ \sigma B : \sigma \in \mathfrak{G} \} \geq \#\{ \sigma \mathfrak{G}_B \sigma^{-1} : \sigma \in \mathfrak{G} \}$, and the equality holds if and only if  $\gamma \mathfrak{G}_B \gamma^{-1} \neq \mathfrak{G}_B$, $\forall \gamma \in \mathfrak{G}\setminus\mathfrak{G}_B$. As a matter of fact, the last relation is essentially sufficient for being a block.

\begin{Lemma}\label{Lem-BlockFeature}
Let $\mathfrak{G}$ be a permutation group acting on $V$. Suppose $\mathfrak{H}$ is a subgroup of $\mathfrak{G}$ and $T$ is one of orbits of $\mathfrak{H}$ such that 
\begin{enumerate}

\item[(1)] the normalizer $N(\mathfrak{H})$ is $\mathfrak{H}$ itself, {\it i.e.,} $\gamma \mathfrak{H} \gamma^{-1} \neq \mathfrak{H}$, $\forall \gamma \in \mathfrak{G} \setminus \mathfrak{H}$;

\item[(2)] if $\sigma\mathfrak{H}\sigma^{-1} \neq \mathfrak{H}$, $\sigma\in\mathfrak{G}$, then $\sigma T \cap T = \emptyset$.
 
\end{enumerate}
Then $T$ is a block for $\mathfrak{G}$.
 \end{Lemma}
 
\begin{proof}[\bf Proof to Lemma \ref{LemBlockStructure}] 
We begin with a simple observation that if $\xi\mathfrak{H}=\zeta\mathfrak{H}$ then $\xi T = \zeta T$, where $\xi$ and $\zeta$ are two permutations in $\mathfrak{G}$. 
In fact, if $\xi\mathfrak{H}=\zeta\mathfrak{H}$ then there exists $\delta$ in $\mathfrak{H}$ so that $\xi = \zeta\cdot\delta$, 
and thus $\xi T = \zeta\cdot\delta T$. Since $T$ is an orbit of $\mathfrak{H}$ and $\delta\in\mathfrak{H}$, 
$\zeta\cdot\delta T = \zeta T$ and therefore $\xi T = \zeta T$. 
As a result, we can define a map $\phi$ from the set $\{ \sigma \mathfrak{H} : \sigma \in \mathfrak{G} \}$ to the set $O_T=\{\sigma T : \sigma\in\mathfrak{G}\}$, 
{\it s.t.,}  $\phi(\sigma\mathfrak{H})=\sigma T$.

One can readily see that $\phi$ is surjective. This is because left cosets of $\mathfrak{H}$ in $\mathfrak{G}$ constitute a partition of $\mathfrak{G}$, so for any element $\sigma$ in $\mathfrak{G}$ there is a left coset $\alpha\mathfrak{H}$ such that $\sigma \in\alpha\mathfrak{H}$ and $\alpha\mathfrak{H}=\sigma\mathfrak{H}$, and thus $\phi(\sigma\mathfrak{H})=\sigma T$. Accordingly, $|O_T|\leq [\mathfrak{G}:\mathfrak{H}]$. 

Further, it is apparent that $|S|/|T|\leq |O_T|$ for the action of $\mathfrak{G}$ on $S$ is transitive.
In accordance with the assumption that $|S|/|T|=[\mathfrak{G}:\mathfrak{H}]$, we now have 
$$ [\mathfrak{G}:\mathfrak{H}] = |O_T| = |S|/|T|.$$
Consequently, if $\sigma \notin \mathfrak{H}$ then $\sigma T \cap T = \emptyset$, and hence $T$ is a block for $\mathfrak{G}$.

\vspace{2mm}
For the second claim that $\mathfrak{H}$ is the stabilizer of $T$, we first note that $\mathfrak{H}\leq\dfix{G}{T}$, so $[\mathfrak{G}:\dfix{G}{T}] \leq  [\mathfrak{G}:\mathfrak{H}]$. On the other hand, by the same argument one can readily show that there exists a surjection from $\{ \sigma \mathfrak{G}_T : \sigma \in \mathfrak{G} \}$ to $\{ \sigma T : \sigma \in \mathfrak{G} \}$. Consequently, $[\mathfrak{G}:\dfix{G}{T}] \geq \# \{ \sigma T : \sigma \in \mathfrak{G} \} = |S| / |T|$, and therefore $[\mathfrak{G}:\dfix{G}{T}] \geq  [\mathfrak{G}:\mathfrak{H}]$.

\vspace{2mm}
We now turn to the second assertion that if $B$ is a block for $\mathfrak{G}$ 
then $|S|/|B|=[\mathfrak{G}:\dfix{G}{B}]$. First, we have $|O_B|=|S|/|B|$ for $B$ is a block, where $O_B:=\{\sigma B : \sigma\in\mathfrak{G}\}$. Similarly, we can define a map $\phi$ from $\{ \sigma\dfix{G}{B} : \sigma\in\mathfrak{G} \}$ to $O_B$. It is easy to see that $\phi$ is a surjection.

To hold our desire, it is sufficient to show that $\phi$ is an injection. Suppose that $\xi$ and $\zeta$ are two permutations in $\mathfrak{G}$ so that $\xi B = \zeta B$. Then $\zeta^{-1} \xi B = B$ and thus $\zeta^{-1}\xi \in \dfix{G}{B}$, thereby $\xi\dfix{G}{B}=\zeta\dfix{G}{B}$. Hence $\phi$ is an injection.
\end{proof}

\begin{proof}[\bf Proof to Lemma \ref{Lem-BlockFeature}]
Let $\sigma$ be a permutation in $\mathfrak{G}$ such that $\sigma T \cap T \neq \emptyset$. Then $\sigma \mathfrak{H} \sigma^{-1} = \mathfrak{H}$ according to the 2nd relation above, and thus $\sigma \in \mathfrak{H}$ since $N(\mathfrak{H}) = \mathfrak{H}$. Consequently, $\sigma T = T$ and therefore $T$ is a block for $\mathfrak{G}$. 
\end{proof}

In addition, block systems of $\mathfrak{G}$ have a close relation with normal subgroups.

\begin{Lemma}[Dixon and Mortimer \cite{DM}] \label{LemNormalSubgTOBlock}  
Let $\mathfrak{G} $ be a permutation group acting on a finite set $V$, and let $\mathfrak{N}$ be a normal subgroup  of $\mathfrak{G} $. Then the orbits of $\mathfrak{N}$ comprise a system $\mathscr{B}$ of blocks for $\mathfrak{G}$, and furthermore $\mathfrak{N}\le\bigcap_{B\in \mathscr{B}} \mathfrak{G}_{B}$.  
\end{Lemma}

\begin{Lemma}\label{LemBlockTONormalSubgroup} 
Let $\mathfrak{G}$ be a permutation group acting on a finite set $V$, and let $B$ be a block for $\mathfrak{G}$ contained in one orbit of $\mathfrak{G}$. Then $\dfix{G}{\sigma B} =\sigma\dfix{G}{B}\sigma^{-1}$, $\forall \sigma \in \mathfrak{G}$, and $\mathfrak{N} := \bigcap_{\sigma\in\mathfrak{G}} \mathfrak{G}_{\sigma B}$ is a normal subgroup of $\mathfrak{G}$. Moreover, $\mathfrak{N} b = B$ if and only if $\mathfrak{N}\dfix{G}{b} = \dfix{G}{B}$, where $b$ is one member of $B$.  \end{Lemma}

\begin{proof}[\bf Proof] Clearly $\left(\sigma\dfix{G}{B}\sigma^{-1}\right) \sigma B = \sigma B$. Thus $\sigma\dfix{G}{B}\sigma^{-1} \le \dfix{G}{\sigma B}$.  Similarly, $\left(\sigma^{-1} \dfix{G}{\sigma B} \sigma \right) B$ $=$ $B$, so $\sigma^{-1} \dfix{G}{\sigma B} \sigma \le \dfix{G}{B}$. Accordingly, we have $$|\dfix{G}{B}| = |\sigma \dfix{G}{B} \sigma^{-1}| \le |\dfix{G}{\sigma B}| = |\sigma^{-1} \dfix{G}{\sigma B} \sigma| \le |\dfix{G}{B}|.$$
Consequently,  $\dfix{G}{\sigma B} =\sigma\dfix{G}{B}\sigma^{-1}$, $\forall \sigma \in \mathfrak{G}$, and hence $\bigcap_{\sigma\in\mathfrak{G}} \dfix{G}{\sigma B} = \bigcap_{\sigma\in\mathfrak{G}} \left(\sigma \dfix{G}{B} \sigma^{-1} \right) $. 

On the other hand, for any permutation $\gamma \in \mathfrak{G}$, 
\begin{align*} 
\gamma\left(\bigcap_{\sigma\in\mathfrak{G}} \left(\sigma \dfix{G}{B} \sigma^{-1} \right) \right)\gamma^{-1} 
& =  \bigcap_{\sigma\in\mathfrak{G}} \left( (\gamma\sigma) \dfix{G}{B} (\gamma\sigma)^{-1}  \right)\\ 
& = \bigcap_{\sigma\in\mathfrak{G}} \left(\sigma \dfix{G}{B} \sigma^{-1} \right) 
\end{align*}
As a result, $\bigcap_{\sigma\in\mathfrak{G}} \dfix{G}{\sigma B}$ is a normal subgroup of $\mathfrak{G}$.

\vspace{2mm}
We now turn to the second part of the lemma and first show the sufficiency of the assertion. Since $B$ is a block for $\mathfrak{G}$, $\dfix{G}{B} b = B$ for some $b \in B$, which means $\forall \hspace{0.5mm} b' \in B$, $\exists \hspace{0.5mm} \xi\in\dfix{G}{B}$ {\it s.t.,} $b' = \xi b$. On the other hand, in accordance with the assumption that $\mathfrak{N}\dfix{G}{b} = \dfix{G}{B}$, one can find $\delta\in\mathfrak{N}$ and $\zeta\in\dfix{G}{b}$ such that $\xi = \delta\zeta$. Consequently, $b' = \delta\zeta b = \delta b$ and therefore, $B = \mathfrak{N} b$.

For the necessity, one can easily check that if $\delta_1\dfix{G}{b} = \delta_2\dfix{G}{b}$ then $\delta_1 b = \delta_2 b$, where $\delta_1, \delta_2 \in \mathfrak{N}$. Consequently, we can define a map $\phi$ from $\{ \delta\dfix{G}{b} : \delta\in\mathfrak{N} \}$ to $B$ such that $\phi(\delta\dfix{G}{b}) = \delta b$. In accordance with the condition that $\mathfrak{N} b = B$, the map $\phi$ is surjective, so $\# \{ \delta\dfix{G}{b} : \delta\in\mathfrak{N} \}\geq |B|$.

On the other hand, since $B$ is an orbit of $\dfix{G}{B}$, $|B| = [\dfix{G}{B} : \dfix{G}{b}]$. Moreover, according to the definition of $\mathfrak{N}$, it is easy to see that 
$\{ \delta\dfix{G}{b} : \delta\in\mathfrak{N} \} \subseteq 
\{ \xi\dfix{G}{b} : \xi\in\dfix{G}{B} \}$. Consequently, the order of $\{ \delta\dfix{G}{b} : \delta\in\mathfrak{N} \}$ is not more than $|B|$, i.e., 
$\# \{ \delta\dfix{G}{b} : \delta\in\mathfrak{N} \}\leq |B|$.
As a result, the map $\phi$ is a bijection and therefore $\{ \delta\dfix{G}{b} : \delta\in\mathfrak{N} \} = 
\{ \xi\dfix{G}{b} : \xi\in\dfix{G}{B} \}$.\end{proof}

According to the relation above, if $\mathfrak{G}$ possesses such a normal subgroup $\mathfrak{N}$ so that $\mathfrak{N} b = B$ then the distribution of $\mathfrak{N}$ in $\dfix{G}{B}$ is well-proportioned. To be precise, for any left coset $\zeta\dfix{G}{b}$, it must contain some of members of $\mathfrak{N}$. In fact, suppose $\dfix{G}{B} = \cup_{i=1}^l \zeta_i\dfix{G}{b}$ and $\xi \in \mathfrak{G}_B \setminus \left( \cup_{j=1}^k \hspace{0.5mm} \zeta_{i_j} \dfix{G}{b} \right)$ $(1\leq k\leq l-1)$. Since the relationship that $\mathfrak{N} b = B$ is equivalent to that $\mathfrak{N}\dfix{G}{b} = \mathfrak{G}_B$, one can find $\delta\in\mathfrak{N}$ and $\gamma\in\dfix{G}{b}$ such that $\xi = \delta\gamma$, that is an element of the left coset $\delta\dfix{G}{b}$. As a result, $\delta\dfix{G}{b}\cap \cup_{j=1}^k \zeta_{i_j} \dfix{G}{b} = \emptyset$ and thus $\delta\notin \cup_{j=1}^k \zeta_{i_j} \dfix{G}{b}$.

However there does exist such permutation groups such that $\mathfrak{N} b \subsetneq B$, $\forall \hspace{0.5mm} b \in B$. For instance, we consider the action of left product of $A_5$ on itself. Evidently, $\langle (1 2 3 4 5) \rangle$ is a subgroup of $A_5$, that is generated by the permutation $(1 2 3 4 5)$, and the family consisting of left cosets of $\langle (1 2 3 4 5) \rangle$ is indeed a system of blocks of the action by virtue of Lemma \ref{LemBlockStructure}. But it is well-known that $A_5$ is simple and thus it possesses no any non-trivial normal subgroups.

\vspace{3mm}
We are now ready to show how to find blocks for $\mathfrak{G}$. As we have seen in the 1st section, the bipartite graph $\dbipart{G}{v}$ plays a critical role in working out blocks for $\mathfrak{G}$. Recall that the vertex set of $\dbipart{G}{v}$ consists of orbits of $\mathfrak{G}_{v'}$ and of $\mathfrak{G}_{v''}$, and two vertices in the graph are adjacent if the intersection of two subsets corresponding to the vertices is not empty. There are essentially two kinds of blocks for a permutation group, and the component $C[v']$ in $\dbipart{G}{v}$ shows us one of them.  

\begin{Lemma}\label{LemFindBlocks-1}
Let $C[x]$ be a component in $\dbipart{G}{v}$, where $x$ is an element of $V$. Then $C[x] = \langle \mathfrak{G}_{v'}, \mathfrak{G}_{v''} \rangle x$.
\end{Lemma}
\begin{proof}[\bf Proof]
Suppose $y$ is in the subset $\langle \mathfrak{G}_{v'}, \mathfrak{G}_{v''} \rangle x$. Then $\exists ~\sigma_1\ldots\sigma_m\in\mathfrak{G}_{v'}$ and $\gamma_1\ldots\gamma_m\in\mathfrak{G}_{v''}$ {\it s.t.,} $y = (\Pi_{i} \sigma_i\gamma_i) x$. According to the definition to $\dbipart{G}{v}$, it is easy to see that $y \in C[x]$, and thus $\langle \mathfrak{G}_{v'}, \mathfrak{G}_{v''} \rangle x \subseteq C[x]$.

By using the same argument, one can readily see that $C[x] \subseteq \langle \mathfrak{G}_{v'}, \mathfrak{G}_{v''} \rangle x$.
\end{proof}

\begin{Lemma}\label{LemFindBlocks-2}
The component $C[v']$ in $\dbipart{G}{v}$ is a block for $\mathfrak{G}$.
\end{Lemma}
\begin{proof}[\bf Proof]
In accordance with Lemma \ref{LemFindBlocks-1}, $C[v'] = \langle \mathfrak{G}_{v'}, \mathfrak{G}_{v''} \rangle v'$. To show $C[v']$ is a block, it is sufficient to prove that if $\sigma$ is a permutation in $\mathfrak{G}$ such that $\sigma \big(\langle \mathfrak{G}_{v'}, \mathfrak{G}_{v''} \rangle v'\big) \cap \big(\langle \mathfrak{G}_{v'}, \mathfrak{G}_{v''} \rangle v'\big) \neq \emptyset$, then $\sigma \big(\langle \mathfrak{G}_{v'}, \mathfrak{G}_{v''} \rangle v'\big) = \langle \mathfrak{G}_{v'}, \mathfrak{G}_{v''} \rangle v'$.

Suppose $\exists ~\xi,\zeta \in \langle \mathfrak{G}_{v'}, \mathfrak{G}_{v''} \rangle$, {\it s.t.,} $\sigma\xi v' = \zeta v'$. Then $\zeta^{-1} \sigma \xi v' = v'$ $\Rightarrow$ $\sigma \in \zeta\mathfrak{G}_{v'}\xi^{-1} \subseteq \langle \mathfrak{G}_{v'}, \mathfrak{G}_{v''} \rangle$.
\end{proof}

Although component $C[v']$ in $\dbipart{G}{v}$ must be a block for $\mathfrak{G}$, it is possible that $C[v']$ contains only one vertex in $V$. For instance, $C[1]$ in $\dubipart{G}{1}{8}$ only contains the vertex 1 and $C[8]$ only 8, where $\mathfrak{G}$ stands for the automorphism group of the cube (see the diagram below).

\begin{center}\setlength{\unitlength}{0.5bp}
\hspace{-6mm}\includegraphics[width=4cm]{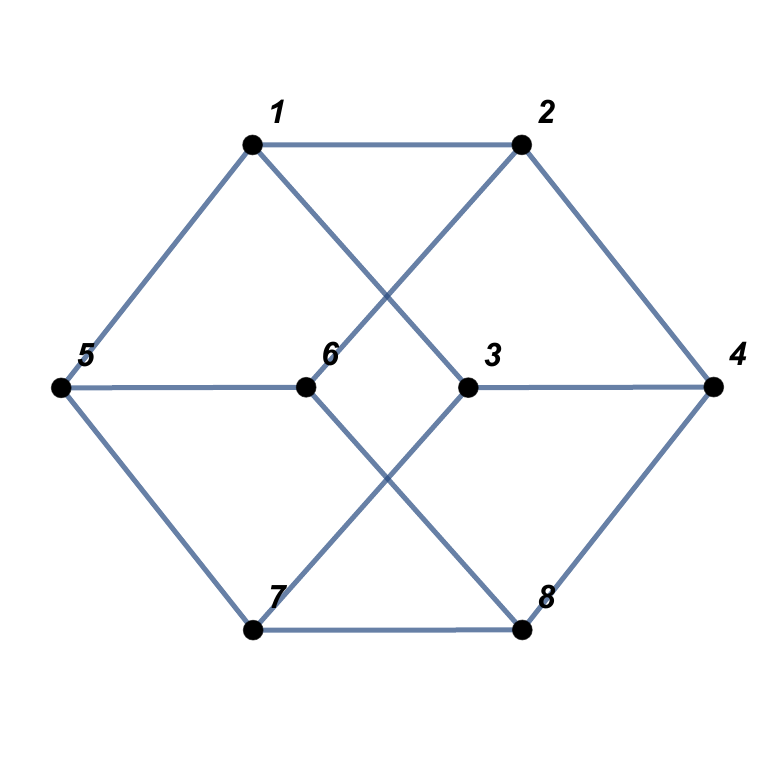}
\end{center}

In order to deal with that case, we introduce a binary relation among vertices in $V$. Evidently, the orbits of $\mathfrak{G}_v$ constitue a partition of $V$, which is denoted by $\Pi_v^*$. Accordingly, one can define a binary relation for any $x$ and $y$ in some orbit $T$ of $\mathfrak{G}$ such that $x \sim y$ if $\Pi_x^*= \Pi_y^*$. Obviously, it is an equivalence relation on $T$, so it could induce a partition of $T$, which is denoted by $\widetilde{\Pi}[T]$.

\begin{Lemma}\label{LemFindBlocks-3}
Every cell in $\widetilde{\Pi}[T]$ contains the same number of vertices, and all cells of $\widetilde{\Pi}[T]$ constitute a block system of $\mathfrak{G}$.
\end{Lemma}
\begin{proof}[\bf Proof]
Since the action of $\mathfrak{G}$ on $T$ is transitive, if $\sigma$ is a permutation in $\mathfrak{G}$ such that $\sigma x = y$ then $\mathfrak{G}_y = \sigma \mathfrak{G}_x \sigma^{-1}$ due to Lemma \ref{LemBlockTONormalSubgroup}, where $x, y \in T$. Accordingly, if $\mathfrak{G}_x X = X$, $X\subseteq T$, then $\mathfrak{G}_y (\sigma X)= \sigma X$. Hence for any $t \in T$, $\Pi_t^*$ have the same number of cells of order one, so the first claim holds.

Suppose $C_s$ is a cell of $\widetilde{\Pi}[T]$ containing the vertex $s$. We pick arbitrarily a member $y$ in $T \setminus C_s$. Let $\sigma$ be a permutation in $\mathfrak{G}$ such that $\sigma s = y$. Then $\sigma C_s \neq C_s$. To show $C_s$ is a block for $\mathfrak{G}$, it is sufficient to prove that $\sigma C_s \cap C_s = \emptyset$.

Note that $y\notin C_s$, which means $s\nsim y$. Consequently, the cell containing $s$ in $\Pi_y^*$ cannot be singleton, otherwise $\mathfrak{G}_y s = s$ $\Rightarrow$ $\mathfrak{G}_y \leq \mathfrak{G}_s$. Then $\mathfrak{G}_y = \mathfrak{G}_s $ and thus $\Pi_y^*=\Pi_s^*$, which contradicts the assumption that $y\notin C_s$. As a result, any member in $C_s$ cannot be singleton in $\Pi_y^*$. On the other hand, $\mathfrak{G}_y (\sigma x) = \left( \sigma \mathfrak{G}_s \sigma^{-1} \right) (\sigma x) = \sigma x$ for any $x \in C_s$, {\it i.e.,} $\sigma x$ is a singleton in $\Pi_y^*$. Therefore $C_s \cap \sigma C_s = \emptyset$.

Accordingly, one can readily see that for any $\sigma \in \mathfrak{G}$, $\sigma C_s$ also belongs to $\widetilde{\Pi}[T]$, so $\widetilde{\Pi}[T]$ is a block system.
\end{proof}

Apparently, there are only two such kinds of blocks for a permutation group, which can be found out by $\{ \Pi_v^* : v \in V \}$. Before we show how to obtain a block family of $\mathfrak{G}$, let us present a different characterization of primitive permutation groups.

\begin{Theorem}\label{ThmNeceAndSuffForPrimitiveness}
Let $\mathfrak{G}$ be a permutation group acting on $V$ transitively. Then $\mathfrak{G}$ is primitive if and only if one of two cases below occurs 
\begin{enumerate}
\item[{\rm i)}] $\dbipart{G}{v}$ is connected, $\forall\hspace{0.6mm} v',v'' \in V$;

\item[{\rm ii)}] $\dbipart{G}{v}$ is a perfect matching consisting of $|S|$ edges, $\forall\hspace{0.6mm} v',v'' \in V$, and $|V|$ is a prime number. In fact, $\mathfrak{G}$ is a circulant group of prime order in this case.
\end{enumerate}
\end{Theorem}
\begin{proof}[\bf Proof]
Let us begin with the sufficiency of our assertion. In the case i), if there exists a non-trivial block $B$ for $\mathfrak{G}$, then the bipartite graph $\dbipart{G}{b}$ cannot be connected for any vertices $b'$ and $b''$ in $B$. In fact, the component $C[b']$ in $\dbipart{G}{b}$, due to Lemma \ref{LemFindBlocks-1}, consisting of vertices in $\left\langle \mathfrak{G}_{b'},\mathfrak{G}_{b''} \right\rangle b'$, is contained in $\mathfrak{G}_B b' = B$. This is in contradiction with the assumption that $\dbipart{G}{b}$ is connected.

Obviously, the action of $\mathfrak{G}$ on $V$ is primitive in the case ii).

\vspace{2mm} We now turn to the necessity of the assertion. Clearly, there are only two possibilities for each stabilizer $\mathfrak{G}_v$: $\mathfrak{G}_v \supsetneq \{ 1 \}$ or $\mathfrak{G}_v = \{ 1 \}$. Because $\mathfrak{G}$ is primitive, the subgroup $\mathfrak{G}_v$ is maximal due to Lemma \ref{LemPrimitiveness1stNeceSuffCond}. Hence for any permutation $\xi \in \mathfrak{G} \setminus \{1\}$, $\langle \xi \rangle = \langle \xi,\mathfrak{G}_v \rangle = \mathfrak{G}$ in the second case, which implies that $\mathfrak{G}$ is a circulant group of prime order.

According to Lemma \ref{LemFindBlocks-3}, $\forall\hspace{0.5mm} v',v'' \in V$, $\mathfrak{G}_{v'} \neq \mathfrak{G}_{v''}$, provided that $\mathfrak{G}_v \supsetneq \{ 1 \}$. On the other hand, the primitiveness of $\mathfrak{G}$ implies that $\langle \mathfrak{G}_{v'},\mathfrak{G}_{v''} \rangle = \mathfrak{G}$. By means of Lemma \ref{LemFindBlocks-1}, $C[v'] = \langle \mathfrak{G}_{v'}, \mathfrak{G}_{v''} \rangle v' = \mathfrak{G} v' = V$, so the graph $\dbipart{G}{v}$ is connected.
\end{proof}

\vspace{3mm}
It is easy to see that if $T$ is one of orbits of $\mathfrak{G}$ such that $\mathfrak{G}_t$ is trivial, where $t \in T$, then $|\mathfrak{G}| = |T|$. Accordingly, in order to obtain a generating set of $\mathfrak{G}$, we do not have to find out a block family of it, but figure out by our algorithm those permutations $\sigma_x$ for every $x \in T$ so that $\sigma_x t = x$.   
 
There are however some interesting properties possessed by $\mathfrak{G}_B$, where $B$ is one of minimal blocks for $\mathfrak{G}$. Let us begin with a straightforward property.

\begin{Lemma}\label{Lemma-BlockStabilizerFeatureGvTrivial}
Let $\mathfrak{G}$ be a permutation group acting on $V$. Suppose for any vertex $v \in V$ the stabilizer $\mathfrak{G}_v$ is trivial and $B$ is one of minimal blocks for $\mathfrak{G}$ contained in some orbit of $\mathfrak{G}$. Then 
\begin{enumerate}

\item[(1)] The action of $\mathfrak{G}_B$ on $B$ is primitive;

\item[(2)] $| \mathfrak{G}_B | = |B|$.
\end{enumerate}
\end{Lemma}

\begin{Lemma}\label{Lem-BlockStabilizerCharacterization-GvTrivial}
Let $\mathfrak{G}$ be a permutation group acting on $V$. Suppose for any vertex $v \in V$ the stabilizer $\mathfrak{G}_v$ is trivial and $B$ is one of minimal blocks for $\mathfrak{G}$ contained in some orbit of $\mathfrak{G}$. Then $\mathfrak{G}_B$ is a circulant group of prime order.
\end{Lemma}
\begin{proof}[\bf Proof]
We first show that for any permutation $\xi \in \mathfrak{G}_B\setminus\{1\}$, the subset $\langle \xi \rangle b := \{ \xi^k b \mid k = 1,\ldots,\mathrm{deg}[\xi] \}$ is a block for $\mathfrak{G}_B$. 

Let $\zeta$ be a permutation in $\mathfrak{G}_B$ such that $\langle \xi \rangle b \cap \zeta\langle \xi \rangle b \neq \emptyset$. Then there exist $s$ and $t$ such that $\xi^s b = \zeta \xi^t b$, so $b = \xi^{-s} \zeta \xi^{t} b$. Consequently, $\xi^{-s} \zeta \xi^{t} = 1$ for $\mathfrak{G}_v$ is trivial, and thus $\zeta = \xi^s\cdot\xi^{-t}$  which belongs to $\langle \xi \rangle$. Therefore, $\langle \xi \rangle b = \zeta\langle \xi \rangle b$, so $\langle \xi \rangle b$ is a block for $\mathfrak{G}_B$.

According to the 1st claim of Lemma \ref{Lemma-BlockStabilizerFeatureGvTrivial}, the action of $\mathfrak{G}_B$ on $B$ is primitive, so $\langle \xi \rangle b = B$ for $\xi$ is not the identity. As a result the order of $\langle \xi \rangle$ is equal to $|B|$, and therefore $\langle \xi \rangle = \mathfrak{G}_B$ due to the 2nd claim in Lemma \ref{Lemma-BlockStabilizerFeatureGvTrivial}. 

As a result, if $\xi$ belongs to $\mathfrak{G}_B$ not being the identity then $\langle \xi \rangle = \mathfrak{G}_B$, that means $\mathfrak{G}_B$ is a circulant group of prime order.
\end{proof}

By means of the result above, one can readily prove the following claims.

\begin{enumerate}
\item[(A)] If $\xi$ is not the identity in $\mathfrak{G}_B$, then any two cycles in the cycle decomposition of $\xi$ are of the same length which is a prime.

\item[(B)] Every orbit of $\mathfrak{G}_B$ is a block for $\mathfrak{G}$.

\end{enumerate}

Let $T$ be the orbit of $\mathfrak{G}$ containing $B$. Since $\mathfrak{G}_t$ is trivial for any $t$ of $T$, $\xi$ has no trivial cycles. Moreover, the fact that $\mathfrak{G}_B$ is a circulant group of prime order implies that any two cycles in the cycle decomposition of $\xi$ are of the same length.

Let $x$ be an element of $T$. Set $\mathfrak{G}_B x = \{ \gamma x \mid \gamma\in\mathfrak{G}_B \}$. Suppose $\sigma$ is a permutation in $\mathfrak{G}$ and $\sigma\mathfrak{G}_B x \cap \mathfrak{G}_B x \neq \emptyset$. Then there exist $\xi$ and $\zeta$ in $\mathfrak{G}_B$ so that $\sigma\xi x = \zeta x$. Consequently, $\zeta^{-1} \sigma \xi x = x$ and thus $\zeta^{-1} \sigma \xi = 1$ because $\mathfrak{G}_t$ is trivial for any $t$ in $T$. As a result, $\sigma = \zeta\xi^{-1}$, belonging to $\mathfrak{G}_B$, so $\sigma\mathfrak{G}_B x = \mathfrak{G}_B x$. Therefore, $\mathfrak{G}_B x$ is a block for $\mathfrak{G}$.

\vspace{2mm}
Recall that if $\mathfrak{N}$ is a normal subgroup of $\mathfrak{G}$ then every orbit of the action of $\mathfrak{N}$ on $V$ is a block for $\mathfrak{G}$, and moreover those orbits indeed form a block system, so Claim (B) above suggests that $\mathfrak{G}_B$ is somewhat like a normal subgroup. As a matter of fact, it is not difficult to see that if $\sigma B$ is also an orbit of $\mathfrak{G}_B$ for any $\sigma$ in $\mathfrak{G}$  then it is a normal subgroup.

\vspace{2mm}
A subset $E(t)$ of $\mathfrak{G}$ is said to be {\it adequate with respect to some orbit $T$} if for any vertex $v$ in $T$ there exists a permutation $\sigma$ in $E(t)$ such that $\sigma t = v$. Now let us see how to find out all block systems of $\mathfrak{G}$ contained in $T$ by means of one partition $\Pi_s^*$ and the adequate subset $E(s)$ relevant to $s \in T$. Because we only focus on one orbit of $\mathfrak{G}$ here, we could make a further assumption that the action of $\mathfrak{G}$ on $V$ is transitive.

In accordance with Lemma \ref{Lemma-BlockCHARACTERISEDbyProj}, any block system of $\mathfrak{G}$ is represented by some IRs of $\mathfrak{G}$, and hence there are at most $n$ block systems. On the other hand, each block must belong to some block family of $\mathfrak{G}$, so if we could find out block families one by one we can ultimately obtain all block systems. 

At first we figure out the family of partitions $\{ \Pi_v^* : v \in V \}$ by virtue of $\Pi_s^*$ and $E(s)$. Then we can easily determine whether the action of $\mathfrak{G}$ on $V$ is primitive or not by the bipartite graph $\dbipart{G}{v}$ due to Theorem \ref{ThmNeceAndSuffForPrimitiveness}. Therefore let us assume in what follows that the action is not primitive.

Obviously, there are two possibilities:
\begin{enumerate}
\item $\forall \hspace{0,5mm} u, v \in V$, $\Pi_u^* \neq \Pi_v^*$;

\item $\exists \hspace{0.5mm} u, v \in V$, {\it s.t.,} $\Pi_u^* = \Pi_v^*$.
\end{enumerate}

By using the same argument in proving Theorem \ref{ThmNeceAndSuffForPrimitiveness}, one can readily see that in the first case, a block $B$ for $\mathfrak{G}$ is minimal if the component $C[v']$ in the bipartite graph $\dbipart{G}{v}$ contains the vertex $v''$. Accordingly, we can figure out one minimal block contained in $V$.

In the second case, we first determine the partition $\widetilde{\Pi}[V]$ with the family $\{ \Pi_v^* : v \in V \}$. In accordance with Lemma \ref{LemFindBlocks-3}, the cells of $\widetilde{\Pi}[V]$ actually constitute a block system. Evidently, it is possible that a cell $C$ of $\widetilde{\Pi}[V]$ is not a minimal block for $\mathfrak{G}$, so let us show how to find out those minimal ones.

Note that if $x$ and $y$ are two members in $C$ then $\mathfrak{G}_x = \mathfrak{G}_y$, for those two partitions $\Pi_x^*$ and $\Pi_y^*$ are the same. Consequently, for any vertex $v$ in $C$, the restriction of $\mathfrak{G}_v$ to $C$ is trivial, {\it i.e.,} $\mathfrak{G}_v|_{C} = \{ 1 \}$. Hence if $\gamma$ is a permutation in $\mathfrak{G}_C$ and $\gamma x \neq x$ for some $x \in C$, then $\gamma\hspace{0.5mm} y \neq y$ for any $y \in C$, and  
$$
\left( \Pi_i \gamma^{k_i} \delta_i \right) x = \left( \Pi_i \gamma^{k_i} \right) x = \gamma^k x \mbox{ for some integer } k, \mbox{ where } \delta_i \in \mathfrak{G}_x.
$$
As a result, $\langle \gamma,\mathfrak{G}_x \rangle x = \langle \gamma|_C \rangle x$, and accordingly one can use the argument in proving Lemma \ref{Lem-BlockStabilizerCharacterization-GvTrivial} to show that $B \subsetneq C$ is a minimal block for $\mathfrak{G}$ if and only if $\mathfrak{G}_B|_C$ is a circulant group of prime order. Now by virtue of the adequate set $E(s)$ one can easily identify a permutation belonging to some minimal block contained in $C$. 

Above all, we see how to figure out one of minimal blocks. In order to find out a block family, we now need to determine a block $\hat{B}$ containing $B$ such that $B$ is maximal in $\hat{B}$. First of all, let figure out the block system $\mathscr{B}$ by virtue of $B$ and the adequate set $E(s)$. Denote by $\Pi_{B, \mathscr{B}}^*$ the orbits of the action of the stabilizer $\mathfrak{G}_B$ on the system $\mathscr{B}$. It is readily to see that we can determine the partition $\Pi_{B, \mathscr{B}}^*$ by means of the bipartite graph $\dbipart{G}{v}$, and then we can obtain the family of partitions $\{ \Pi_{\sigma B, \mathscr{B}}^* : \sigma \in \mathfrak{G} \}$ via the adequate set $E(s)$. 

Now using the same idea in establishing Theorem \ref{ThmNeceAndSuffForPrimitiveness}, one can easily see whether the action of $\mathfrak{G}$ on $\mathscr{B}$ is primitive or not. In the case being imprimitive, there are clearly two possibilities:
\begin{enumerate}
\item $\forall \hspace{0,5mm} B_u, B_v \in \mathscr{B}$, $\Pi_{B_u,\mathscr{B}}^* \neq \Pi_{B_v,\mathscr{B}}^*$;

\item $\exists \hspace{0.5mm} B_u, B_v \in \mathscr{B}$, {\it s.t.,} $\Pi_{B_u,\mathscr{B}}^* = \Pi_{B_v,\mathscr{B}}^*$.
\end{enumerate}
One moment's reflection would show that by virtue of the approach we determine the minimal block $B$ one can now figure out the block $\hat{B}$, and accordingly one can eventually find out a block family containing $B$. It is not difficult to check that there are at most $n^4$ bipartite graphs like $\dbipart{G}{v}$ or partitions like $\widetilde{\Pi}[V]$ we use to obtain all block families, where $n$ is the number of vertices $G$ possesses, so provided that we have one partition $\Pi_s^*$ and its adequate set $E(s)$, we can efficiently work out all block systems of $\mathfrak{G}$.

\section{Blocks and block systems in IRs of $\mathfrak{G}$}

As well-known, block systems expose the structure of the action of $\mathfrak{G}$ on $V$. In this section, we shall see that block systems are represented very well by IRs of $\mathfrak{G}$.\footnote{Note that in this paper we investigate one special permutation representation of $\mathfrak{G}$ in $\R^n$ which is defined in the 1st section by (\ref{Equ-TheRepresent}).} In other words, block systems also reveal quite clearly the structure of the action of $\mathfrak{G}$ on $\R^n$.

Let us begin with an algebraic characterization of IRs of $\mathfrak{G}$ in $\R^n$, which is an analogue of Schur Lemma (see \cite{Serre} for details) that provides a fundamentally algebraic characterization of IRs of a linear representation of a finite group in $\C^n$. Because the scalar field is $\R$ here, the situation is a little more complex.

\begin{Lemma}\label{Lem-SchurLemmaOnEuclideanSpace}
Let $\mathfrak{G}$ be a permutation group acting on $V$ with $n$ elements. Suppose $W$ is a subspace of $\R^n$ which is irreducible for $\mathfrak{G}$ and $\tau$ is an $\mathfrak{G}$-module map on $W$. Then there is a constant $C$ such that $\tau = C\cdot I$ or $\tau = C\cdot S_{\theta}$, where $I$ is the identity operator and $S_\theta$ is an isometry with minimal polynomial $x^2 - 2\cos\theta\cdot x + 1$.
\end{Lemma}
\begin{proof}[\bf Proof]
In the case that $\tau$ has an eigenvalue, one can easily show that $\tau$ would be $C \cdot I$ by means of the same argument for proving Schur Lemma. In fact, $\tau - \lambda I$ cannot be injective if $\lambda$ is an eigenvalue of $\tau$, so $\mathrm{ker}~ (\tau - \lambda I)$ is not trivial. Since $\mathrm{ker}~ (\tau - \lambda I)$ is an $\mathfrak{G}$-invariant subspace and $W$ is irreducible, $\mathrm{ker}~ (\tau - \lambda I) = W$ and thus $\tau - \lambda I = 0$.

In the case of possessing no eigenvalues, we consider the polar decomposition of $\tau$ and assume that $S$ is the isometry on $W$ {\it s.t.,} $\tau = S \sqrt{\tau^* \tau}$. Note that the adjoint operator $\tau^*$ of $\tau$  is also an $\mathfrak{G}$-module map on condition that $\tau$ is, which implies that $\tau^* \tau$ is an $\mathfrak{G}$-module map. Since the operator $\tau^* \tau$ must have an eigenvalue, say, $\lambda$, 
$$
\tau^* \tau = \lambda I \mbox{ according to the first case.}
$$ 
Consequently, $\tau = \sqrt{\lambda} \cdot S$ and the isometry $S$ possesses no eigenvalues. 

Suppose $(\alpha, \beta)$ is an eigenpair of $S$, {\it i.e.,} $x^2 + \alpha x + \beta$ is a factor of the characteristic polynomial of $S$. Then $S^2 + \alpha S + \beta I$ is also an $\mathfrak{G}$-module map with a  non-trivial kernel. Since $W$ is irreducible and $\mathrm{ker}~( S^2 + \alpha S + \beta I )$ is $\mathfrak{G}$-invariant, $\mathrm{ker}~( S^2 + \alpha S + \beta I ) = W$ and thus $S^2 + \alpha S + \beta I = 0$. Consequently, $x^2 + \alpha x + \beta$ is the minimal polynomial of $S$.

On the other hand, there exists an orthonormal basis in an $S$-invariant subspace of dimension 2 so that the matrix of $S$ with respect to the basis is 
$\big( 
\begin{smallmatrix} 
\cos\theta & -\sin\theta \\ 
\sin\theta & \cos\theta 
\end{smallmatrix} \big).$ 
Therefore, $S$ has the minimal polynomial $x^2 - 2\cos\theta\cdot x + 1$.
\end{proof}

Let $U$ be an $\mathfrak{G}$-invariant subspace. In order to see a geometric feature enjoyed by IRs of $\mathfrak{G}$, we need to focus on a subspace $U[v]$ of $U$, which is comprised of those vectors $\pmb{x}$ in $U$ such that $\xi \hspace{0.4mm} \pmb{x} = \pmb{x}$, $\forall \hspace{0.5mm} \xi\in \mathfrak{G}_v$.

\begin{Lemma}\label{Lemma-IRsFeatureGvNonTrivial}
Let $\mathfrak{G}$ be a permutation group acting on $V$ transitively. Suppose $U$ is an invariant subspace for $\mathfrak{G}$ and the stabilizer $\mathfrak{G}_v$ is not trivial, $v\in V$. Then $U$ is irreducible for $\mathfrak{G}$ if and only if 
$\mathrm{dim}~ U[v] = 1$ for any vertex $v$ of $V$,
{\it i.e.,} the subspace $U[v]$ is irreducible for $\mathfrak{G}_v$.
\end{Lemma}
\begin{proof}[\bf Proof]
The sufficiency of our assertion is quite simple because it is equivalent to that if $U$ is reducible then $\mathrm{dim}~U[v] \geq 2$, $\forall v\in V$. 

Suppose $\oplus_{k=1}^s W_k$ is a decomposition of $U$ into IRs of $\mathfrak{G}$ and $s \ge 2$. Then the subspace spanned by vectors $\{ \dproj{W_k}(\pmb{e}_v) \mid k = 1,\ldots,s \}$ is a subspace of $U[v]$ and of dimension not less than 2 according to our assumption.   

\vspace{3mm}
We now turn to the necessity of the assertion and first list three basic facts about $U[v]$.
\begin{enumerate}
\item $\sigma U[v] = U[\sigma v]$, $\forall \sigma\in\mathfrak{G}$.

\item If $\tau_v$ is a linear map on $U[v]$, then $\tau = \sum_{\sigma\in\mathfrak{G}} \sigma\tau_v \sigma^{-1} \circ \dproj{U[\sigma v]}$ is an $\mathfrak{G}$-module map on $U$.

\item If $\phi$ is an $\mathfrak{G}$-module map on $U$ then $\phi U[v] \subseteq U[v]$.
\end{enumerate}

Let us present a brief explanation for those claims and begin with the 1st one. In accordance with the definition to the subspace $U[v]$, the subspace $U[\sigma v]$ would be $\{ \pmb{y} \in U \mid \zeta\pmb{y} = \pmb{y}, \forall\zeta\in\mathfrak{G}_{\sigma v} \}$. It is readily to see that $\mathfrak{G}_{\sigma v} = \sigma \mathfrak{G}_v \sigma^{-1}$.  Consequently, $\forall\xi\in\mathfrak{G}_v$ and $\pmb{x}\in U[v]$, $(\sigma\xi\sigma^{-1}) \circ (\sigma\pmb{x}) = \sigma\pmb{x}$, and thus $\sigma U[v] \subseteq U[\sigma v]$. 

On the other hand, for any vector $\pmb{y}$ in $U[\sigma v]$ and permutation $\xi$ in $\mathfrak{G}_v$,
$$ (\sigma\xi\sigma^{-1}) \pmb{y} = \pmb{y} \Leftrightarrow 
\xi\sigma^{-1} \pmb{y} = \sigma^{-1} \pmb{y}, $$
which implies that $\sigma^{-1} \pmb{y}$ belongs to $U[v]$, and thus $U[\sigma v] \subseteq \sigma U[v]$.

\vspace{3mm}
Before verifying the 2nd claim, we examine a relevant relation. Suppose $\gamma$ is a permutation in $\mathfrak{G}$. Then
\begin{equation}\label{Equation-GModuleMap}
\dproj{U[\sigma v]}\circ \gamma^{-1} = \gamma^{-1} \circ\dproj{U[\gamma\sigma v]}. 
\end{equation}
We take arbitrarily a vector $\pmb{u}$ from $U$ and express it as 
$$
\pmb{x} + \pmb{y} = 
\dproj{U[\gamma\sigma v]}( \pmb{u} ) + \dproj{U[\gamma\sigma v]^{\perp}}( \pmb{u} ),
$$
where the subspace $U[\gamma\sigma v]^{\perp}$ is the orthogonal complement of $U[\gamma\sigma v]$.  
Clearly, $\gamma^{-1} \circ \dproj{U[\gamma\sigma v]}( \pmb{u} ) = \gamma^{-1} (\pmb{x}),$ which belongs to the subspace $U[\sigma v]$ by virtue of the claim i). On the other hand, 
\begin{align*}
\dproj{U[\sigma v]} \circ \gamma^{-1}( \pmb{u} )
& = \dproj{U[\sigma v]} \circ \gamma^{-1}( \pmb{x} + \pmb{y} ) \\
& = \dproj{U[\sigma v]} \circ \gamma^{-1}( \pmb{x} ) + \dproj{U[\sigma v]} \circ \gamma^{-1}( \pmb{y} ) \\
& =  \gamma^{-1}( \pmb{x} ) + \dproj{U[\sigma v]} \circ \gamma^{-1}( \pmb{y} ). 
\end{align*}
It is easy to see that $ \pmb{y} \perp U[\gamma\sigma v] \Leftrightarrow \gamma^{-1}\pmb{y} \perp U[\sigma v]$, so $\dproj{U[\sigma v]} \circ \gamma^{-1}( \pmb{u} ) = \gamma^{-1} (\pmb{x})$. Therefore, the equation (\ref{Equation-GModuleMap}) holds.

As a result, 
\begin{align*}
\gamma \tau \gamma^{-1} 
& = \gamma\left( \sum_{\sigma\in\mathfrak{G}} \sigma \tau_v \sigma^{-1} \circ \dproj{U[\sigma v]} \right)\gamma^{-1} \\
& = \sum_{\sigma\in\mathfrak{G}} \gamma\sigma \tau_v \sigma^{-1}\gamma^{-1} \circ \dproj{U[\gamma\sigma v]} \\
& = \sum_{\sigma\in\mathfrak{G}} (\gamma\sigma) \tau_v (\gamma\sigma)^{-1} \circ \dproj{U[\gamma\sigma v]} \\
& = \tau,
\end{align*}
so $\tau$ is an $\mathfrak{G}$-module map on $U$.

\vspace{3mm}
Finally, we examine the 3rd claim. Since $\phi$ is an $\mathfrak{G}$-module map on $U$, for any vector $\pmb{x}$ in $U[v]$ and $\xi$ in $\mathfrak{G}_v$,
$$
\xi ( \phi\pmb{x} ) = \xi\circ\phi (\pmb{x}) = \phi\circ\xi (\pmb{x}) = \phi\pmb{x},
$$
so $\phi U[v] \subseteq U[v]$ as claimed. As the matter of fact, one can readily see that if $\phi \neq 0$ then $\phi U[v] = U[v]$.

\vspace{3mm}
We are now ready to turn to the main part and to prove the necessity of our assertion. Suppose $T_1,\ldots,T_t$ are the orbits of the action of $\mathfrak{G}_v$ on $V$ and $\sigma_{ij'} v$ and $\sigma_{ij''} v$ belong to the same orbit of $\mathfrak{G}_v$, $i=1,\ldots,t$ and $j',j'' = 1,\ldots,r_i$. Let us partition the family of cosets $\{ \sigma\mathfrak{G}_v \mid \sigma\in\mathfrak{G} \}$ according to the orbits of $\mathfrak{G}_v$ on $V$ and write $\mathfrak{G}$ as $\cup_{i=1}^t \cup_{j=1}^{r_i} \sigma_{ij}\mathfrak{G}_v$. Then 
$$
\tau = \sum_{\sigma\in\mathfrak{G}} \sigma \tau_v \sigma^{-1} \circ \dproj{U[\sigma v]}
= \sum_{i=1}^t \sum_{j=1}^{r_i} \sum_{\xi\in\mathfrak{G}_v} 
(\sigma_{ij}\xi) \tau_v (\sigma_{ij}\xi)^{-1} \circ \dproj{U[\sigma_{ij} v]}.
$$

We assume for a contradiction that $d = \mathrm{dim}~ U[v] \geq 2$ and the group $\pmb{b}_1,\ldots,\pmb{b}_d$ is an orthonormal basis of $U[v]$. Then one can define a linear map $\tau_v$ on $U[v]$ {\it s.t.,} 
$$
\tau_v : \pmb{b}_s \mapsto 
\begin{cases}
\pmb{b}_1 & \text{ if } s=1, \\
\pmb{0}   & \text{ if } s>1.
\end{cases}
$$

Let us now evaluate $\tau ( \pmb{b}_1 )$ step by step.
\begin{align*}
\sum_{\xi\in\mathfrak{G}_v} 
(\sigma_{ij}\xi) \tau_v (\sigma_{ij}\xi)^{-1} \circ \dproj{U[\sigma_{ij} v]}( \pmb{b}_1 ) 
& = 
\sum_{\xi} \big( \sigma_{ij}\xi \big) \tau_v \big( \xi^{-1} \sigma_{ij}^{-1} \big)
\left( \sum_{s=1}^d \langle \pmb{b}_1, \sigma_{ij}\pmb{b}_s \rangle \sigma_{ij}\pmb{b}_s \right) \\
& = 
\sum_{\xi} \sigma_{ij}\xi \tau_v \xi^{-1}  
\left( \sum_{s=1}^d \langle \pmb{b}_1, \sigma_{ij}\pmb{b}_s \rangle \pmb{b}_s \right) \\
& = 
\sum_{\xi} \sigma_{ij}\xi \tau_v  
\left( \sum_{s=1}^d \langle \pmb{b}_1, \sigma_{ij}\pmb{b}_s \rangle \pmb{b}_s \right) \\
& = 
\sum_{\xi} \sigma_{ij}\xi  
\left( \langle \pmb{b}_1, \sigma_{ij}\pmb{b}_1 \rangle \pmb{b}_1 \right) \\
& = 
|\mathfrak{G}_v| \langle \pmb{b}_1, \sigma_{ij}\pmb{b}_1 \rangle \cdot \sigma_{ij}  \pmb{b}_1. 
\end{align*}
According to our way of organizing cosets of $\mathfrak{G}_v$, $\sigma_{ij} v$ and $\sigma_{i1} v$ belong to the same orbit of $\mathfrak{G}_v$, so for any $j \in [r_i]$, there exists a permutation $\zeta_j$ in $\mathfrak{G}_v$ {\it s.t.,} $\zeta_j \sigma_{i1} v = \sigma_{ij} v$.
Consequently, 
\begin{align*}
\sum_{j=1}^{r_i} \langle \pmb{b}_1,\sigma_{ij}\pmb{b}_1 \rangle \cdot \sigma_{ij}\pmb{b}_1
& =  \sum_{j} \langle \pmb{b}_1,\zeta_j\sigma_{i1}\pmb{b}_1 \rangle \cdot \zeta_j\sigma_{i1}\pmb{b}_1 \\
& = \sum_{j}  \langle \zeta_j^{-1}\pmb{b}_1,\sigma_{i1}\pmb{b}_1 \rangle \cdot \zeta_j\sigma_{i1}\pmb{b}_1 \\
& = \sum_{j}  \langle \pmb{b}_1,\sigma_{i1}\pmb{b}_1 \rangle \cdot \zeta_j\sigma_{i1}\pmb{b}_1 \\
& = \langle \pmb{b}_1,\sigma_{i1}\pmb{b}_1 \rangle \left( \sum_{j}  \zeta_j\right) \sigma_{i1}\pmb{b}_1. 
\end{align*}
In accordance with the definition to the subspace $U[v]$ and the claim iii) above, 
$$
\left( \sum_{j=1}^{r_i}  \zeta_j\right) \sigma_{i1}\pmb{b}_1 = 
\langle \pmb{b}_1,\sigma_{i1}\pmb{b}_1 \rangle \cdot |T_i| \pmb{b}_1,
$$
where $T_i$ is the orbit of $\mathfrak{G}_v$ containing the vertex $\sigma_{i1} v$. As a result, 
$$
\sum_{j=1}^{r_i} \langle \pmb{b}_1,\sigma_{ij}\pmb{b}_1 \rangle \cdot \sigma_{ij}\pmb{b}_1 = 
\langle \pmb{b}_1,\sigma_{i1}\pmb{b}_1 \rangle^2 \cdot |T_i| \pmb{b}_1.
$$
Finally, we see that 
$$
\tau (\pmb{b}_1 ) 
= \sum_{i=1}^t \sum_{j=1}^{r_i} \sum_{\xi\in\mathfrak{G}_v} 
(\sigma_{ij}\xi) \tau_v (\sigma_{ij}\xi)^{-1} \circ \dproj{U[\sigma_{ij} v]}(\pmb{b}_1)
=|\mathfrak{G}_v| \cdot \sum_{i=1}^t \langle \pmb{b}_1,\sigma_{i1}\pmb{b}_1 \rangle^2 \cdot |T_i| \pmb{b}_1.
$$
Note that the subspace $U$ is irreducible, so the $\mathfrak{G}$-module map $\tau$ would be equal to $C\cdot I$ for some constant $C$ according to Lemma \ref{Lem-SchurLemmaOnEuclideanSpace}. It is not difficult however to see $\tau( \pmb{b}_2 ) \neq C\cdot\pmb{b}_2$, which is a contradiction.  
\end{proof}

By virtue of the same idea employed in proving the lemma above, one can easily see the relation below.
\begin{Proposition}\label{Prop-IRsFeatureGvNonTrivial}
Let $\mathfrak{G}$ be a permutation group acting on $V$. Suppose $U$ is an invariant subspace for $\mathfrak{G}$ and the stabilizer $\mathfrak{G}_v$ is not trivial, $v\in V$. Then $U$ is irreducible for $\mathfrak{G}$ if and only if 
\begin{center}
\begin{minipage}{8cm}
$\mathrm{dim}~ U[v] \leq 1$ for any element $v\in V$

and $\mathrm{span}\{ \sigma\pmb{u}_v : \sigma\in\mathfrak{G} \} = U$ if $\mathrm{dim}~ U[v] = 1$, 
\end{minipage}
\end{center}
where the vector $\pmb{u}_v$ belongs to $U[v]$ not equal to $\pmb{0}$.
 \end{Proposition}

Let us now see how to use a similar geometric feature to characterize IRs of $\mathfrak{G}$ in the case that point stabilizers are trivial. Apparently, the subspace $U[v]$ used in Lemma \ref{Lemma-IRsFeatureGvNonTrivial} become pointless now, so we have to introduce the following instead. Let $B$ be one of minimal blocks for $\mathfrak{G}$. Set 
$$
U[B] = \left\{ \pmb{u} \in U : \zeta \left( \sum_{k=1}^{|B|} \xi^k \right) \pmb{u} 
= \left( \sum_{k=1}^{|B|} \xi^k \right) \pmb{u}, 
~~\forall\xi\in \mathfrak{G}_B\setminus\{1\} \mbox{ and } \zeta\in\mathfrak{G}_B  \right\}.
$$
Obviously, $U[B]$ is a subspace of $U$. Moreover, it is an $\mathfrak{G}_B$-invariant subspace. In fact, for any vector $\pmb{u}_B$ in $U[B]$ and permutation $\delta$ in $\mathfrak{G}_B$, 
$$
\sum_{k=1}^{|B|}  (\zeta - I) \xi^k (\delta \pmb{u}_B )
 = \delta \sum_k (\zeta - I) \xi^k ( \pmb{u}_B )
 = \delta (\pmb{0}) = \pmb{0},
$$
for $\mathfrak{G}_B$ is a circulant group due to Lemma \ref{Lem-BlockStabilizerCharacterization-GvTrivial} and thus $\delta \pmb{u}_B$ also belongs to the subspace $U[B]$.

\begin{Lemma}\label{Lemma-IRsFeatureGvTrivial}
Let $\mathfrak{G}$ be a permutation group acting on $V$ transitively. Suppose $U$ is an invariant subspace for $\mathfrak{G}$ and the stabilizer $\mathfrak{G}_v$ is trivial, $v\in V$. Then $U$ is irreducible for $\mathfrak{G}$ if and only if 
the subspace $U[B]$ is irreducible for $\mathfrak{G}_B$, where $B$ is one of minimal blocks for $\mathfrak{G}$, {\it i.e.,}
\begin{equation*}
{\rm dim}~ U[B] = 
\begin{cases}
1 & \text{ if \hspace{0.5mm}} \left( \sum_{k=1}^{|B|} \xi^k \right) \dproj{U}(\pmb{e}_b) \neq\pmb{0} \text{ or } |B| = 2, \\
2 & \text{ if \hspace{0.5mm}} \left( \sum_{k=1}^{|B|} \xi^k \right) \dproj{U}(\pmb{e}_b) = \pmb{0} \text{ and } |B| > 2,
\end{cases}
\end{equation*}
where $b$ is a member of $B$.
\end{Lemma}
\begin{proof}[\bf Proof]
By virtue of essentially the same argument used in proving the sufficiency of the lemma \ref{Lemma-IRsFeatureGvNonTrivial}, one can readily prove the sufficiency of this assertion. As to the necessity, we again start with a list of properties of $U[B]$.
\begin{enumerate}
\item $\sigma U[B] = U[\sigma B]$, $\forall \sigma\in\mathfrak{G}$.

\item If $\tau_{_B}$ is a linear map on $U[B]$, then $\tau = \sum_{\sigma\in\mathfrak{G}} \sigma\tau_{_B} \sigma^{-1} \circ \dproj{U[\sigma B]}$ is an $\mathfrak{G}$-module map on $U$.

\item If $\phi$ is an $\mathfrak{G}$-module map on $U$ then $\phi U[B] \subseteq U[B]$.
\end{enumerate}

Evidently, these properties of $U[B]$ are similar to that of $U[v]$, so essentially one can employ the same ideas in establishing those claims in Lemma \ref{Lemma-IRsFeatureGvNonTrivial} to check what are  listed above. We show the 1st one here as an example. 

In accordance with the definition to the subspace $U[B]$, it is clear that
$$
U[\sigma B] = 
\left\{ \pmb{y} \in U \mid \beta \left( \sum_{k=1}^{|\sigma B|} \alpha^k \right) \pmb{y} 
= \left( \sum_{k=1}^{|\sigma B|} \alpha^k \right) \pmb{y}, 
~~\forall\alpha\in \mathfrak{G}_{\sigma B}\setminus\{1\} \mbox{ and } \beta\in\mathfrak{G}_{\sigma B}  \right\}.
$$ It is readily to see that $\mathfrak{G}_{\sigma B} = \sigma \mathfrak{G}_B \sigma^{-1}$.  Consequently, for any $\xi\in\mathfrak{G}_B\setminus\{1\}$, $\zeta\in\mathfrak{G}_B$ and $\pmb{x}\in U[B]$, 
$$
\sigma\zeta\sigma^{-1} \left( \sum_{k=1}^{|B|} \sigma\xi^k\sigma^{-1} \right) (\sigma\pmb{x})
 = \sigma\circ \zeta \left( \sum_{k} \xi^k \right) \pmb{x}
 = \sigma \left( \sum_{k} \xi^k \right) \pmb{x}
 =        \left( \sum_{k} \sigma\xi^k\sigma^{-1} \right) (\sigma\pmb{x}),
$$ so $\sigma\pmb{x}$ belongs to $U[\sigma B]$ and thus $\sigma U[B] \subseteq U[\sigma B]$. 

On the other hand, for any arbitrarily chosen vector $\pmb{y}$ in $U[\sigma B]$, $\beta \left( \sum_{k=1}^{|\sigma B|} \alpha^k \right) \pmb{y} = \left( \sum_{k=1}^{|\sigma B|} \alpha^k \right) \pmb{y}$, 
$\forall \alpha\in \mathfrak{G}_{\sigma B}\setminus\{1\}$
 and 
$\beta\in\mathfrak{G}_{\sigma B}$. 
Consequently, 
$$ 
\sigma \circ \zeta \left( \sum_{k=1}^{|B|} \xi^k \right) (\sigma^{-1}\pmb{y})
= \sigma\zeta\sigma^{-1} \left( \sum_{k} \sigma\xi^k\sigma^{-1} \right) \pmb{y}
= \left( \sum_{k} \sigma\xi^k\sigma^{-1} \right) \pmb{y}
= \sigma \left( \sum_{k} \xi^k \right) (\sigma^{-1}\pmb{y}),
$$ 
where $\sigma\xi\sigma^{-1}=\alpha$ and $\sigma\zeta\sigma^{-1}=\beta$. 
Hence, $\zeta \left( \sum_{k=1}^{|B|} \xi^k \right) (\sigma^{-1}\pmb{y}) 
= \left( \sum_{k} \xi^k \right) (\sigma^{-1}\pmb{y})$,
which implies that $\sigma^{-1} \pmb{y}$ belongs to $U[B]$, and thus $U[\sigma B] \subseteq \sigma U[B]$.

\vspace{3mm}
Again, we prove the necessity of the assertion by a contradiction and assume that $U[B]$ is not irreducible. Suppose $W_{B,1},\ldots,W_{B,R}$ comprise a group of IRs of $\mathfrak{G}_B$ {\it s.t.,} $U[B] = \oplus_{q=1}^R W_{B,q}$ and $R\geq 2$. Note that $\mathrm{dim}~ W_{B,q} = 1$ or $2$ because $\mathfrak{G}_B$ is a circulant group according to Lemma \ref{Lem-BlockStabilizerCharacterization-GvTrivial}.

We organize the family of cosets $\{ \sigma\mathfrak{G}_B \mid \sigma\in\mathfrak{G} \}$ according to the orbits of the action of $\mathfrak{G}_B$ on $V$. Since every orbit of $\mathfrak{G}_B$ is a block for $\mathfrak{G}$, if there are $t$ orbits of $\mathfrak{G}_B$ then $\mathfrak{G} = \cup_{i=1}^t \sigma_{i}\mathfrak{G}_B$. Accordingly, the $\mathfrak{G}$-module map $\tau = \sum_{\sigma\in\mathfrak{G}} \sigma \tau_{_B} \sigma^{-1} \circ \dproj{U[\sigma B]}$ could be expressed as follows
$$
\tau  
= \sum_{i=1}^t \sum_{\xi\in\mathfrak{G}_B} 
(\sigma_{i}\xi) \tau_{_B} (\sigma_{i}\xi)^{-1} \circ \dproj{U[\sigma_{i} B]}.
$$

Let us define a linear map $\tau_{_B}$ on $U[B]$ {\it s.t.,} 
$$
\tau_{_B} : \pmb{b}_{q,k} \mapsto 
\begin{cases}
\pmb{b}_{1,k} & \text{ if } q=1, \\
\pmb{0}   & \text{ if } q>1,
\end{cases}
$$
where $\pmb{b}_{q,1},\ldots,\pmb{b}_{q,d_q}$ constitute an orthonormal basis of $W_{B,q}$, $q = 1,\ldots,R$. It is easy to see that if there is an irreducible representation, say, without losing any generality, $W_{B,1}$, so that $\exists\pmb{w}_1\in W_{B,1}$, not being trivial, {\it s.t.,} $\xi \pmb{w}_1 = \pmb{w}_1$, $\forall \xi\in\mathfrak{G}_B$, then $W_{B,1} = \mathrm{span} \{ \pmb{w}_1 \}$. Accordingly, one can use the same argument in proving Lemma \ref{Lemma-IRsFeatureGvNonTrivial} to prove this assertion. For that reason, we assume that any IR of $\mathfrak{G}_B$ in $U[B]$ is of dimension 2. As a matter of fact, we can make a further assumption that there are no two IRs of $\mathfrak{G}_B$ in $U[B]$ which are isomorphic to one another with respect to $\mathfrak{G}_B$.

We now begin to evaluate $\tau (\pmb{b}_{1,k})$, $1\le k \le 2$, step by step. First of all, one should note that 
\begin{align*}
\dproj{U[\sigma_{i} B]} (\pmb{b}_{1,k}) 
& = \dproj{\sigma_{i} U[B]}( \pmb{b}_{1,k} ) \\
& = \dproj{\sigma_{i} \oplus_{q=1}^R W_{B,q} }( \pmb{b}_{1,k} ) \\
& = \sum_{q=1}^R \sum_{p=1}^{2} \langle \pmb{b}_{1,k},\sigma_{i}\pmb{b}_{q,p} \rangle 
\cdot \sigma_{i}\pmb{b}_{q,p}.
\end{align*}
Consequently, 
\begin{align*}
\sum_{\xi\in\mathfrak{G}_B} 
(\sigma_{i}\xi) \tau_{_B} (\sigma_{i}\xi)^{-1} \circ \dproj{U[\sigma_{i} B]} (\pmb{b}_{1,k}) 
& = 
\sum_{\xi} \big( \sigma_{i}\xi \big) \tau_{_B}  \big( \xi^{-1} \sigma_{i}^{-1} \big) \left( \sum_{q,p} \langle \pmb{b}_{1,k},\sigma_{i}\pmb{b}_{q,p} \rangle
\cdot \sigma_{i}\pmb{b}_{q,p} \right) \\
& = 
\sum_{\xi} \sigma_{i}\xi \tau_{_B} \xi^{-1}  
\left( \sum_{q,p} \langle \pmb{b}_{1,k},\sigma_{i}\pmb{b}_{q,p} \rangle
\cdot \pmb{b}_{q,p}\right) \\
& = 
\sum_{\xi} \sigma_{i}\xi \tau_{_B}  
\left( \sum_{q,p} \langle \pmb{b}_{1,k},\sigma_{i}\pmb{b}_{q,p} \rangle
\cdot  \xi^{-1}\pmb{b}_{q,p} \right) \\
& = 
\sum_{\xi} \sigma_{i}\xi~  
\left( \sum_{p=1}^{2} \langle \pmb{b}_{1,k},\sigma_{i}\pmb{b}_{1,p} \rangle
\cdot  \xi^{-1}\pmb{b}_{1,p} \right) \\
& = 
\sum_{\xi} \sigma_{i}~  
\left( \sum_{p} \langle \pmb{b}_{1,k},\sigma_{i}\pmb{b}_{1,p} \rangle\cdot \pmb{b}_{1,p} \right) \\
& =   
|\mathfrak{G}_B| \cdot \sum_{p} \langle \pmb{b}_{1,k},\sigma_{i}\pmb{b}_{1,p} \rangle
\cdot \sigma_{i}\pmb{b}_{1,p} \\
& = |\mathfrak{G}_B| \cdot
\dproj{\sigma_{i} W_{B,1}} (\pmb{b}_{1,k}). 
\end{align*}
Consequently,
\begin{align*}
\tau (\pmb{b}_{1,k}) 
&= \sum_{i=1}^{t} \sum_{\xi\in\mathfrak{G}_B} 
(\sigma_{i}\xi) \tau_{_B} (\sigma_{i}\xi)^{-1} \circ \dproj{U[\sigma_{i} B]} (\pmb{b}_{1,k}) \\ 
& = |\mathfrak{G}_B| \cdot \sum_{i}
\dproj{\sigma_{i} W_{B,1}} (\pmb{b}_{1,k}). 
\end{align*}
Since $\dproj{W_{B,1}}(\pmb{b}_{1,k}) = \pmb{b}_{1,k}$, the vector $\sum_{i} \dproj{\sigma_{i} W_{B,1}} (\pmb{b}_{1,k})$ cannot vanish. Therefore $\tau$ is a bijection and $\tau W_{B,q} = W_{B,q}$, $q=1,\ldots,R$ because $W_{B,q'} \ncong W_{B,q''}$, $1 \le q' \leq q'' \leq R$. However, one can readily verify that $\tau (\pmb{b}_{2,k})$, $1\leq k \leq 2$, does not belong to $W_{B,2}$, which is a contradiction.
\end{proof}

\begin{Proposition}\label{Prop-IRsFeatureGvTrivial}
Let $\mathfrak{G}$ be a permutation group acting on $V$. Suppose $U$ is an invariant subspace for $\mathfrak{G}$ and the stabilizer $\mathfrak{G}_v$ is trivial, $v\in V$. Then $U$ is irreducible for $\mathfrak{G}$ if and only if 
the subspace $U[B]$ is irreducible for $\mathfrak{G}_B$ and $\mathrm{span}\{ \sigma\pmb{u}_B \mid \sigma\in\mathfrak{G} \} = U$ if $\mathrm{dim}~ U[B] > 0$,
where $B$ is one of minimal blocks for $\mathfrak{G}$ and the vector $\pmb{u}_B$ belongs to $U[B]$ not equal to $\pmb{0}$.  \end{Proposition}

Let $\Pi$ be an equitable partition of $G$. As we have noted in the 1st section, there is a close relation between the eigenvectors of $\mathbf{A}(G)$ and that of $\mathbf{A}(G / \Pi)$. To be precise, if $\pmb{x}_{\lambda}$ is an eigenvector of $\mathbf{A}(G / \Pi)$, corresponding to the eigenvalue $\lambda$, then $\mathbf{R} \pmb{x}_{\lambda}$ is an eigenvector of $\mathbf{A}(G)$, corresponding to $\lambda$ also, where $\mathbf{R}$ is the characteristic matrix of $\Pi$. Accordingly, we say that the eigenvector $\pmb{x}_{\lambda}$ of $\mathbf{A}(G / \Pi)$ {\it ``lifts''} to an eigenvector of $\mathbf{A}(G)$. Moreover 
all eigenvectors of $\mathbf{A}(G)$ could be divided into two classes: those that are constant on every cell of $\Pi$ and those that sum to zero on each cell of $\Pi$, and the first class consists of vectors lifted from eigenvectors of $\mathbf{A}(G / \Pi)$.

In other words, if $\Pi$ possesses $t$ cells and $x$ and $y$ are two vertices of $G$  belonging to the same cell of $\Pi$, then 
$$
\langle \pmb{e}_x,\dproj{V_\lambda}(\pmb{R}_j) \rangle = \langle \pmb{e}_y,\dproj{V_\lambda}(\pmb{R}_j) \rangle, ~\forall~\lambda\in\mathrm{spec}~\mathbf{A}(G) \mbox{\it ~and } j\in [t], 
$$
where $\pmb{R}_j$ is the characteristic vector of the $j$th cell of $\Pi$. As we shall see below, the relation above is also sufficient for being equitable.

\begin{Lemma}\label{Lemma-EquitablePartProj}
Let $\Pi$ be a partition of $V$ with $t$ cells. Then $\Pi$ is equitable if and only if for any two elements $x$ and $y$ belonging to the same cell of $\Pi$, $
\langle \pmb{e}_x,\dproj{V_\lambda}(\pmb{R}_j) \rangle = \langle \pmb{e}_y,\dproj{V_\lambda}(\pmb{R}_j) \rangle, ~\forall~\lambda\in\mathrm{spec}~\mathbf{A}(G) \mbox{\it ~and } j\in [t].
$ 
\end{Lemma}
\begin{proof}[\bf Proof]
We have discussed the necessity of our assertion, so let us show the sufficiency now. Obviously, the vectors $\pmb{R}_1,\ldots,\pmb{R}_t$ comprise an orthogonal basis of $U_\Pi$, the column space of $\mathbf{R}$. To prove $U_\Pi$ is $\mathbf{A}(G)$-invariant, it suffices to show that $\mathbf{A}(G) \pmb{R}_k$, $1\leq k \leq t$, can be written as a linear combination of $\pmb{R}_1,\ldots,\pmb{R}_t$. 

In fact, 
\begin{align*}
\mathbf{A}(G) \pmb{R}_k 
& = \mathbf{A}(G) \left( \sum_{\lambda \hspace{0.2mm} \in \hspace{0.2mm} \mathrm{spec} \hspace{0.2mm} \mathbf{A}(G)} \dproj{V_{\lambda,G}}( \pmb{R}_k ) \right) \\
& = \sum_{\lambda} \mathbf{A}(G)\dproj{V_{\lambda,G}}( \pmb{R}_k ) \\
& = \sum_{\lambda} \lambda\cdot\dproj{V_{\lambda,G}}( \pmb{R}_k ).
\end{align*}
In accordance with our assumption, one can readily see that $\dproj{V_{\lambda,G}}( \pmb{R}_k )$ can be represented as a linear combination of $\pmb{R}_1,\ldots,\pmb{R}_t$, so does $\mathbf{A}(G) \pmb{R}_k$.
\end{proof}

It is clear that by means of the action of $\mathfrak{G}$ on $V$, we can naturally obtain an action of $\mathfrak{G}$ on a partition $\Pi$. To be precise, if $\sigma$ is a permutation of $\mathfrak{G}$ then $\sigma\hspace{0.4mm}\Pi = \{ \sigma\hspace{0.4mm} C \mid C $ is a cell of $\Pi \}$. Obviously we can consider the same action of $\mathfrak{H}$ on $\Pi$ for any subgroup $\mathfrak{H}$ of $\mathfrak{G}$.

One can readily see that if each cell of $\Pi$ is comprised of the union of some of orbits of $\mathfrak{G}$ or $\Pi$ is actually a block system of $\mathfrak{G}$, then $\sigma \hspace{0.4mm} \Pi = \Pi$. It is interesting that we can characterize that kind of partitions in virtue of eigenspaces of $\mathbf{A}(G/\Pi)$ provided that the partition $\Pi$ concerned is equitable.

\begin{Lemma}\label{Lemma-EquitablePartH-Invariant}
Let $\Pi$ be an equitable partition of a graph $G$, and let $\mathfrak{H}$ be a subgroup of the automorphism group $\mathfrak{G}$ of $G$. Then $\xi \hspace{0.4mm} \Pi = \Pi$, $\forall \xi \in \mathfrak{H}$ if and only if the subspace $\mathbf{R} V_{\lambda}^{G/\Pi}$ is $\mathfrak{H}$-invariant, $\forall \lambda \in \mathrm{spec} \hspace{0.4mm} \mathbf{A}(G/\Pi)$, where $\mathbf{R}$ is the characteristic matrix of $\Pi$ and $V_{\lambda}^{G/\Pi}$ is the eigenspace of $\mathbf{A}(G/\Pi)$ corresponding to $\lambda$. 
\end{Lemma}
\begin{proof}[\bf Proof]
Let us start with the necessity of our assertion. The key observation is that 
\begin{equation}\label{Equation-X} 
\xi \hspace{0.4mm} \Pi = \Pi  \Leftrightarrow \exists \hspace{0.5mm} \hat{\xi} \in \mathrm{Sym} \hspace{0.4mm} V(G/\Pi) ~ s.t., ~ \mathbf{P}_\xi \mathbf{R} = \mathbf{R} \mathbf{P}_{\hat{\xi}},
\end{equation}
where $\mathbf{P}_\xi$ and $\mathbf{P}_{\hat{\xi}}$ are the permutation matrices corresponding to $\xi$ and $\hat{\xi}$ respectively.

Accordingly, if for any permutation $\xi$ in $\mathfrak{H}$ $\xi \hspace{0.4mm} \Pi = \Pi$, then
$$
\mathbf{R}\mathbf{A}(G/\Pi)  \mathbf{P}_{\hat{\xi}} 
= \mathbf{A}(G) \mathbf{R}\mathbf{P}_{\hat{\xi}} 
= \mathbf{A}(G)\mathbf{P}_\xi\mathbf{R}
= \mathbf{P}_\xi\mathbf{A}(G)\mathbf{R} 
= \mathbf{P}_\xi\mathbf{R} \mathbf{A}(G/\Pi)
= \mathbf{R}\mathbf{P}_{\hat{\xi}}\mathbf{A}(G/\Pi).
$$
Consequently, 
$$
(\mathbf{R}^T\mathbf{R}) (\mathbf{A}(G/\Pi)\mathbf{P}_{\hat{\xi}})
= (\mathbf{R}^T\mathbf{R}) (\mathbf{P}_{\hat{\xi}}\mathbf{A}(G/\Pi)).
$$
Evidently, $\mathbf{R}^T\mathbf{R}$ is a diagonal matrix with $\left( \mathbf{R}^T\mathbf{R} \right)_{ii} = | C_i |$, the order of the $i$th cell of $\Pi$, so it is an invertible matrix. As a result, 
$\mathbf{A}(G/\Pi)\mathbf{P}_{\hat{\xi}} = \mathbf{P}_{\hat{\xi}}\mathbf{A}(G/\Pi)$, which means $\hat{\xi}$ is an automorphism of the graph $G/\Pi$, so every eigenspace of $\mathbf{A}(G/\Pi)$ is invariant for the permutation $\hat{\xi}$ due to Lemma \ref{Lem-AutomorphismAndOperator}.

\vspace{3mm}
We now show the sufficiency. By means of the relation (\ref{Equation-X}), it suffices to show that $\forall \hspace{0.4mm} \xi \in \mathfrak{G}$, $\exists \hspace{0.5mm} \hat{\xi} \in \mathrm{Sym} \hspace{0.4mm} V(G/\Pi) ~ s.t., ~ \mathbf{P}_\xi \mathbf{R} = \mathbf{R} \mathbf{P}_{\hat{\xi}}$. Suppose for a contradiction that there exists a permutation $\zeta$ in $\mathfrak{H}$ such that $\mathbf{P}_\zeta \mathbf{R} \neq \mathbf{R} \mathbf{P}_{\hat{\zeta}}$, $\forall\hspace{0.5mm}\hat{\zeta} \in \mathrm{Sym} \hspace{0.4mm} V(G/\Pi)$. Then one can readily see that there would exist a vertex $\hat{v}$ of $V(G/\Pi)$ and an eigenvalue $\lambda$ of $\mathrm{spec} \hspace{0.5mm} \mathbf{A}(G/\Pi)$ so that
$$
\zeta\hspace{0.5mm} \mathbf{R}\dproj{V_{\lambda}^{G/\Pi}}( \pmb{e}_{\hat{v}} ) \notin \mathbf{R}V_{\lambda}^{G/\Pi},
$$
which is in contradiction with the assumption that $\mathbf{R} V_{\lambda}^{G/\Pi}$ is $\mathfrak{H}$-invariant.
\end{proof}

Now let us see how blocks for $\mathfrak{G}$ are represented by IRs of $\mathfrak{G}$.

\begin{Lemma}\label{Lemma-EquitablePartProjOntoIRs}
Let $G$ be a vertex-transitive graph and let $B$ be a block for $\mathfrak{G}$.  Suppose $V_\lambda$ is an eigenspace of $\mathbf{A}(G)$ and $\oplus_p W_{\lambda,p}$ is a decomposition of $V_\lambda$ into IRs of $\mathfrak{G}$. If $\pmb{R}_B$ be the characteristic vector of the block $B$, then 
\begin{equation*}
\dproj{W_{\lambda,p}}(\pmb{R}_B) = 
\begin{cases}
|B|\cdot\dproj{W_{\lambda,p}}(\pmb{e}_b) & \text{if \hspace{0.4mm}} \forall b',b'' \in B, ~\dproj{W_{\lambda,p}}(\pmb{e}_{b'}) = \dproj{W_{\lambda,p}}(\pmb{e}_{b''}), \\
~\pmb{0} & \text{otherwise,}
\end{cases}
\end{equation*}
where $b$ is an element of $B$.
\end{Lemma}
\begin{proof}[\bf Proof]
Apparently, 
\begin{equation}\label{Eq-ProjectionsOfABlock-TransitiveCase} 
\dproj{W_{\lambda,p}}(\pmb{R}_B) 
= \dproj{W_{\lambda,p}}\left( \sum_{b\in B} \pmb{e}_b \right) 
= \sum_b \dproj{W_{\lambda,p}}( \pmb{e}_b ).
\end{equation}
Consequently, if all projections of $\pmb{e}_b s$ are the same then $\dproj{W_{\lambda,p}}(\pmb{R}_B) = |B|\cdot\dproj{W_{\lambda,p}}(\pmb{e}_b)$. As a result, let us assume that there are two projections $\dproj{W_{\lambda,p}}(\pmb{e}_{x})$ and $\dproj{W_{\lambda,p}}(\pmb{e}_{y})$ not equal to one another, where $x$ and $y$ are two members of $B$.

We begin with the case that the stabilizer $\mathfrak{G}_v$ is not trivial where $v$ is a vertex of $G$. Note that $B$ is a block, so $\mathfrak{G}_b$ is a subgroup of $\mathfrak{G}_B$ for any member $b$ in $B$ according to Lemma \ref{LemBlock}. Consequently, for any permutation $\xi \in \mathfrak{G}_b$, 
\begin{equation}\label{Eq-X}
\xi\circ\dproj{W_{\lambda,p}}(\pmb{R}_B) = \dproj{W_{\lambda,p}}\circ\xi( \pmb{R}_B ) = \dproj{W_{\lambda,p}}(\pmb{R}_B). 
\end{equation}
Set $W_{\lambda,p}[v] = \{ \pmb{w} \in W_{\lambda,q} \mid \xi \hspace{0.3mm} \pmb{w} = \pmb{w}, \forall \xi\in\mathfrak{G}_v \}$. The equation (\ref{Eq-X}) shows us that $\dproj{W_{\lambda,p}}(\pmb{R}_B) \in \cap_{b\in B} W_{\lambda,p}[b]$.

According to Lemma \ref{Lemma-IRsFeatureGvNonTrivial}, $\mathrm{dim} \hspace{0.4mm} W_{\lambda,p}[b] =  1$, $\forall b\in B$ , so there are only two possibilities: 
$$
W_{\lambda,p}[b'] \cap W_{\lambda,p}[b''] = \pmb{0} \mbox{ or } W_{\lambda,p}[b'] = W_{\lambda,p}[b''],
$$ for any two members $b'$ and $b''$ in $B$. 
Evidently, $\dproj{W_{\lambda,p}}(\pmb{R}_B) = \pmb{0}$ in the first case and in the second case $\dproj{W_{\lambda,p}}(\pmb{e}_{x}) = - \dproj{W_{\lambda,p}}(\pmb{e}_{y})$. 

It is not difficult to see that in the second case, $|B|$ is an even number and exactly half of the members in $B$ possess a positive projection onto $W_{\lambda,p}[b]$, and thus $\dproj{W_{\lambda,p}}(\pmb{R}_B) = \pmb{0}$. In fact, if it is not the case, then there is, due to Equation (\ref{Eq-ProjectionsOfABlock-TransitiveCase}), a natural number $k$ so that $\dproj{W_{\lambda,p}}(\pmb{R}_B) = k\cdot\dproj{W_{\lambda,p}}(\pmb{e}_{b^*})$ for some $b^*$ in $B$. Consequently, for any $\gamma$ in $\mathfrak{G}_B$,
\begin{align*}
\gamma\circ\dproj{W_{\lambda,p}}(\pmb{e}_{b^*})
& = \gamma \left( \frac{1}{k} \cdot \dproj{W_{\lambda,p}}(\pmb{R}_B) \right) \\
& = \frac{1}{k} \cdot \gamma\circ\dproj{W_{\lambda,p}}(\pmb{R}_B) \\
& = \frac{1}{k} \cdot \dproj{W_{\lambda,p}}\circ\gamma(\pmb{R}_B) \\
& = \dproj{W_{\lambda,p}}(\pmb{e}_{b^*}),
\end{align*}
which is in contradiction with the fact that $B$ is an orbit of $\mathfrak{G}_B$.

\vspace{3mm}
We now turn to the case that the stabilizer $\mathfrak{G}_v$ is trivial and first introduce a subspace of $W_{\lambda,p}$ similar to $W_{\lambda,p}[v]$: 
$$
W_{\lambda,p}[B] = \left\{ \pmb{w} \in W_{\lambda,p} \mid \zeta \left( \sum_{k=1}^{|B|} \xi^k \right) \pmb{w} 
= \left( \sum_{k=1}^{|B|} \xi^k \right) \pmb{w}, 
~~\forall\xi\in \mathfrak{G}_B\setminus\{1\} \mbox{ and } \zeta\in\mathfrak{G}_B  \right\}.
$$

According to Lemma \ref{Lemma-IRsFeatureGvTrivial}, $\mathrm{dim} \hspace{0.4mm} W_{\lambda,p}[B] = 1$ or $2$. Since $\mathfrak{G}_B$ is, due to Lemma \ref{Lem-BlockStabilizerCharacterization-GvTrivial}, a circulant group of prime order, if $\mathrm{dim} \hspace{0.4mm} W_{\lambda,p}[B] = 1$ then $|B| = 2$ and $\dproj{W_{\lambda,p}}(\pmb{e}_{x}) = -\dproj{W_{\lambda,p}}(\pmb{e}_{y})$. Hence $\dproj{W_{\lambda,p}}(\pmb{R}_B) = \pmb{0}$. If $\mathrm{dim} \hspace{0.4mm} W_{\lambda,p}[B] = 2$ then $\sum_{b \in B} \dproj{W_{\lambda,p}}( \pmb{e}_b ) = \sum_{s=1}^{|B|} \xi^s \circ \dproj{W_{\lambda,p}}( \pmb{e}_b ) = \pmb{0}$ in accordance with Lemma \ref{Lemma-IRsFeatureGvTrivial}, which also implies that $\dproj{W_{\lambda,p}}(\pmb{R}_B) = \pmb{0}$. 
\end{proof}

We are now ready to present a characterization of blocks through IRs of $\mathfrak{G}$.

\begin{Lemma}\label{Lemma-BlockCHARACTERISEDbyProj}
Let $\mathfrak{G}$ be a permutation group acting on $V$ transitively and let $S$ be a subset of $V$. Set $\mathscr{W}_S = \{ W \mbox{ is an irreducible representation of } \mathfrak{G} : \dproj{W}(\pmb{R}_S) \neq \pmb{0} \}$, where $\pmb{R}_S$ is the characteristic vector of $S$. Then $S$ is a block for $\mathfrak{G}$ if and only if 
\begin{center}
\begin{minipage}{12.5cm}
$\dproj{W}(\pmb{R}_S) = |S|\cdot\dproj{W}(\pmb{e}_s)$, $\forall~ W\in\mathscr{W}_S$ and $s \in S$, 

and ~$\forall~ \sigma\in \mathfrak{G}\setminus\mathfrak{G}_S$, $\exists W_\sigma \in\mathscr{W}_S$, {\it s.t.,} $\dproj{W_\sigma}(\pmb{R}_S) \neq \dproj{W_\sigma}(\pmb{R}_{\sigma S})$.
\end{minipage}
\end{center}
\end{Lemma}
\begin{proof}[\bf Proof]
Note that we have proven the 1st part of the necessity in Lemma \ref{Lemma-EquitablePartProjOntoIRs}, so let focus on the 2nd part. 

Since $S$ is a block for $\mathfrak{G}$, $\left\langle \pmb{R}_S,\pmb{R}_{\sigma S} \right\rangle = 0$, $\forall \sigma\in\mathfrak{G}\setminus\mathfrak{G}_B$. On the other hand, 
\begin{align*}
\left\langle \pmb{R}_S,\pmb{R}_{\sigma S} \right\rangle
& = \left\langle \sum_{\lambda \in \mathrm{spec} \hspace{0.3mm} \mathbf{A}(G)} \dproj{V_{\lambda}}( \pmb{R}_S ),\sum_{\lambda \in \mathrm{spec} \hspace{0.3mm} \mathbf{A}(G)} \dproj{V_{\lambda}}( \pmb{R}_{\sigma S} ) \right\rangle \\
& = \sum_{\lambda} \left\langle \dproj{V_{\lambda}}( \pmb{R}_S ),\dproj{V_{\lambda}}( \pmb{R}_{\sigma S} ) \right\rangle \\
& = \sum_{\lambda} \left\langle \sum_{p} \dproj{W_{\lambda,p}}( \pmb{R}_S ),\sum_{p} \dproj{W_{\lambda,p}}( \pmb{R}_{\sigma S} ) \right\rangle \\
& = \sum_{\lambda,p} \left\langle \dproj{W_{\lambda,p}}( \pmb{R}_S ),\dproj{W_{\lambda,p}}( \pmb{R}_{\sigma S} ) \right\rangle, 
\end{align*}
where $\oplus_{p} W_{\lambda,p}$ is a decomposition of $V_{\lambda}$ into IRs of $\mathfrak{G}$. Because $S$ is a block, the last sum is actually equal to 
$$
\sum_{W \in \mathscr{W}_S} \left\langle \dproj{W}( \pmb{R}_S ),\dproj{W}( \pmb{R}_{\sigma S} ) \right\rangle.
$$

Accordingly, if $\dproj{W}( \pmb{R}_S ) = \dproj{W}( \pmb{R}_{\sigma S} )$, $\forall W \in \mathscr{W}_S$, then 
$$
\sum_{W \in \mathscr{W}_S} \left\langle \dproj{W}( \pmb{R}_S ),\dproj{W}( \pmb{R}_{\sigma S} ) \right\rangle = \sum_{W} \left\| \dproj{W}( \pmb{R}_S ) \right\|^2 > 0,
$$
which is in contradiction with the fact that $\left\langle \pmb{R}_S,\pmb{R}_{\sigma S} \right\rangle = 0$. 

\vspace{3mm}
We now turn to the sufficiency of the assertion and assume for a contradiction that $S$ is not a block for $\mathfrak{G}$. Then there exists a permutation $\gamma$ in $\mathfrak{G}\setminus\mathfrak{G}_S$ so that $\gamma S \cap S \neq \emptyset$ and $\gamma S \neq S$.

Suppose $x$ is one of elements in $S$ which also belongs to $\gamma S$. Then $\gamma^{-1} x$ is an element of $S$ as well. In fact, 
$$
x \in \gamma S ~\Rightarrow~ \exists y_x \in S : x = \gamma y_x ~\Rightarrow~ \gamma^{-1} x = y_x.
$$

On the one hand, 
$$
\gamma \notin \mathfrak{G}_S ~\Rightarrow~ \exists W_{\gamma} \in \mathscr{W}_S : 
\gamma \hspace{0.4mm} \dproj{W_{\gamma}}( \pmb{R}_S ) \neq \dproj{W_{\gamma}}( \pmb{R}_S ).
$$
On the other hand, 
\begin{align*}
\gamma \hspace{0.4mm} \dproj{W_{\gamma}}( \pmb{R}_S )
& = \gamma \left( |S|\cdot \dproj{W_{\gamma}}( \pmb{e}_{\gamma^{-1} x} )  \right) \\
& = |S|\cdot \dproj{W_{\gamma}}\circ\gamma( \pmb{e}_{\gamma^{-1} x} ) \\
& = |S|\cdot \dproj{W_{\gamma}}( \pmb{e}_{x} ) \\
& = \dproj{W_{\gamma}}( \pmb{R}_S ).
\end{align*}
That is a contradiction.
\end{proof}

We have seen how blocks for $\mathfrak{G}$ are represented by IRs of $\mathfrak{G}$, so now let us see how stabilizers of blocks are represented by IRs. More precisely, we shall show how to determine the orbits of $\mathfrak{G}_B$ by means of IRs of $\mathfrak{G}$, where $B$ is a block for the group. Apparently, the orbits of $\mathfrak{G}_B$ is an equitable partition, which is denoted by $\Pi_B^*$.

Set $\mathscr{W}_B = \{ W \mbox{ is an irreducible representation of }\mathfrak{G} : \dproj{W}( \pmb{R}_B ) \neq \pmb{0} \}.$ Two vertices $x$ and $y$ in $V$ is said to be {\it related} if 
$$
\langle \pmb{e}_x,\dproj{W}(\pmb{R}_B) \rangle = \langle \pmb{e}_y,\dproj{W}(\pmb{R}_B) \rangle, ~ \forall~W\in\mathscr{W}_{B}.
$$ 
One can easily see that the relation is an equivalence relation, so it induces a partition of $V$ denoted by $\Pi_B^{\mathrm{IRP}}$. 
 
\begin{Theorem}\label{Thm-PartProjEigSpaceVsPartOrbitsGv}
Let $G$ be a vertex-transitive graph and let $\mathfrak{G}$ be the automorphism group of $G$. Suppose $B$ is a block for $\mathfrak{G}$ such that  the stabilizer $\mathfrak{G}_B$ is not trivial. Then 
$$
\Pi_B^{\mathrm{IRP}} = \Pi_B^*. 
$$
\end{Theorem}

Before proving the assertion above, we establish an auxiliary result revealing an interesting relation. 

\begin{Lemma}\label{Lemma-EnumerationToPi}
Let $G$ be a vertex-transitive graph and let $\mathfrak{G}$ be the automorphism group of $G$. Suppose $B$ is a block for $\mathfrak{G}$. Then
$$
| \mathscr{W}_B | = | \Pi_B^* |,
$$
{\it i.e.,} the number of orbits of the action of $\mathfrak{G}_B$ on $V$ is equal to the order of the family $\mathscr{W}_B$.
\end{Lemma}
\begin{proof}[\bf Proof]
First of all, one can easily check that the partition $\Pi_B^*$ is an equitable partition of $G$, so we can build a quotient graph $G / \Pi_B^*$. Apparently, $| \Pi_B^* | = \dim \R^{ | G / \Pi_B^* | }$.

Next, we introduce a subspace $\{ \pmb{v}\in\R^n \mid \xi \hspace{0.4mm} \pmb{v} = \pmb{v}, \forall \xi\in\mathfrak{G}_B \}$ of $\R^n$ related to IRs contained in $\mathscr{W}_B$, which is denoted by $\R^n[B]$. Let $\mathbf{R}_{\Pi_B^*}$ be the characteristic matrix of the partition $\Pi_B^*$ and let $\xi$ be a permutation in $\mathfrak{G}_B$. One can readily see that 
$\xi \left( \mathbf{R}_{\Pi_B^*} \pmb{x}_{_Q} \right) = \mathbf{R}_{\Pi_B^*} \pmb{x}_{_Q}$ for any vector $\pmb{x}_{_Q}$ in $\R^{ | G / \Pi_B^* | }$. Hence, $\mathbf{R}_{\Pi_B^*} \pmb{x}_{_Q}$ belongs to $\R^n[B]$. On the other hand, a moment's reflection shows the relation below.
\begin{Claim}
A group of vectors $\pmb{x}_{_Q,_1},\ldots,\pmb{x}_{_Q,_l}$ in $\R^{ | G / \Pi_B^* | }$ are linearly dependent if and only if $\mathbf{R}_{\Pi_B^*}\pmb{x}_{_Q,_1},\ldots,$ $\mathbf{R}_{\Pi_B^*}\pmb{x}_{_Q,_l}$ are linearly dependent.
\end{Claim}

\noindent Accordingly, $\dim \R^{ | G / \Pi_B^* | } \leq \dim \R^n[B]$.

Now let us show that $\dim \R^n[B] \leq | \mathscr{W}_B |$. Suppose $\oplus_{\lambda \hspace{0.3mm} \in \hspace{0.3mm} \mathrm{spec} \hspace{0.3mm} \mathbf{A}(G)} \oplus_{q=1}^{r_{_\lambda}} W_{\lambda,q}$ is a decomposition of $\R^n$ into IRs of $\mathfrak{G}$. Then for any vector $\pmb{v}$ of $\R^n[B]$, $\pmb{v} = \sum_{\lambda,q} \dproj{W_{\lambda,q}}( \pmb{v} )$. Because each irreducible representation $W_{\lambda,q}$ is $\mathfrak{G}$-invariant, $\xi \circ \dproj{W_{\lambda,q}}( \pmb{v} ) = \dproj{W_{\lambda,q}}\circ\xi( \pmb{v} ) = \dproj{W_{\lambda,q}}( \pmb{v} )$, so the vector $\dproj{W_{\lambda,q}}( \pmb{v} )$ belongs to the subspace $\{ \pmb{w} \in W_{\lambda,q} \mid \xi \hspace{0.4mm} \pmb{w} = \pmb{w}, \forall\xi\in\mathfrak{G}_B \}$, which is denoted by $W_{\lambda,q}[B]$. Hence 
$$
\R^n[B] = \mathrm{span} \left\{ \pmb{x} \mbox{ is an eigenvector of }\mathbf{A}(G) \mid \pmb{x} \in 
W_{\lambda,q}[B], \lambda \in \mathrm{spec} \hspace{0.3mm} \mathbf{A}(G) \mbox{ and } q = 1,\ldots,r_{_\lambda} 
\right\}.
$$
According to fundamental lemmas \ref{Lemma-IRsFeatureGvNonTrivial} and \ref{Lemma-IRsFeatureGvTrivial} and the lemma \ref{Lemma-EquitablePartProjOntoIRs} concerning the feature of projections of blocks, $\dim W_{\lambda,q}[B] = 1$ provided that $W_{\lambda,q} \in \mathscr{W}_B$ and $\dim W_{\lambda,q}[B] = 0$ otherwise. Therefore $\dim \R^n[B] \leq | \mathscr{W}_B |$.

Finally, we show that $| \mathscr{W}_B | \leq \dim \R^{ | G / \Pi_B^* | }$. Apparently, $\xi \hspace{0.4mm} \dproj{W_{\lambda,q}}( \pmb{R}_B ) = \dproj{W_{\lambda,q}}( \pmb{R}_B )$, $\forall \xi\in\mathfrak{G}_B$, so for any irreducible representation $W_{\lambda,q}$ in $\mathscr{W}_B$ we can construct out of $\dproj{W_{\lambda,q}}( \pmb{R}_B )$ exactly one vector $\pmb{w}_{_\lambda,_q }$ in $\R^{ | G / \Pi_B^* | }$ so that $\dproj{W_{\lambda,q}}( \pmb{R}_B ) = \mathbf{R}_{\Pi_B^*} \pmb{w}_{_\lambda,_q }$. According to the claim above, the group $\{ \pmb{w}_{_\lambda,_q } \mid  \mathbf{R}_{\Pi_B^*} \pmb{w}_{_\lambda,_q } = \dproj{W_{\lambda,q}}( \pmb{R}_B ), W_{\lambda,q}\in\mathscr{W}_B \}$ consists of orthogonal vectors in the space $\R^{ | G / \Pi_B^* | }$, and thus the inequality follows.
\end{proof}

\begin{proof}[\bf Proof to Theorem \ref{Thm-PartProjEigSpaceVsPartOrbitsGv}]
It is clear that $\Pi_B^*$ is a refinement of $\Pi_B^{\mathrm{IRP}}$, so if for any two orbits $T'$ and $T''$ of $\mathfrak{G}_B$ there exists an irreducible representation $W$ in $\mathscr{W}_B$ such that $\langle \pmb{e}_{t'},\dproj{W}(\pmb{R}_B) \rangle\neq \langle \pmb{e}_{t''},\dproj{W}(\pmb{R}_B) \rangle$, where $t' \in T'$ and $t'' \in T''$, then our assertion follows.

As we have seen, if $W$ is one member of $\mathscr{W}_B$ then $\xi \hspace{0.4mm} \dproj{W}( \pmb{R}_B ) = \dproj{W}( \pmb{R}_B )$, $\forall \xi\in\mathfrak{G}_B$, so for any irreducible representation $W$ in $\mathscr{W}_B$ we can construct out of $\dproj{W}( \pmb{R}_B )$ exactly one vector $\pmb{x}_{_W}$ in $\R^{ | G / \Pi_B^* | }$ such that $\dproj{W}( \pmb{R}_B ) = \mathbf{R}_{\Pi_B^*} \pmb{x}_{_W}$. In accordance with Lemma \ref{Lemma-EnumerationToPi}, $| \mathscr{W}_B | = | \Pi_B^* |$, so $| \mathscr{W}_B | = \dim \R^{ | G/ \Pi_B^* | }$. Hence the group of vectors $\{ \pmb{x}_{_W} \in \R^{ | G/ \Pi_B^* | } : \mathbf{R}_{\Pi_B^*} \pmb{x}_{_W} = \dproj{W}( \pmb{R}_B ), W \in \mathscr{W}_B \}$ constitute a  basis of $\R^{ | G/ \Pi_B^* | }$, for the group  $\{ \dproj{W}( \pmb{R}_B ) : W \in \mathscr{W}_B \}$ is comprised of vectors orthogonal to one another.

Suppose $\pmb{e}_{T'}$ and $\pmb{e}_{T''}$ are characteristic vectors of $T'$ and $T''$ in the space $\R^{ | G/ \Pi_B^* | }$, respectively. Then $\langle \pmb{e}_{T'},\pmb{e}_{T''} \rangle = 0$. Accordingly, there exists an irreducible representation $W$ in $\mathscr{W}_B$ such that $\langle \pmb{e}_{t'},\dproj{W}(\pmb{R}_B) \rangle\neq \langle \pmb{e}_{t''},\dproj{W}(\pmb{R}_B) \rangle$.
\end{proof}

It is not difficult to see that by replacing lemmas \ref{Lemma-IRsFeatureGvNonTrivial} and \ref{Lemma-IRsFeatureGvTrivial} with propositions \ref{Prop-IRsFeatureGvNonTrivial} and \ref{Prop-IRsFeatureGvTrivial} we  can prove general versions for Lemma \ref{Lemma-EquitablePartProjOntoIRs}, \ref{Lemma-BlockCHARACTERISEDbyProj}, \ref{Lemma-EnumerationToPi} and Theorem \ref{Thm-PartProjEigSpaceVsPartOrbitsGv} without the restriction that the action of $\mathfrak{G}$ is transitive.

What we have established focuses on how IRs of $\mathfrak{G}$ represent blocks and block systems contained in one orbit of $\mathfrak{G}$. Naturally we should investigate how blocks in distinct orbits are connected to each other. The relation is as one may expect really intriguing, for a block system in one orbit may not be represented in another orbit. 

\begin{Lemma}\label{Lem-ConnectionBlockSystemInDistinctOrbits} 
Let $\mathfrak{G}$ be a permutation group acting on $V$. Suppose $B$ is a block for $\mathfrak{G}$ contained in one orbit $T'$ of $\mathfrak{G}$ and $S$ is an orbit of $\mathfrak{G}_B$ in another orbit $T''$ of $\mathfrak{G}$. Then the subset $B^+ := \{ \alpha b \mid \alpha \in \mathfrak{G}_S \}$ is a block for $\mathfrak{G}$, $b\in B$, and $\mathfrak{G}_S = \mathfrak{G}_{B^+}$.
\end{Lemma}
\begin{proof}[\bf Proof]
Evidently, there are two possible cases: $B^+ = B$ or $B^+ \supsetneq B$. Since $S$ is an orbit of $\mathfrak{G}_B$, $\mathfrak{G}_B \leq \mathfrak{G}_S$. Similarly, $\mathfrak{G}_S \leq \mathfrak{G}_{B^+}$ for $B^+$ is an orbit of $\mathfrak{G}_S$. Accordingly, we hold the desire in the 1st case.

Now let us show that $B^+$ is a block for $\mathfrak{G}$ in the 2nd case. Suppose $\sigma$ is a permutation in $\mathfrak{G}$ such that $\sigma B^+ \cap B^+ \neq \emptyset$. In accordance with the definition to $B^+$, there exist $\alpha$ and $\beta$ in $\mathfrak{G}_S$ so that $\sigma\alpha b = \beta b$. Consequently, $\beta^{-1} \sigma \alpha b = b$ and thus $\beta^{-1} \sigma \alpha \in \mathfrak{G}_b$. Since $B$ is block, $\mathfrak{G}_b \leq \mathfrak{G}_B$. Hence $\mathfrak{G}_b \leq \mathfrak{G}_S$ and therefore $\beta^{-1} \sigma \alpha \in \mathfrak{G}_S$, which means $\sigma$ belongs to $\mathfrak{G}_S$. Note that $\mathfrak{G}_S \leq \mathfrak{G}_{B^+}$, so $\sigma$ belongs to $\mathfrak{G}_S$. As a result, $\sigma B^+ = B^+$ according to the definition to $B^+$.

In summary, the subgroup $\mathfrak{G}_S$ of $\mathfrak{G}$ contains the stabilizer $\mathfrak{G}_b$ with one orbit $B^+$ which is a block for $\mathfrak{G}$ and contains the vertex $b$. Then in accordance with the generalized version of Lemma \ref{LemBlock&Stabilizers}, one can see $\mathfrak{G}_S = \mathfrak{G}_{B^+}$. 
\end{proof}

It is possible that $\mathfrak{G}_{B^+}$ is properly larger than $\mathfrak{G}_B$. Let us take the Petersen graph as an example. Recall that the stabilizer $\mathfrak{G}_1$ possesses two non-trivial orbits $\{ 2,5,6 \}$ and $\{ 3,4,7,8,9,10 \}$, and $\{ \{3,9,10\},\{4,7,8\} \}$ is a block system of $\mathfrak{G}_1$. Let $\mathfrak{H}$ be the stabilizer of the block $B := \{ 3,9,10 \}$ in $\mathfrak{G}_1$. Obviously, $\mathfrak{H} \lneq \mathfrak{G}_1$. It is easy to see however that the action of $\mathfrak{H}$ on the orbit $S := \{ 2,5,6 \}$ is transitive, so the stabilizer of $S$ is $\mathfrak{G}_1$ itself. The reason why it happens is that there is no non-trivial block system in the orbit $\{ 2,5,6 \}$ corresponding to the system $\{ \{3,9,10\},\{4,7,8\} \}$. As a result, in order to see all structures of the action of $\mathfrak{G}$ on its one orbit, say, $T$, we may have to examine all those IRs of $\mathfrak{G}$ such that $\dproj{W}( \pmb{e}_t ) \neq \pmb{0}$, $t \in T$. 

An irreducible representation $W$ of $\mathfrak{G}$ is said to be {\it  relevant} to a block system $\mathscr{B}$ if for any $B\in\mathscr{B}$, $\dproj{W}(\pmb{R}_B) = |B|\cdot\dproj{W}(\pmb{e}_b)$, $b\in B$, and $\dproj{W}(\pmb{R}_C) = \pmb{0}$ for any block $C$ containing $B$ properly.

\begin{Theorem}\label{Thm-ConnectionBlockSystemInDistinctOrbits} 
Let $\mathfrak{G}$ be a permutation group acting on $V$. Suppose $B$ is a block for $\mathfrak{G}$ contained in an orbit $T'$ of $\mathfrak{G}$ and $S$ is an orbit of $\mathfrak{G}_B$ in another orbit $T''$ of $\mathfrak{G}$. Then $S$ is a block for $\mathfrak{G}$ if and only if there exists an irreducible representation $W$ of $\mathfrak{G}$ relevant to the block system $\mathscr{B}^+$ containing the block $B^+ := \{ \alpha b \mid \alpha \in \mathfrak{G}_S \}$ such that all projections of $\{ \pmb{e}_s \mid s\in S \}$ onto $\dproj{W}(\pmb{R}_{B^+})$ are the same and for any $s\in S$, $\dproj{W}(\pmb{e}_s)$ and $\dproj{W}(\pmb{R}_{B^+})$ are non-trivially linearly dependent. 
\end{Theorem}
\begin{proof}[\bf Proof]
Note that $S$ is actually one orbit of $\mathfrak{G}_S$, so, in order to show that $S$ is a block for $\mathfrak{G}$, it is sufficient to show that there exists one element $s$ in $S$ such that $\mathfrak{G}_s \leq \mathfrak{G}_S$.
In accordance with Lemma \ref{Lem-ConnectionBlockSystemInDistinctOrbits}, $\mathfrak{G}_S = \mathfrak{G}_{B^+}$. Consequently, it suffices to show that for any permutation $\delta$ in $\mathfrak{G}_s$, $\delta$ belongs to $\mathfrak{G}_{B^+}$.

It is easy to see that the requirement enjoyed by $S$ implies that 
$$
\left\langle \frac{1}{ \| \dproj{W}( \pmb{e}_s ) \|} \dproj{W}( \pmb{e}_s ),\frac{1}{ \| \dproj{W}( \pmb{R}_{B^+} ) \|} \dproj{W}( \pmb{R}_{B^+} ) \right\rangle = 1 \mbox{ or } -1, ~ \forall s \in S
$$ 
and 
$$
\dproj{W}( \pmb{e}_{s'} ) = \dproj{W}( \pmb{e}_{s''} ), ~ \forall s',s'' \in S.
$$
Consequently, for any member $s$ in $S$,
$$
\dproj{W}( \pmb{e}_s ) = \left\langle \dproj{W}( \pmb{e}_s ),\frac{1}{ \| \dproj{W}( \pmb{R}_{B^+} ) \|} \dproj{W}( \pmb{R}_{B^+} ) \right\rangle \cdot \frac{1}{ \| \dproj{W}( \pmb{R}_{B^+} ) \| }\dproj{W}( \pmb{R}_{B^+} ),
$$
and thus
\begin{align*}
\delta \hspace{0.4mm} \dproj{W}( \pmb{R}_{B^+} ) 
& = \frac{ \| \dproj{W}( \pmb{R}_{B^+} ) \| }{ \left\langle \dproj{W}( \pmb{e}_s ),\frac{1}{ \| \dproj{W}( \pmb{R}_{B^+} ) \|} \dproj{W}( \pmb{R}_{B^+} ) \right\rangle  }\cdot \delta \hspace{0.4mm} \dproj{W}( \pmb{e}_{s} ) \\
& = \frac{ \| \dproj{W}( \pmb{R}_{B^+} ) \| }{ \left\langle \dproj{W}( \pmb{e}_s ),\frac{1}{ \| \dproj{W}( \pmb{R}_{B^+} ) \|} \dproj{W}( \pmb{R}_{B^+} ) \right\rangle  }\cdot \dproj{W}( \pmb{e}_{s} ) \\ 
& = \dproj{W}( \pmb{R}_{B^+} ).
\end{align*}

Since $W$ is relevant to $\mathscr{B}^+$, $\sigma \neq \mathfrak{G}_{B^+} \Rightarrow \sigma \hspace{0.4mm} \dproj{W}( \pmb{R}_{B^+} ) \neq \dproj{W}( \pmb{R}_{B^+} )$. As a result, the permutation $\delta$, belonging to $\mathfrak{G}_s$, must be one member of $\mathfrak{G}_{B^+}$.

\vspace{3mm}
We now turn to the necessity of our assertion. Note first that there is at least one IR of $\mathfrak{G}$ relevant to $\mathscr{B}^+$ onto which the projection of $\pmb{R}_S$ is non-trivial. In fact, if it is not the case then $\dproj{W_{B^+}}( \pmb{R}_{T''} ) = \pmb{0}$ for any $W_{B^+}$ of $\mathfrak{G}$ relevant to $\mathscr{B}^+$, so one cannot figure out $\Pi_B^*$ by means of $\Pi_B^{\mathrm{IRP}}$, which is in contradiction with the fact that $\Pi_B^* = \Pi_B^{\mathrm{IRP}}$. 

Accordingly, let us assume that $W$ is one of IRs of $\mathfrak{G}$ relevant to $\mathscr{B}^+$ such that $\dproj{W}( \pmb{R}_S ) \neq \pmb{0}$. Since $S$ is a block for $\mathfrak{G}$, $\dproj{W}( \pmb{R}_S ) = |S| \cdot \dproj{W}( \pmb{e}_s )$ due to Lemma \ref{Lemma-EquitablePartProjOntoIRs}, and thus the 1st claim follows.

In accordance with Lemma \ref{Lem-ConnectionBlockSystemInDistinctOrbits}, $\mathfrak{G}_S = \mathfrak{G}_{B^+}$. Consequently, for any permutation $\xi \in \mathfrak{G}_{B^+}$, 
$$
\xi \hspace{0.4mm} \dproj{W}( \pmb{R}_S ) = \dproj{W}( \pmb{R}_S ) \mbox{ and }
\xi \hspace{0.4mm} \dproj{W}( \pmb{R}_{B^+} ) = \dproj{W}( \pmb{R}_{B^+} ).
$$
Then these two vectors $\dproj{W}( \pmb{R}_S )$ and $\dproj{W}( \pmb{R}_{B^+} )$ would be linearly dependent, otherwise the subspace $W[B^+] := \{ \pmb{w} \in W \mid \xi \hspace{0.4mm} \pmb{w} = \pmb{w}, \forall \xi in \mathfrak{G}_{B^+} \}$ is not irreducible, which is in contradiction with propositions \ref{Prop-IRsFeatureGvNonTrivial} and \ref{Prop-IRsFeatureGvTrivial}. Therefore, the 2nd claim follows.
\end{proof}

In virtue of those connections between block systems in distinct orbits of $\mathfrak{G}$, it would be worthwhile investigating {\it more general blocks} which enjoy the main feature ($\sigma B = B$ or $\sigma B \cap B = \emptyset$) of blocks but may contain vertices belonging to different orbits, for that kind of blocks or block systems could reveal more deeply the structure of the action of $\mathfrak{G}$ on $V$ and on $\R^n$.

There is an interesting connection between block systems of $\mathfrak{G}$ and equitable partitions of $G$. We prove the result here only for the transitive case, but by means of Lemma \ref{Lem-ConnectionBlockSystemInDistinctOrbits} and Theorem \ref{Thm-ConnectionBlockSystemInDistinctOrbits} one could establish a more general result.

\begin{Theorem}\label{Thm-EquitablePart&BlockSystem}
Let $G$ be a vertex-transitive graph and let $\mathfrak{G}$ be the automorphism group of $G$. Suppose $B$ is a block for $\mathfrak{G}$. Then the block system $\mathscr{B}$ containing $B$ as one member constitutes an equitable partition of $G$.
\end{Theorem}
\begin{proof}[\bf Proof]
Let $\Pi_{\mathscr{B}}$ be the partition of $V(G)$ consisting of blocks in the system $\mathscr{B}$ and let $\mathbf{R} = (\pmb{R}_1 \cdots \pmb{R}_s)$ be the characteristic matrix of $\Pi_{\mathscr{B}}$. In accordance with Lemma \ref{Lemma-EquitablePartProj}, it is sufficient, in order to establish our result, to show that if $x$ and $y$ belong to the same cell of $\Pi_{\mathscr{B}}$, then $
\langle \pmb{e}_x,\dproj{V_\lambda}(\pmb{R}_j) \rangle = \langle \pmb{e}_y,\dproj{V_\lambda}(\pmb{R}_j) \rangle, ~\forall~\lambda\in\mathrm{spec}~\mathbf{A}(G) \mbox{\it ~and } j\in [s].
$

Pick arbitrarily one eigenspace $V_{\lambda}$ and suppose $\oplus_{p} W_{\lambda,p}$ is a decomposition of $V_{\lambda}$ into IRs of $\mathfrak{G}$. Let us begin with a simple but useful observation that 
\begin{equation}\label{Eq-RelationProj}
\mbox{if } \dproj{W_{\lambda,p}}( \pmb{R}_B ) = \pmb{0} \mbox{ then } \dproj{W_{\lambda,p}}( \pmb{R}_{\sigma B} ) = \pmb{0}, ~\forall \sigma\in\mathfrak{G}.
\end{equation}
In fact, since $\sigma$ is an automorphism of $G$ and $W_{\lambda,p}$ is an $\mathfrak{G}$-invariant subspace, $\dproj{W_{\lambda,p}}\circ\sigma = \sigma\circ\dproj{W_{\lambda,p}}$ according to Lemma \ref{ProjOperatorCommutative}. Hence
$$
\dproj{W_{\lambda,p}}( \pmb{R}_{\sigma B} ) 
= \dproj{W_{\lambda,p}}\circ\sigma( \pmb{R}_{B} )
= \sigma\circ\dproj{W_{\lambda,p}}( \pmb{R}_{B} )
= \sigma( \pmb{0} ) = \pmb{0}.
$$  

Suppose without the loss of generality that $x$ and $y$ belong to the block $\gamma B$ for some $\gamma\in\mathfrak{G}$. Then 
\begin{align*}
\langle \pmb{e}_x,\dproj{V_{\lambda}}( R_{\sigma B} ) \rangle
& = \left\langle \sum_{\eta\in\mathrm{spec} \hspace{0.3mm} \mathbf{A}(G)} \dproj{V_{\eta}}( \pmb{e}_x ),\dproj{V_{\lambda}}( R_{\sigma B} ) \right\rangle \\
& = \left\langle  \dproj{V_{\lambda}}( \pmb{e}_x ),\dproj{V_{\lambda}}( R_{\sigma B} ) \right\rangle \\
& = \left\langle \sum_{p} \dproj{W_{\lambda,p}}( \pmb{e}_x ),\sum_{p} \dproj{W_{\lambda,p}}( R_{\sigma B} ) \right\rangle \\
& = \sum_{p} \left\langle \dproj{W_{\lambda,p}}( \pmb{e}_x ),\dproj{W_{\lambda,p}}( R_{\sigma B} ) \right\rangle.
\end{align*}
Due to Lemma \ref{Lemma-EquitablePartProjOntoIRs}, the last sum can be divided into two parts:
\begin{equation}\label{Eq-TwoParts}
\sum_{q'} \left\langle \dproj{W_{\lambda,q'}}( \pmb{e}_x ),\dproj{W_{\lambda,q'}}( R_{\sigma B} ) \right\rangle
+ 
\sum_{q''} \left\langle \dproj{W_{\lambda,q''}}( \pmb{e}_x ),\dproj{W_{\lambda,q''}}( R_{\sigma B} ) \right\rangle
\end{equation}
{\it s.t.,} $\dproj{W_{\lambda,q'}}( R_{\gamma B} ) = \pmb{0}$ and 
$\dproj{W_{\lambda,q''}}( R_{\gamma B} ) = |B| \cdot \dproj{W_{\lambda,q''}}( \pmb{e}_{\gamma b} )$, where $b$ is one member in $B$. 

In the first case, $\dproj{W_{\lambda,q'}}( R_{\sigma B} )$ is also equal to $\pmb{0}$ by means of the relation (\ref{Eq-RelationProj}) and thus the first term in the sum (\ref{Eq-TwoParts}) is equal to 0. In the second case, 
\begin{align*}
\sum_{q''} \left\langle \dproj{W_{\lambda,q''}}( \pmb{e}_x ),\dproj{W_{\lambda,q''}}( R_{\sigma B} ) \right\rangle
& = \sum_{q''} \left\langle \frac{1}{|B|}\cdot\dproj{W_{\lambda,q''}}( R_{\gamma B} ),\dproj{W_{\lambda,q''}}( R_{\sigma B} ) \right\rangle \\
& = \sum_{q''} \left\langle \dproj{W_{\lambda,q''}}( \pmb{e}_y ),\dproj{W_{\lambda,q''}}( R_{\sigma B} ) \right\rangle.
\end{align*}
Therefore, $
\langle \pmb{e}_x,\dproj{V_\lambda}(\pmb{R}_j) \rangle = \langle \pmb{e}_y,\dproj{V_\lambda}(\pmb{R}_j) \rangle, ~\forall~\lambda\in\mathrm{spec}~\mathbf{A}(G) \mbox{\it ~and } j\in [s].
$
\end{proof}

\section{Geometric features of IRs of $\mathfrak{G}$}

The major goal of this section is to establish Theorem \ref{Thm-IsomorphismThmBetweenIRs}, which exposes the geometric feature enjoyed by isomorphic IRs of $\mathfrak{G}$ and so provides an apparatus by means of that we could decompose an eigenspace of $\dAM{G}$ into IRs of $\mathfrak{G}$.

Recall that the norm $\| \pmb{v} \|$ of a vector $\pmb{v}$ is defined as $\sqrt{ \langle \pmb{v},\pmb{v} \rangle}$.

\begin{Lemma}\label{Lemma-IRFeatureIsometry}
Suppose $\mathcal{S}$ is an isometry on $\R^n$ and $U$ is an $\mathcal{S}$-invariant subspace in $\R^n$. Let $\pmb{u}_0$ be a vector in $U$ with norm 1. If $\langle \pmb{u}_0,\mathcal{S}\pmb{u}_0 \rangle \geq \langle \pmb{u},\mathcal{S}\pmb{u} \rangle$ or $\langle \pmb{u}_0,\mathcal{S}\pmb{u}_0 \rangle \leq \langle \pmb{u},\mathcal{S}\pmb{u} \rangle$ for any vector $\pmb{u}$ in $U$ with norm 1, then the subspace $\mathrm{span}\{ \pmb{u}_0,\mathcal{S}\pmb{u}_0 \}$ is invariant for $\mathcal{S}$, {\it i.e.,} $\pmb{u}_0$ belongs to an irreducible representation of the circulant group $\langle \mathcal{S} \rangle$.
\end{Lemma}
\begin{proof}[\bf Proof]
We prove the assertion for the case that $\left\langle \pmb{u}_0,\mathcal{S}\pmb{u}_0  \right\rangle \geq \left\langle \pmb{u},\mathcal{S}\pmb{u}  \right\rangle$ for any vector $\pmb{u}$ in $U$ with norm 1. If the inner product $\left\langle \pmb{u}_0,\mathcal{S}\pmb{u}_0  \right\rangle = 1$, our assertion holds evidently. For that reason we assume $| \left\langle \pmb{u}_0,\mathcal{S}\pmb{u}_0  \right\rangle | < 1$ in what follows.

Suppose $\oplus_{p=1}^s W_p$ is a decomposition of $U$ into IRs of the group $\langle \mathcal{S} \rangle$. The key observation is that if $\dproj{W_p}( \pmb{u}_0 ) \neq \pmb{0}$, $1\leq p \leq s$, then $\left\langle \pmb{w}_p,\mathcal{S}\pmb{w}_p  \right\rangle = \left\langle \pmb{u}_0,\mathcal{S}\pmb{u}_0 \right\rangle$, where the vector $\pmb{w}_p = \frac{1}{ \| \dproj{W_p}( \pmb{u}_0 ) \| } \cdot \dproj{W_p}( \pmb{u}_0 )$.

In fact, let us assume without losing any generality that $\pmb{u}_0 = \sum_{p=1}^r u_p\pmb{w}_p$ for some integer $r \leq s$. Then
\begin{align*}
\left\langle \pmb{u}_0,\mathcal{S}\pmb{u}_0  \right\rangle
& = \left\langle \sum_{p=1}^r u_p\pmb{w}_p,\mathcal{S}\sum_{p=1}^r u_p\pmb{w}_p  \right\rangle \\
& = \sum_{p} u_k^2 \left\langle \pmb{w}_p,\mathcal{S}\pmb{w}_p  \right\rangle \\
& \leq \left( \sum_{p} u_k^2 \right) \left\langle \pmb{u}_0,\mathcal{S}\pmb{u}_0  \right\rangle \\
& = \left\langle \pmb{u}_0,\mathcal{S}\pmb{u}_0  \right\rangle.
\end{align*}
Consequently, $\left\langle \pmb{w}_p,\mathcal{S}\pmb{w}_p  \right\rangle = \left\langle \pmb{u}_0,\mathcal{S}\pmb{u}_0 \right\rangle$, $p = 1,\ldots,r$.

On the other hand, one can readily verify that if two IRs $W'$ and $W''$ of the group $\langle \mathcal{S} \rangle$ enjoy the relation that $\exists \hspace{0.7mm} \pmb{w}' \in W'$ and $\pmb{w}'' \in W''$ {\it s.t.,} $\| \pmb{w}' \| = \| \pmb{w}'' \| = 1$ and $\left\langle \pmb{w}',\mathcal{S}\pmb{w}' \right\rangle = \left\langle \pmb{w}'',\mathcal{S}\pmb{w}'' \right\rangle$ then those two representations $W'$ and $W''$ are isomorphic with respect to $\langle \mathcal{S} \rangle$. Therefore, those IRs $W_1,\ldots,W_r$ are isomorphic to one another with respect to $\langle \mathcal{S} \rangle$.

We are now ready to show that the subspace $\mathrm{span}\{ \pmb{u}_0,\mathcal{S}\pmb{u}_0 \}$ is $\mathcal{S}$-invariant. One moment's reflection enables us to see that it is sufficient to show that $\mathcal{S}^2 \pmb{u}_0$ is also a vector in  $\mathrm{span}\{ \pmb{u}_0,\mathcal{S}\pmb{u}_0 \}$. Because those IRs $W_1,\ldots,W_r$ are isomorphic to one another with respect to $\langle \mathcal{S} \rangle$, there are two real numbers $r'$ and $r''$ such that $\mathcal{S}^2 \pmb{w}_p = r' \pmb{w}_p + r'' \mathcal{S}\pmb{w}_p$, $p = 1,\ldots,r$. As a result, 
\begin{align*}
\mathcal{S}^2 \pmb{u}_0 
& = \mathcal{S}^2 \left( \sum_{p=1}^r u_p\pmb{w}_p \right) \\
& = \sum_{p} u_p \cdot \mathcal{S}^2 \pmb{w}_p \\
& = \sum_{p} u_p \cdot \left( r' \pmb{w}_p + r'' \mathcal{S}\pmb{w}_p \right) \\
& = r' \left( \sum_{p} u_p \pmb{w}_p \right) + r'' \mathcal{S} \left( \sum_{p} u_p \pmb{w}_p \right) \\
& = r' \pmb{u}_0 + r'' \mathcal{S} \pmb{u}_0.
\end{align*}
Clearly, the vector $r' \pmb{u}_0 + r'' \mathcal{S} \pmb{u}_0$ belongs to the subspace $\mathrm{span}\{ \pmb{u}_0,\mathcal{S}\pmb{u}_0 \}$ and thus $\mathcal{S}^2 \pmb{u}_0$ also belongs to the subspace.
\end{proof}

\begin{Lemma}\label{Lemma-IsomorphismicIRsKeepingAnglesUp}
Suppose $\mathfrak{G}$ is a permutation group acting on $V$ transitively and $W'$ and $W''$ are two irreducible representations of $\mathfrak{G}$ isomorphic to one another. If $\phi$ is an $\mathfrak{G}$-module isomorphism from $W'$ to $W''$, then for any two vectors $\pmb{u}$ and $\pmb{v}$ in $W'$ with norm 1 
$$ \langle \pmb{u},\pmb{v} \rangle = \left\langle \frac{1}{ C_\phi } \phi\hspace{0.4mm}\pmb{u},\frac{1}{ C_\phi }\phi\hspace{0.4mm}\pmb{v} \right\rangle$$
where $C_\phi = \| \phi\hspace{0.4mm}\pmb{u} \| = \| \phi\hspace{0.4mm}\pmb{v} \|$. In other words, an $\mathfrak{G}$-module isomorphism between two IRs preserves angles among vectors.
\end{Lemma}
\begin{proof}[\bf Proof]
Let $\sigma$ be a permutation in $\mathfrak{G}$. Suppose $W' = \oplus_{k=1}^s X_k$ is the standard decomposition with respect to the circulant group $\langle \sigma \rangle$, {\it i.e.,} each subspace $X_k$ consists of all those IRs of $\langle \sigma \rangle$ isomorphic to one another and $X_i$ is not isomorphic to $X_j$ if $i \neq j$. In virtue of the Lemma \ref{Lemma-IRFeatureIsometry}, we can make a further assumption that $\left\langle \pmb{x}_i,\sigma\pmb{x}_i \right\rangle > \left\langle \pmb{x}_j,\sigma\pmb{x}_j \right\rangle$ if $i > j$, where $\pmb{x}_i$ and $\pmb{x}_j$ are two vectors in $X_i$ and $X_j$ respectively with norm 1, $1\leq i,j \leq s$.

Since $\phi$ is an $\mathfrak{G}$-module isomorphism, it is an $\langle \sigma \rangle$-module isomorphism, and thus $X_k$ is isomorphic to $\phi X_k$, $k = 1,\ldots,s$. Furthermore, one can easily see that $\phi X_i \perp \phi X_j$ if $i \neq j$. 

In accordance with the characterization of $\mathfrak{G}$-module map on IRs, which is summarized in Lemma \ref{Lem-SchurLemmaOnEuclideanSpace}, it is not difficult to see that if $\pmb{u}$ and $\pmb{v}$ are two vectors in $W'$ such that $\| \pmb{u} \| = \| \pmb{v} \|$ then $\| \phi \hspace{0.4mm} \pmb{u} \| = \| \phi \hspace{0.4mm} \pmb{v} \|$. 

To establish the relation we claim, the key is to show that if a group of vectors $\{ \pmb{b}_{k,1}',\ldots,\pmb{b}_{k,d_k}' \mid k = 1,\ldots,s \}$ constitute an orthonormal basis of $W'$, where those vectors $\pmb{b}_{k,1}',\ldots,\pmb{b}_{k,d_k}'$ belong to $X_k$, then $\{ \phi \hspace{0.4mm} \pmb{b}_{k,1}',\ldots,\phi \hspace{0.4mm}\pmb{b}_{k,d_k}' \mid k = 1,\ldots,s \}$ are also a group of orthogonal vectors. In fact, if $\pmb{u} = \sum_{k,i} u_{k,i} \pmb{b}_{k,i}'$ and $\pmb{v} = \sum_{k,i} v_{k,i} \pmb{b}_{k,i}'$ then
\begin{align*}
\left\langle \pmb{u},\pmb{v} \right\rangle
& = \left\langle \sum_{k,i} u_{k,i} \pmb{b}_{k,i}',\sum_{k,i} v_{k,i} \pmb{b}_{k,i}' \right\rangle \\
& = \sum_{k,i} u_{k,i} \cdot v_{k,i}  \\
& = \sum_{k,i} \left\langle 
u_{k,i} \cdot \frac{1}{ \| \phi \hspace{0.4mm} \pmb{b}_{k,i}' \| } \phi \hspace{0.4mm} \pmb{b}_{k,i}',
v_{k,i} \cdot \frac{1}{ \| \phi \hspace{0.4mm} \pmb{b}_{k,i}' \| } \phi \hspace{0.4mm} \pmb{b}_{k,i}'
\right\rangle \\
& = \left\langle 
\sum_{k,i} u_{k,i} \frac{1}{ \| \phi \hspace{0.4mm} \pmb{b}_{k,i}' \| } \phi \hspace{0.4mm} \pmb{b}_{k,i}',
\sum_{k,i} v_{k,i} \frac{1}{ \| \phi \hspace{0.4mm} \pmb{b}_{k,i}' \| } \phi \hspace{0.4mm} \pmb{b}_{k,i}'
\right\rangle \\
& = \left\langle 
\frac{1}{ \| \phi \hspace{0.4mm} \pmb{b}_{k,i}' \| } \sum_{k,i} u_{k,i} \phi \hspace{0.4mm} \pmb{b}_{k,i}',
\frac{1}{ \| \phi \hspace{0.4mm} \pmb{b}_{k,i}' \| } \sum_{k,i} v_{k,i} \phi \hspace{0.4mm} \pmb{b}_{k,i}'
\right\rangle \\
& = \left\langle 
\frac{1}{ \| \phi \hspace{0.4mm} \pmb{b}_{k,i}' \| } \phi\hspace{0.4mm}\pmb{u},
\frac{1}{ \| \phi \hspace{0.4mm} \pmb{b}_{k,i}' \| } \phi\hspace{0.4mm}\pmb{v}
\right\rangle \\
& =  \left\langle 
\frac{1}{ \| \phi \hspace{0.4mm} \pmb{u} \| } \phi\hspace{0.4mm}\pmb{u},
\frac{1}{ \| \phi \hspace{0.4mm} \pmb{v} \| } \phi\hspace{0.4mm}\pmb{v}
\right\rangle.
\end{align*}
As a result, in order to prove the assertion, it is sufficient to show that if $\pmb{u}_k$ and $\pmb{v}_k$ are two vectors in $X_k$ with norm 1, $1 \leq k \leq s$, then $\left\langle \pmb{u}_k,\pmb{v}_k \right\rangle = \left\langle \frac{1}{ \| \phi\hspace{0.4mm} \pmb{u}_k \|} \phi\hspace{0.4mm}\pmb{u}_k,\frac{1}{ \| \phi \hspace{0.4mm} \pmb{v}_k \|}\phi\hspace{0.4mm}\pmb{v}_k \right\rangle$.

According to the definition to $X_k$ and Lemma \ref{Lemma-IRFeatureIsometry}, the subspace $\mathrm{span} \hspace{0.4mm} \{ \pmb{u}_k,\sigma\pmb{u}_k \}$ is irreducible for $\langle \sigma \rangle$. Because $\phi$ is an $\mathfrak{G}$-module isomorphism,  
$$
\left\langle \pmb{u}_k,\sigma\pmb{u}_k \right\rangle = \left\langle \frac{1}{ \| \phi\hspace{0.4mm} \pmb{u}_k \|} \phi\hspace{0.4mm}\pmb{u}_k,\frac{1}{ \| \phi \hspace{0.4mm} (\sigma\pmb{u}_k) \|}\phi\hspace{0.4mm} (\sigma\pmb{u}_k) \right\rangle.
$$
Suppose $\dproj{ \mathrm{span} \hspace{0.4mm} \{ \pmb{u}_k,\sigma\pmb{u}_k \} }( \pmb{v}_k ) = w_1 \pmb{u}_k + w_2 \sigma\pmb{u}_k $.  Then
\begin{align*}
\left\langle \pmb{u}_k,\pmb{v}_k \right\rangle 
& = \left\langle \pmb{u}_k, \dproj{ \mathrm{span} \hspace{0.4mm} \{ \pmb{u}_k,\sigma\pmb{u}_k \} }( \pmb{v}_k )
+ \dproj{ \big( \mathrm{span} \hspace{0.4mm} \{ \pmb{u}_k,\sigma\pmb{u}_k \} \big)^\perp }( \pmb{v}_k ) \right\rangle \\
& = \left\langle \pmb{u}_k,w_1 \pmb{u}_k + w_2 \sigma\pmb{u}_k \right\rangle  \\
& = w_1 + w_2 \cdot \left\langle \pmb{u}_k,\sigma\pmb{u}_k \right\rangle  \\
& = w_1 \cdot \left\langle \frac{1}{ \| \phi\hspace{0.4mm} \pmb{u}_k \|} \phi\hspace{0.4mm}\pmb{u}_k,\frac{1}{ \| \phi \hspace{0.4mm} \pmb{u}_k \|}\phi\hspace{0.4mm} \pmb{u}_k \right\rangle
 + w_2 \cdot \left\langle \frac{1}{ \| \phi\hspace{0.4mm} \pmb{u}_k \|} \phi\hspace{0.4mm}\pmb{u}_k,\frac{1}{ \| \phi \hspace{0.4mm} (\sigma\pmb{u}_k) \|}\phi\hspace{0.4mm} (\sigma\pmb{u}_k) \right\rangle \\ 
& = \left\langle \frac{1}{ \| \phi\hspace{0.4mm} \pmb{u}_k \|} \phi\hspace{0.4mm}\pmb{u}_k, 
w_1 \frac{1}{ \| \phi\hspace{0.4mm} \pmb{u}_k \|} \phi\hspace{0.4mm}\pmb{u}_k + 
w_2 \frac{1}{ \| \phi \hspace{0.4mm} (\sigma\pmb{u}_k) \|}\phi\hspace{0.4mm} (\sigma\pmb{u}_k)
\right\rangle \\
& = \left\langle \frac{1}{ \| \phi\hspace{0.4mm} \pmb{u}_k \|} \phi\hspace{0.4mm}\pmb{u}_k, 
 \frac{1}{ \| \phi\hspace{0.4mm} \pmb{u}_k \|} \left( w_1 \phi\hspace{0.4mm}\pmb{u}_k + w_2 \phi\hspace{0.4mm} (\sigma\pmb{u}_k) \right)
\right\rangle \\
& = \left\langle \frac{1}{ \| \phi\hspace{0.4mm} \pmb{u}_k \|} \phi\hspace{0.4mm}\pmb{u}_k,\frac{1}{ \| \phi \hspace{0.4mm} \pmb{v}_k \|}\phi\hspace{0.4mm} \pmb{v}_k \right\rangle.
\end{align*}
The last equation holds because the subspace $\hspace{0.4mm}\mathrm{span} \hspace{0.4mm} \{ \phi \hspace{0.4mm}  \pmb{u}_k,\phi ( \sigma\pmb{u}_k ) \}$ is $\langle\sigma\rangle$-invariant.
\end{proof}

\begin{Theorem}[\bf Isomorphism Theorem]\label{Thm-IsomorphismThmBetweenIRs}
Let $G$ be a vertex-transitive graph with automorphism group $\mathfrak{G}$ and let $W'$ and $W''$ be two irreducible representations of $\mathfrak{G}$. Then the following statements are equivalent.

\begin{enumerate}
\item[i)] $W'$ and $W''$ are isomorphic representations;

\item[ii)] The angle between $\dproj{W'}(\pmb{e}_u)$ and $\dproj{W'}(\pmb{e}_v)$ is equal to the angle between $\dproj{W''}(\pmb{e}_u)$ and $\dproj{W''}(\pmb{e}_v)$, $\forall u,v \in V(G)$, i.e.,
$$
\left\langle \pmb{w}_u',\pmb{w}_v' \right\rangle = \left\langle \pmb{w}_u'',\pmb{w}_v'' \right\rangle,
$$
where $\pmb{w}_x' = \frac{1}{\| \dproj{W'}(\pmb{e}_x) \|}\dproj{W'}(\pmb{e}_x)$ and $\pmb{w}_x'' = \frac{1}{\| \dproj{W'}(\pmb{e}_x) \|}\dproj{W''}(\pmb{e}_x)$,  $x = u$ or $v$.

\item[iii)] The map $\phi : \pmb{w}_{v}' \mapsto \pmb{w}_{v}''$, $v\in V(G)$, is an invertiblly linear map.

\end{enumerate}
\end{Theorem}

\begin{proof}[\bf Proof]
Let us begin with the relation that $ i) \Rightarrow ii)$ and consider the case that $\mathfrak{G}_x$ is non-trivial, where $x$ is a vertex of $G$. We first single out a kind of subspaces of $W'$ relevant to $\mathfrak{G}$-module isomorphism. Set $W'[x] = \{ \pmb{w} \in W' \mid \xi\hspace{0.4mm}\pmb{w} = \pmb{w}, \forall \xi\in\mathfrak{G}_x  \}$. Similarly, we can define the subspace $W''[x]$ in $W''$.

The key observation is that if $\phi$ is an $\mathfrak{G}$-module isomorphism from $W'$ to $W''$ then $\phi\hspace{0.4mm} W'[x] = W''[x]$, $\forall x\in V(G)$. In fact, for any non-trivial vector $\pmb{z}_x'$ in $W'[x]$ and permutation $\zeta$ in $\mathfrak{G}_x$, $\zeta( \phi\hspace{0.4mm}\pmb{z}_x' ) = \phi\circ\zeta\hspace{0.4mm}\pmb{z}_x' = \phi\hspace{0.4mm}\pmb{z}_x'$, so $\phi\hspace{0.4mm}\pmb{z}_x'$ belongs to $W''[x]$, so $\phi\hspace{0.4mm} W'[x] \subseteq W''[x]$. Moreover, in accordance with Lemma \ref{Lemma-IRsFeatureGvNonTrivial},  $\dim W'[x] = \dim W''[x] = 1$. Hence $\phi\hspace{0.4mm} W'[x] = W''[x]$.

On the other hand, it is easy to check that $\dproj{W'}( \pmb{e}_x )$ belongs to $W'[x]$. Consequently, $W'[x] = \mathrm{span}\hspace{0.5mm}\big\{ \dproj{W'}( \pmb{e}_x ) \big\}$. Similarly, $W''[x] = \mathrm{span}\hspace{0.5mm}\big\{ \dproj{W''}( \pmb{e}_x ) \big\}$. Accordingly, there would be an $\mathfrak{G}$-module map $\phi$ such that 
$$
\phi \left( \dproj{W'}( \pmb{e}_x ) \right) = \dproj{W'}( \pmb{e}_x ) \mbox{ or } - \dproj{W'}( \pmb{e}_x ), ~ x\in V(G).
$$

Because the $\mathfrak{G}$-module isomorphism $\phi$ preserves angles among vectors, if there exists one vertex $u$ in $V(G)$ such that $\phi ( \dproj{W'}( \pmb{e}_u ) ) = \dproj{W''}( \pmb{e}_u )$ then $\phi ( \dproj{W'}( \pmb{e}_v ) ) = \dproj{W''}( \pmb{e}_v )$ for any vertex $v$ in $V(G)$. Suppose it is not the case, {\it i.e.,} there exists some vertex $u^*$ so that $\phi ( \dproj{W'}( \pmb{e}_{u^*} ) ) = - \dproj{W''}( \pmb{e}_{u^*} )$. 

Let $\gamma$ be one of permutations in $\mathfrak{G}$ such that $\gamma u = u^*$ and let $\pmb{w}_u' = \frac{1}{ \| \dproj{W'}( \pmb{e}_u ) \| } \dproj{W'}( \pmb{e}_u )$ and $\pmb{w}_{u^*}' = \frac{1}{ \| \dproj{W'}( \pmb{e}_{u^*} ) \| } \dproj{W'}( \pmb{e}_{u^*} )$. On the one hand, according to Lemma \ref{Lemma-IsomorphismicIRsKeepingAnglesUp},
\begin{align*}
\left\langle \pmb{w}_u',\pmb{w}_{u^*}' \right\rangle 
& =  \left\langle \frac{1}{ \| \phi\hspace{0.4mm} \pmb{w}_u' \| } \phi\hspace{0.4mm}\pmb{w}_u',\frac{1}{ \| \phi\hspace{0.4mm}\pmb{w}_{u^*}' \| } \phi\hspace{0.4mm}\pmb{w}_{u^*}' \right\rangle \\
& = \left\langle  \pmb{w}_u'',-\pmb{w}_{u^*}'' \right\rangle \\
& = - \left\langle \pmb{w}_u'',\pmb{w}_{u^*}'' \right\rangle,
\end{align*}
where $\pmb{w}_u'' = \frac{1}{ \| \dproj{W''}( \pmb{e}_u ) \| } \dproj{W''}( \pmb{e}_u )$ and $\pmb{w}_{u^*}'' = \frac{1}{ \| \dproj{W''}( \pmb{e}_{u^*} ) \| } \dproj{W''}( \pmb{e}_{u^*} )$.

On the other hand, since $W'$ and $W''$ are two $\mathfrak{G}$-invariant subspaces,
\begin{align*}
\left\langle \pmb{w}_u',\pmb{w}_{u^*}' \right\rangle 
& = \left\langle \pmb{w}_u',\pmb{w}_{\gamma\hspace{0.4mm} u}' \right\rangle \\
& = \left\langle \pmb{w}_u',\gamma\hspace{0.4mm} \pmb{w}_{u}' \right\rangle \\
& = \left\langle 
\frac{1}{ \| \phi\hspace{0.4mm} \pmb{w}_u' \| } \phi\hspace{0.4mm}\pmb{w}_u',
\frac{1}{ \| \phi\hspace{0.4mm} \big( \gamma\hspace{0.4mm} \pmb{w}_{u}' \big) \| } \phi\hspace{0.4mm} \big( \gamma\hspace{0.4mm} \pmb{w}_{u}' \big) 
\right\rangle \\
& = \left\langle 
\frac{1}{ \| \phi\hspace{0.4mm} \pmb{w}_u' \| } \phi\hspace{0.4mm}\pmb{w}_u',
\frac{1}{ \| \gamma \hspace{0.4mm} \big( \phi \hspace{0.4mm} \pmb{w}_{u}' \big) \| } \gamma \hspace{0.4mm} \big( \phi \hspace{0.4mm} \pmb{w}_{u}' \big) 
\right\rangle \\
& = \left\langle 
\pmb{w}_u'',\gamma \hspace{0.4mm} \pmb{w}_u''
\right\rangle \\
& = \left\langle\pmb{w}_u'',\pmb{w}_{u^*}'' \right\rangle.
\end{align*}
This is a contradiction provided that $\left\langle \pmb{w}_u',\pmb{w}_{u^*}' \right\rangle \neq 0$. In the case that $\left\langle \pmb{w}_u',\pmb{w}_{u^*}' \right\rangle = 0$, which means $W'[u] = W'[u^*]$, one can prove the relation by the same argument but  selecting vertices with fixing subspace not the same as $W'[u]$. Hence if there exists one vertex $u$ in $V(G)$ such that $\phi ( \dproj{W'}( \pmb{e}_u ) ) = \dproj{W''}( \pmb{e}_u )$ then $\phi ( \dproj{W'}( \pmb{e}_v ) ) = \dproj{W''}( \pmb{e}_v )$ for any vertex $v$ in $V(G)$. In virtue of Lemma \ref{Lemma-IsomorphismicIRsKeepingAnglesUp}, we have $
\left\langle \pmb{w}_u',\pmb{w}_v' \right\rangle = \left\langle \pmb{w}_u'',\pmb{w}_v'' \right\rangle,
$ $\forall u, v \in V(G)$, in the case that $\mathfrak{G}_x$ is not trivial.

We now turn to the case that $\mathfrak{G}_x$ only contains the identity of $\mathfrak{G}$. Let $B$ be one of minimal blocks for $\mathfrak{G}$. Instead of featuring relation among subspaces $W'[v]s$, we consider an alternative $W'[B]$ relevant to $\mathfrak{G}_B$, which is defined as 
$$
\left\{ \pmb{w} \in W' \mid \zeta \left( \sum_{k=1}^{|B|} \xi^k \right) \pmb{w} 
= \left( \sum_{k=1}^{|B|} \xi^k \right) \pmb{w}, 
~~\forall \xi\in \mathfrak{G}_B\setminus\{1\} \mbox{ and } \zeta\in\mathfrak{G}_B  \right\}.
$$
According to Lemma \ref{Lemma-IRsFeatureGvTrivial}, there are only two possibilities for $W'[B]$, {\it i.e.,}
$\dim W'[B] = 1$ or $2$. It is not difficult to see that in the case that $\dim W'[B] = 1$ one can employ the same idea in dealing with the case above to establish the relation, so let us consider the case that $\dim W'[B] = 2$ in what follows. 

We first show that $\phi\hspace{0.4mm} W'[B] = W''[B]$. Suppose $\pmb{z}'$ is a vector in $W'[B]$. Then 
$$
\zeta \left( \sum_{k=1}^{|B|} \xi^k \right) \big( \phi\hspace{0.4mm}\pmb{z}' \big)
= \phi \circ \zeta \left( \sum_{k=1}^{|B|} \xi^k \right) \pmb{z}'
= \phi \left( \sum_{k=1}^{|B|} \xi^k \right) \pmb{z}'
= \left( \sum_{k=1}^{|B|} \xi^k \right) \big( \phi\pmb{z}' \big),
$$ 
where $\xi\in \mathfrak{G}_B\setminus\{1\}$ and $\zeta\in\mathfrak{G}_B$. Hence, $\phi\hspace{0.4mm} W'[B] \subseteq W''[B]$. Note that $\phi$ is an $\mathfrak{G}$-module isomorphism from $W'$ to $W''$ and $\dim W'[B] = \dim W''[B] = 2$, so $\phi\hspace{0.4mm} W'[B] = W''[B]$. Therefore these two irreducible representations $W'[B]$ and $W''[B]$ are isomorphic with respect to the stabilizer $\mathfrak{G}_B$.

In accordance with Lemma \ref{Lem-BlockStabilizerCharacterization-GvTrivial}, $\mathfrak{G}_B$ is a circulant group of prime order, so for any $b \in B$ and $\xi \in \mathfrak{G}_B \setminus \{ 1 \}$, 
$$
\left\langle \pmb{w}_b',\xi\hspace{0.4mm}\pmb{w}_{b}' \right\rangle 
= \left\langle \pmb{w}_b'',\xi\hspace{0.4mm}\pmb{w}_{b}'' \right\rangle,
$$
where $\pmb{w}_b' = \frac{1}{ \| \dproj{W'}( \pmb{e}_b ) \| } \dproj{W'}( \pmb{e}_b )$ and $\pmb{w}_b'' = \frac{1}{ \| \dproj{W''}( \pmb{e}_b ) \| } \dproj{W''}( \pmb{e}_b )$. Consequently, 
$$
\left\langle \pmb{w}_s',\pmb{w}_{t}' \right\rangle 
= \left\langle \pmb{w}_s'',\pmb{w}_{t}'' \right\rangle, ~ \forall s,t \in B.
$$
As a result, there could be an $\mathfrak{G}$-module isomorphism $\phi$ so that $\phi : \dproj{W'}( \pmb{e}_v ) \mapsto \dproj{W''}( \pmb{e}_v )$, $v\in V(G)$, so we can hold the desire by means of Lemma \ref{Lemma-IsomorphismicIRsKeepingAnglesUp}.

\vspace{3mm}
We now show the relation that $ii) \Rightarrow iii)$. Because $W'$ is irreducible, 
$$
W' = \mathrm{span}\hspace{0.4mm} \left\{ \dproj{W'}( \pmb{e}_x ) \mid x \in V(G) \right\},
$$ 
so let us assume that $\pmb{w}_{v_1}',\ldots,\pmb{w}_{v_d}'$ constitute a basis of $W'$, where $\pmb{w}_{v_i}' = \frac{1}{ \| \dproj{W'}( \pmb{e}_{v_i} ) \| } \dproj{W'}( \pmb{e}_{v_i} )$, $i = 1,\ldots,d$. We now define a linear map $\psi : \pmb{w}_{v_k}' \mapsto \pmb{w}_{v_k}''$, $k = 1,\ldots,d$. Let $\pmb{x}$ and $\pmb{y}$ be two vectors in $W'$ with non-trivial norms. Suppose $\pmb{x} = \sum_{k=1}^d x_k \pmb{w}_{v_k}'$ and $\pmb{y} = \sum_{k=1}^d y_k \pmb{w}_{v_k}'$. Then 
\begin{align*}
\left\langle \pmb{x},\pmb{y} \right\rangle 
& = \left\langle \sum_{k=1}^d x_k \pmb{w}_{v_k}',\sum_{k=1}^d y_k \pmb{w}_{v_k}' \right\rangle \\
& = \sum_{i,j} x_i y_j \left\langle \pmb{w}_{v_i}',\pmb{w}_{v_j}' \right\rangle \\
& = \sum_{i,j} x_i y_j \left\langle \pmb{w}_{v_i}'',\pmb{w}_{v_j}'' \right\rangle \\
& = \sum_{i,j} x_i y_j \left\langle \psi\hspace{0.4mm}\pmb{w}_{v_i}',\psi\hspace{0.4mm}\pmb{w}_{v_j}' \right\rangle \\
& = \left\langle \sum_{k=1}^d x_k\cdot \psi\hspace{0.4mm} \pmb{w}_{v_k}',\sum_{k=1}^d y_k\cdot \psi\hspace{0.4mm} \pmb{w}_{v_k}' \right\rangle \\
& = \left\langle \psi\hspace{0.4mm} \pmb{x},\psi\hspace{0.4mm} \pmb{y} \right\rangle.
\end{align*}
Hence the linear map $\psi$ preserves the inner product and thus $\psi$ is a linear isomorphism between subspaces $W'$ and $W''$. Therefore, the group of vectors $\pmb{w}_{v_1}'',\ldots,\pmb{w}_{v_d}''$ comprise a basis of $W''$. 

Let us show that $\psi\big( \pmb{w}_{x}' \big) = \pmb{w}_{x}''$, $\forall x \in V(G)$. In fact, for any positive integer $k$ not greater that $d$, we have 
\begin{align*}
\left\langle \psi\hspace{0.4mm} \pmb{w}_x',\pmb{w}_{v_k}'' \right\rangle 
& = \left\langle \psi\hspace{0.4mm} \pmb{w}_x', \psi\hspace{0.4mm} \pmb{w}_{v_k}' \right\rangle \\
& = \left\langle \hspace{0.4mm} \pmb{w}_x',\hspace{0.4mm} \pmb{w}_{v_k}' \right\rangle \\
& = \left\langle \hspace{0.4mm} \pmb{w}_x'',\hspace{0.4mm} \pmb{w}_{v_k}'' \right\rangle.
\end{align*}
Accordingly, $\left\langle \psi\hspace{0.4mm} \pmb{w}_x' - \pmb{w}_x'',\pmb{w}_{v_k}'' \right\rangle = 0$, which means
$$
\big( \psi\hspace{0.4mm} \pmb{w}_x' - \pmb{w}_x'' \big) \perp \pmb{w}_{v_k}'', ~ \forall k \in [d].
$$
Since those vectors $\pmb{w}_{v_1}'',\ldots,\pmb{w}_{v_d}''$ constitute a basis of $W''$, the vector $\psi\hspace{0.4mm} \pmb{w}_x' - \pmb{w}_x''$ must vanish, so $\psi\hspace{0.4mm} \pmb{w}_x' = \pmb{w}_x''$.

\vspace{3mm}
Finally, we show the relation that $iii) \Rightarrow i)$. The only thing we need to do is to verify that $\sigma\circ\phi = \phi\circ\sigma$, $\forall \sigma\in\mathfrak{G}$. Let us begin with a particular group of vectors $\{ \pmb{w}_v' \mid \pmb{w}_v' =  \frac{1}{ \| \dproj{W'}( \pmb{e}_ v ) \| } \dproj{W'}( \pmb{e}_ v )$ and $v \in V(G) \}$. In accordance with the definition to the linear map $\phi$, we have
$$
\sigma\circ\phi \big( \pmb{w}_v' \big) 
= \sigma \big(  \pmb{w}_v'' \big) 
= \pmb{w}_{\sigma v}''
= \phi \big( \pmb{w}_{\sigma v}' \big) 
= \phi\circ\sigma \big( \pmb{w}_{v}' \big).
$$ 

Note that $W'$ is irreducible according to our assumption, so $W'$ is spanned by the group $\{ \pmb{w}_v' \mid \pmb{w}_v' =  \frac{1}{ \| \dproj{W'}( \pmb{e}_ v ) \| } \dproj{W'}( \pmb{e}_ v )$ and $v \in V(G) \}$. Thus we can assume that $\pmb{w}_{v_1}',\ldots,\pmb{w}_{v_d}'$ constitute a basis of $W'$. On the other hand, the map $\phi$ is an isomorphism between two subspaces $W'$ and $W''$ due to our assumption, so $\pmb{w}_{v_1}'',\ldots,\pmb{w}_{v_d}''$ comprise a basis of $W''$. 

Let $\pmb{x}$ be a vector in $W'$. Suppose $\pmb{x} = \sum_{k=1}^d x_k \pmb{w}_{v_k}'$. Then 
\begin{align*}
\sigma\circ\phi \left( \pmb{x} \right) 
& = \sigma\circ\phi \left( \sum_{k} x_k \pmb{w}_{v_k}' \right)  \\
& = \sigma \left( \sum_{k} x_k \phi\hspace{0.4mm} \pmb{w}_{v_k}' \right)  \\
& = \sigma \left( \sum_{k} x_k \pmb{w}_{v_k}'' \right)  \\
& = \sum_{k} x_k \pmb{w}_{\sigma v_k}''  \\
& = \sum_{k} x_k \cdot \phi\hspace{0.4mm} \pmb{w}_{\sigma v_k}'  \\
& = \sum_{k} x_k \cdot \phi\circ\sigma \hspace{0.4mm} \pmb{w}_{v_k}'  \\
& = \phi\circ\sigma \left( \pmb{x} \right).  \\
\end{align*}
Therefore, our claim follows.
\end{proof}

\begin{Corollary}[\bf Isomorphism Theorem - General Version]\label{Corollary-IsomorphismThmBetweenIRs}
Let $G$ be a vertex-transitive graph with automorphism group $\mathfrak{G}$ and let $W'$ and $W''$ be two irreducible representations of $\mathfrak{G}$. Suppose $T$ is one of orbits of $\mathfrak{G}$ such that none of vectors $\dproj{W'}( \pmb{e}_t )$ and $\dproj{W''}( \pmb{e}_t )$ is trivial, $t\in T$. Then the following statements are equivalent.

\begin{enumerate}
\item[i)] $W'$ and $W''$ are isomorphic representations;

\item[ii)] The angle between $\dproj{W'}(\pmb{e}_u)$ and $\dproj{W'}(\pmb{e}_v)$ is equal to the angle between $\dproj{W''}(\pmb{e}_u)$ and $\dproj{W''}(\pmb{e}_v)$, $\forall u,v \in T$, i.e.,
$$
\left\langle \pmb{w}_u',\pmb{w}_v' \right\rangle = \left\langle \pmb{w}_u'',\pmb{w}_v'' \right\rangle,
$$
where $\pmb{w}_x' = \frac{1}{\| \dproj{W'}(\pmb{e}_x) \|}\dproj{W'}(\pmb{e}_x)$ and $\pmb{w}_x'' = \frac{1}{\| \dproj{W'}(\pmb{e}_x) \|}\dproj{W''}(\pmb{e}_x)$, $x = u$ or $v$.

\item[iii)] The map $\phi : \pmb{w}_{v}' \mapsto \pmb{w}_{v}''$, $v\in V(G)$, is an invertiblly linear map.

\end{enumerate}
\end{Corollary} 

We are now ready to show how to decompose every eigenspace of $\dAM{G}$ into IRs of $\mathfrak{G}$. There are in general two steps: splitting each eigenspace by orbits of $\mathfrak{G}$ and then decomposing those smaller subspaces into IRs of $\mathfrak{G}$. More precisely, we first separate the subspace $X_{\lambda} := \drm{span}{ \{ V_{\lambda} \mid T \} }$, which is spanned by a group of vectors $\{ \dproj{ V_{\lambda} }( \pmb{e}_x ) \mid x \in T \}$ where $T$ is an orbit of $\mathfrak{G}$, from the eigenspace $V_{\lambda}$, and then decompose $X_{\lambda}$ into IRs by virtue of a partition $\Pi_s^*$ and its adequate set $E(s)$ where $s$ is a vertex in $T$. 

In fact, if $W_{\lambda,p}$ is an irreducible representation of $\mathfrak{G}$ contained in $X_{\lambda}$, it is clear that the subspace $\drm{span}{ \{ \dproj{W_{\lambda,p}}( \pmb{e}_x ) : x \in T \} }$ is equal to $W_{\lambda,p}$. As a result, to single out $W_{\lambda,p}$ from $X_{\lambda}$ we only need to focus on the orbit $T$, and therefore we could make a further assumption that the action of $\mathfrak{G}$ on $V$ is transitive. Then there are two cases: the point stabilizer $\mathfrak{G}_s$ is trivial or not. Let us first consider the 2nd case, {\it i.e.,} $\mathfrak{G}_s \supsetneq \{1\}$.

In accordance with Lemma \ref{Lemma-BlockCHARACTERISEDbyProj}, for any block system there are some of IRs representing the system. Recall that $V_{\lambda}[s]$ is the subspace in $V_{\lambda}$ spanned by those vectors $\pmb{x}$ such that $\xi\hspace{0.5mm} \pmb{x} = \pmb{x}$, $\forall \hspace{0.5mm} \xi \in \mathfrak{G}_s$. It is apparent that if $\hat{B}$ is one of maximal blocks for $\mathfrak{G}$ then the subspace $\cap_{b \in \hat{B}} V_{\lambda}[b]$ must be contained in those $\mathfrak{G}$-invariant subspaces which can be expressed as a direct sum of IRs representing the system $\hat{\mathscr{B}}$ containing $\hat{B}$. Since we could find out all block families of $\mathfrak{G}$ (see the 2nd section for details), we can separate those $\mathfrak{G}$-invariant subspaces, representing distinct block systems, from $V_{\lambda}$ one by one in this way. Hence we can assume that the subspace $X_{\lambda} \subseteq V_{\lambda}$ we now deal with does not contain any IR representing any block system of $\mathfrak{G}$.

Before splitting the subspace $X_{\lambda}$, we can determine the number of IRs of $\mathfrak{G}$ contained in $X_{\lambda}$ by virtue of Lemma \ref{Lemma-IRsFeatureGvNonTrivial}. In order to split $X_{\lambda}$, we need to refine step by step the subspace $X_{\lambda}[s] := \{ \pmb{x} \in X_{\lambda} \mid \xi\hspace{0.5mm}\pmb{x} = \pmb{x} \hspace{2mm} \forall\hspace{0.5mm}\xi\in\mathfrak{G}_s \}$. In the first place, set $X_{\lambda}^1[s] = \{ \pmb{x} \in X_{\lambda}[s] : \| \pmb{x} \| = 1 \}$. Suppose $\mathfrak{G}_s$ possesses $t$ orbits and $\sigma_1,\ldots,\sigma_{t-1}$ are those permutations in $\mathfrak{G}$ moving $s$ to distinct orbit except the one only containing $s$. Then we can obtain a group of subspaces $X_{\lambda}^1[s] \supseteq X_{\lambda}^1[s;\sigma_1] \supseteq X_{\lambda}^1[s;\sigma_1,\sigma_2] \supseteq \cdots \supseteq X_{\lambda}^1[s;\sigma_1,\sigma_2,\ldots,\sigma_{t-1}]$ such that 
$$
X_{\lambda}^1[s;\sigma_1] = \{ \pmb{x} \in X_{\lambda}^1[s] : \langle \pmb{x},\sigma_{1}\pmb{x} \rangle \geq \langle \pmb{y},\sigma_{1}\pmb{y} \rangle \hspace{2mm} \forall \hspace{0.5mm} \pmb{y} \in X_{\lambda}^1[s] \},
$$
and
$$
X_{\lambda}^1[s;\sigma_1,\ldots,\sigma_{i+1}] = \{ \pmb{x} \in X_{\lambda}^1[s;\sigma_1,\ldots,\sigma_{i}] : \langle \pmb{x},\sigma_{i+1}\pmb{x} \rangle \geq \langle \pmb{y},\sigma_{i+1}\pmb{y} \rangle \hspace{2mm} \forall \hspace{0.5mm} \pmb{y} \in X_{\lambda}^1[s;\sigma_1,\ldots,\sigma_{i}] \},
$$
where $i = 1,\ldots,t-2$. 

In accordance with Theorem \ref{Thm-IsomorphismThmBetweenIRs}, the subspace $X_{\lambda}^1[s;\sigma_1,\ldots,\sigma_{t-1}]$ must be contained in some irreducible representation of $\mathfrak{G}$, and thus we can figure out the irreducible representation for it is equal to $\drm{span}{ \{ \gamma \hspace{0,5mm} \pmb{x} : \gamma \in E(s) \} }$. In the same way, we can separate one by one IRs from $X_{\lambda}$ and eventually decompose it as a direct sum of IRs.

It is clear that there is only one kind of operation - inner product. As a result, it is easy to see that we can figure out the subspace $X_{\lambda}^1[s;\sigma_1,\ldots,\sigma_{t-1}]$ within time $n^C$ for some constant $C$.

As to the case that $\mathfrak{G}_s$ is trivial for any $s \in V$, we replace $X_{\lambda}[s]$ with 
$$
X_{\lambda}[B] = \left\{ \pmb{x} \in X_{\lambda} : \zeta \left( \sum_{k=1}^{|B|} \xi^k \right) \pmb{x} 
= \left( \sum_{k=1}^{|B|} \xi^k \right) \pmb{x}, 
~~\forall\xi\in \mathfrak{G}_B\setminus\{1\} \mbox{ and } \zeta\in\mathfrak{G}_B  \right\},
$$
where $B$ is one of minimal blocks for $\mathfrak{G}$, and then use the same way to split $X_{\lambda}$ into IRs of $\mathfrak{G}$.

\begin{Lemma}\label{LemIsomorphism} If $\phi$ is an isomorphism between two IRs $W_1$ and $W_2$ of $\mathfrak{G}$ on $\R^n$, then the subspace $\mathrm{span} \{ \sigma\big(\pmb{w}_1 + \phi \pmb{w}_1\big) \mid \sigma\in\mathfrak{G} \}$ is isomorphic to $W_1$ where $\pmb{w}_1$ is a non-trivial vector in $W_1$.
\end{Lemma}

By combing Theorem \ref{Thm-IsomorphismThmBetweenIRs} and Lemma \ref{LemIsomorphism}, one could deduce an interesting conclusion that in the case of being transitive $\mathfrak{G}$ possesses no two IRs which are isomorphic to one another in $\R^n$ with respect to the permutation representation we consider in the paper. As a matter of fact, the reason why $\mathfrak{G}$ possesses a number of IRs isomorphic to one another is that the ways they represent the action of $\mathfrak{G}$ are to some extant different. The corollary \ref{Cor-SpanInvariantSubspaces-Isomorphic} below reveals this relation definitely.

\begin{proof}[\bf Proof] Since $W_1$ is an irreducible representation, one can choose a group of permutations  $\sigma_1=1,\ldots,\sigma_d$ of $\mathfrak{G}$ such that $\sigma_1\pmb{w}_1=\pmb{w}_1,\ldots,\sigma_d\pmb{w}_1$ form a basis of $W_1$. It is plain to verify that $\sigma_1\pmb{w}_1+\sigma_1(\phi\pmb{w}_1),\ldots,\sigma_d\pmb{w}_1 + \sigma_d(\phi\pmb{w}_1)$ form a basis of the subspace $\mathrm{span} \{ \sigma\big(\pmb{w}_1 + \phi \pmb{w}_1 \big) \mid \sigma\in\mathfrak{G} \}$. For that reason, there is natually a 
linear isomorphism $\psi$ between the two subspaces $W_1$ and $\mathrm{span} \{ \sigma\big(\pmb{w}_1 + \phi \pmb{w}_1 \big) \mid \sigma \}$ such that 
$$
\psi:\sigma_i\pmb{w}_1 \mapsto 
\sigma_i\pmb{w}_1 + \sigma_i(\phi\pmb{w}_1), ~~ i=1,\ldots,d.
$$

As a matter of fact, it is not difficult to check that $\psi$ is an isomorphism between those two representations. Let $\pmb{w}$ be a vector in $W_1$. Suppose $\pmb{w}=\sum_{i=1}^d x_i\cdot\sigma_i\pmb{w}_1$ where $x_i$ is the coordinate of $\pmb{w}$ with respect to $\sigma_i\pmb{w}_1$ ($i=1,\ldots,d$). Picking arbitrarily a permutation $\gamma$ from $\mathfrak{G}$, we assume $\gamma \hspace{0.4mm} \pmb{w}=\sum_{i=1}^d y_i\cdot\sigma_i\pmb{w}_1$. Consequently,
\begin{align*}
\psi\circ\gamma( \pmb{w} ) 
& = \psi\left[\sum_{i=1}^d y_i\cdot\sigma_i\pmb{w}_1\right] \\
& = \sum_{i=1}^d y_i\cdot\psi(\sigma_i\pmb{w}_1)\\
& = \sum_{i=1}^d y_i\left(
\sigma_i\pmb{w}_1 + \sigma_i(\phi\pmb{w}_1)
\right).
\end{align*}
On the other hand,
\begin{align*}
\gamma\circ\psi(\pmb{w}) 
& = \gamma\circ\psi\left[\sum_{i=1}^d x_i\cdot\sigma_i\pmb{w}_1\right]\\
& = \gamma\left[\sum_{i=1}^d x_i\cdot\psi( \sigma_i\pmb{w}_1 ) \right] \\
& = \gamma\left[\sum_{i=1}^d x_i\left(\sigma_i\pmb{w}_1 + \sigma_i (\phi\pmb{w}_1) \right)\right]\\
& = \gamma\left[\sum_{i=1}^d x_i\cdot\sigma_i\pmb{w}_1\right] + \gamma\left[\sum_{i=1}^d x_i\cdot\sigma_i (\phi\pmb{w} _1) \right] \\
& = \sum_{i=1}^d y_i\cdot\sigma_i\pmb{w}_1 + \phi\circ\gamma\left[\sum_{i=1}^d x_i\cdot\sigma_i\pmb{w}_1\right]\\
& = \sum_{i=1}^d y_i\cdot\sigma_i\pmb{w}_1 + \sum_{i=1}^d y_i\cdot\sigma_i (\phi\pmb{w}_1).
\end{align*}

Accordingly, $\psi\gamma=\gamma\psi$ for any $\gamma\in \mathfrak{G}$, and thus $\psi$ is an isomorphism between the two representations $W_1$ and $\mathrm{span} \{ \sigma\big(\pmb{w}_1 + \phi \pmb{w}_1 \big) \mid \sigma \}$.
\end{proof}

Apparently, one can employ the same idea to establish a more general result concerning a number of IRs isomorphic to one another rather than two IRs. To be precise, if there are a group of IRs $W, W_1, W_2,\ldots,W_n$ of $\mathfrak{G}$ such that $\phi_p : W \rightarrow W_p $ is an isomorphism between two representations, $p = 1,\ldots,n$, then the subspace $\mathrm{span} \{ \sigma\big(\pmb{w} + \sum_{p=1}^n  \phi_p \pmb{w}\big) \mid \sigma\in\mathfrak{G} \}$ is isomorphic to $W$ where $\pmb{w}$ is a non-trivial vector in $W$.

\begin{Corollary}\label{Cor-DecomposeInvariantSubspaces}
Let $\mathfrak{G}$ be a permutation group with orbits $T_1,\ldots,T_s$ and let $U$ be an $\mathfrak{G}$-invariant subspace with a decomposition $\oplus W_p$ into IRs of $\mathfrak{G}$. Then $\mathrm{span} \hspace{0.6mm} \{ \dproj{U}( \pmb{e}_{t_i} ) \mid t_i \in T_i \} = U$ for some $i \in [s]$ if and only if $\dproj{W_p}( \pmb{e}_{t_i} ) \neq \pmb{0}$ $\forall p$ and any two of IRs involved are not isomorphic to each other.
\end{Corollary}

Suppose $W_1,\ldots,W_k$ are a group of IRs of $\mathfrak{G}$ and $T_1,\ldots,T_k$ are some of orbits of $\mathfrak{G}$, which are said to be {\it independent orbits with respect to $W_1,\ldots,W_k$} if the group of vectors 
$$
\{ \left( \left\| \dproj{W_p}( \pmb{e}_{t_1} ) \right\|,\ldots,\left\| \dproj{W_p}( \pmb{e}_{t_k} ) \right\| \right) \mid p = 1,\ldots,k \} \mbox{ are linearly independent,}
$$
where $t_i$ belong to the orbit $T_i$, $i=1,\ldots,k$.

\begin{Corollary}\label{Cor-SpanInvariantSubspaces-Isomorphic}
Let $\mathfrak{G}$ be a permutation group and let $U$ be an $\mathfrak{G}$-invariant subspace with a decomposition $\oplus_{p=1}^k W_p$ into IRs of $\mathfrak{G}$ such that any two of them are isomorphic to one another. Suppose $T_1,\ldots,T_k$ are some of orbits of $\mathfrak{G}$. Then $\mathrm{span} \hspace{0.6mm} \{ \dproj{U}( \pmb{e}_{t_i} ) \mid t_i \in T_i, i=1,\ldots,k \} = U$ if and only if $T_1,\ldots,T_k$ are independent.
\end{Corollary}

Recall that in the 1st section we employ two apparatuses based on an equitable partition $\Pi$ to split subspaces in a decomposition of $\R^n$: those eigenspaces of $\dAM{ G / \Pi }$ and those subspaces spanned by one of cells of $\Pi$. These two corollaries above show us why the second way is often effective.

\section{The algorithm}

As shown in sections 2 and 4, provided that we are given a partition $\Pi_s^*$, consisting of orbits of $\mathfrak{G}_s$, and the adequate set $E(s)$ associated with the vertex $s$, belonging to some orbit $T$ of $\mathfrak{G}$, we can find out all block systems of $\mathfrak{G}$ contained in $T$ and decompose $\bigoplus_{ \lambda\in\drm{spec}{\dAM{G}} } \hspace{0.6mm} \drm{span}{\{ V_{\lambda} : T \}}$ into a direct sum of IRs of $\mathfrak{G}$, where $\drm{span}{\{ V_{\lambda} : T \}} = \drm{span}{ \{ \dproj{V_{\lambda}}( \pmb{e}_x ) : x \in T \} }$. By dealing with orbits one by one, we can eventually find out all block systems of $\mathfrak{G}$ and decompose $\bigoplus V_{\lambda}$ into a direct sum of IRs of $\mathfrak{G}$. In this last section, let us present the algorithm showing how to decide if two vertices are symmetric or not and how to figure out an automorphism relevant.

As we have illustrated with two examples in section 1, it is the {\it dist}. of OPSB in a decomposition of $\R^n$ that really matters in determining whether or not two vertices are symmetric and in figuring out one automorphism moving one of vertices to another if they are symmetric. There is however a big obstacle to analyzing the {\it dist}. of OPSB in $\oplus X_{\lambda,k}$ \footnote{ To keep our description succinct, we here say ``the {\it dist}. of OPSB in $\oplus X_{\lambda,k}$'' rather than ``the {\it dist}. of OPSB in each subspace $X_{\lambda,k}$''.} $-$ the dimensions of those subspaces involved. As one can imagine, if some $X_{\lambda,k}$ has a large dimension the {\it dist}. of OPSB in $X_{\lambda,k}$ could be a real mess, so one cannot clarify the symmetries in $X_{\lambda,k}$ efficiently. As a result, we use a series of equitable partitions of $G$ to split those subspaces of large dimension into smaller pieces.

Let $\Pi$ and $\widetilde{\Pi}$ be two equitable  partitions of $G$ such that $\Pi \prec \widetilde{\Pi}$.\footnote{ By ``$\Pi \prec \widetilde{\Pi}$'', we mean $\Pi$ is refined properly by $\widetilde{\Pi}$, {\it i.e.,} there is at least one cell of $\Pi$ which is comprised of two or more cells of $\widetilde{\Pi}$. } Set $X_{\lambda} = V_{\lambda} \ominus \dChMat{\Pi} V_{\lambda}^{ G / \Pi }$ for every $\lambda \hspace{0.5mm} \in \hspace{0.5mm} \drm{spec}{\dAM{G / \Pi}}$. It is clear that if $\lambda \hspace{0.5mm} \in \hspace{0.5mm}  \drm{spec}{\dAM{G / \widetilde{\Pi}}} \cap \drm{spec}{\dAM{G / \Pi}}$ and the multiplicity of $\lambda$ with respect to $\dAM{G / \widetilde{\Pi}}$ is bigger than that with respect to $\dAM{G / \Pi}$, then $X_{\lambda}$ must be split by $\dChMat{\widetilde{\Pi}} V_{\lambda}^{ G / \widetilde{\Pi} }$. This is the first way of splitting subspaces by virtue of equitable partitions.

There is another way fortunately and the assertion below reveals how it works. Set $\mathscr{R}(\Pi) = \{ \widetilde{\Pi} \mbox{ is an EP of } G : \Pi \preceq \widetilde{\Pi} \}$ and denote by $\Pi|_C$ the partition of $C$ induced by $\Pi$, {\it i.e.,} by restricting $\Pi$ to $C$ one can have a partition of the cell.

\begin{Lemma}\label{Lem-SplittingEigSpaces}
Let $\Pi$ be an equitable partition of $G$ and let $C$ be a cell of $\Pi$. If there exists one equitable partition $\widetilde{\Pi}$ in $\mathscr{R}(\Pi)$ such that $\Pi \prec \widetilde{\Pi}$ and $\Pi|_C = \widetilde{\Pi}|_C$ then $\drm{span}{ \{ X_{\lambda} : C \} } \subsetneq X_{\lambda}$ for any eigenvalue $\lambda$ of $\dAM{G / \widetilde{\Pi}}$ the multiplicity of which with respect to $\dAM{G / \widetilde{\Pi}}$ is bigger than that with respect to $\dAM{G / \Pi}$, where $X_{\lambda} = V_{\lambda} \ominus \dChMat{\Pi} V_{\lambda}^{ G / \Pi }$ and $\drm{span}{ \{ X_{\lambda} : C \} } = \drm{span}{\{ \dproj{X_{\lambda}}( \pmb{e}_x ) : x \in C \}}$.
\end{Lemma}
\begin{proof}[\bf Proof]
According to our assumption, there are two possible cases for eigenvalues of $\dAM{G / \widetilde{\Pi}}$ concerned: 
$$
\lambda \hspace{0.4mm} \in \hspace{0.4mm} \drm{spec}{ \dAM{G/\Pi} } \mbox{ and } 
\dChMat{\Pi} V_{\lambda}^{ G / \Pi } \subsetneq \dChMat{\widetilde{\Pi}} V_{\lambda}^{ G / \widetilde{\Pi} },
\mbox{ or } \lambda \hspace{0.4mm} \notin \hspace{0.4mm} \drm{spec}{ \dAM{G/\Pi} }.
$$

In the first case, suppose $\pmb{r}_1$ is a vector in $\dChMat{\Pi} V_{\lambda}^{ G / \Pi }$ and $\pmb{r}_2$ is in $\dChMat{\widetilde{\Pi}} V_{\lambda}^{ G / \widetilde{\Pi} } \ominus \dChMat{\Pi} V_{\lambda}^{ G / \Pi }$. Then $\langle \pmb{r}_2,\pmb{r}_1 \rangle = 0$. Recall that given an equitable partition $\Pi$ the eigenvectors of $\dAM{G}$ can be divided into two classes: those that are constant on every cell of $\Pi$ and those that sum to zero on each cell of $\Pi$. Because the cell $C$ is common to $\Pi$ and $\widetilde{\Pi}$, the $x$th coordinate of $\pmb{r}_2$ must be 0 for any $x\in C$. On the other hand, $\widetilde{\Pi}$ properly refines $\Pi$, so $\pmb{r}_2$ cannot be trivial and thus belongs to $X_{\lambda} \ominus \drm{span}{ \{ X_{\lambda} : C \}}$. 

By the same argument, one can easily prove the assertion in the second case.
\end{proof}

It is worthwhile pointing out the difference between two methods above used to split subspaces of $\R^n$. In order to use the 1st method, we have to figure out an EP $\widetilde{\Pi}$ at first, which is refiner than the original one $\Pi$. However as long as there exists such an EP $\widetilde{\Pi}$ enjoying the requirement of Lemma \ref{Lem-SplittingEigSpaces}, the 2nd method is effective. 

On the other hand, we can also use a decomposition of $\R^n$ to analyze EPs. Let $\mathfrak{H}$ be a subgroup of $\drm{Aut}{G}$ and let $\Pi$ be an EP of $G$ such that any orbit of $\mathfrak{H}$ is contained entirely in one cell of $\Pi$. Suppose $\bigoplus_{\lambda\hspace{0.4mm}\in\hspace{0.4mm}\drm{spec}{ \dAM{G} }} \hspace{0.4mm} \left( \dChMat{\Pi} V_{\lambda}^{G / \Pi} \oplus X_{\lambda} \right)$ is the decomposition of $\R^n$ obtained according to $\Pi$, where $X_{\lambda}$ is the orthogonal complement of $\dChMat{\Pi} V_{\lambda}^{G / \Pi}$ in $V_{\lambda}$. Two cells $C_1$ and $C_2$ of $\Pi$ are said to be {\it irrelevant with respect to $\mathfrak{H}$} if $\forall \hspace{0.5mm} w_1 \in C_1$, $C_2 \in \dEP{w_1}{ \left( \dChMat{\Pi} V_{\lambda}^{G / \Pi} \oplus X_{\lambda} \right) }$, or $\forall \hspace{0.5mm} w_2 \in C_2$, $C_1 \in \dEP{w_2}{ \left( \dChMat{\Pi} V_{\lambda}^{G / \Pi} \oplus X_{\lambda} \right) }$.

Suppose that for any $w_1$ in $C_1$, $C_2$ is a cell of $\dEP{w_1}{ \left( \dChMat{\Pi} V_{\lambda}^{G / \Pi} \oplus X_{\lambda} \right) }$. In accordance with Lemma \ref{Lem-SplittingEigSpaces}, there must exist $\lambda \in \drm{spec}{ \dAM{G} }$ with non-trivial subspace $U_{\lambda} := X_{\lambda} \ominus \drm{span}{\{ X_{\lambda} : C_2 \}}$. Suppose further that $Y_{\lambda}$ is the subspace of $\drm{span}{\{ X_{\lambda} : C_2 \}}$ spanned by those vectors that are constant on the cell $C_2$ and $Z_{\lambda} = \drm{span}{\{ X_{\lambda} : C_2 \}} \ominus Y_{\lambda}$. In summary, we may decompose every eigenspace $V_{\lambda}$ in three  steps: 
$$
V_{\lambda} = \dChMat{\Pi} V_{\lambda}^{G / \Pi} \oplus X_{\lambda}  = \dChMat{\Pi} V_{\lambda}^{G / \Pi} \oplus \left(\drm{span}{\{ X_{\lambda} : C_2 \}} \oplus U_{\lambda}\right) = \dChMat{\Pi} V_{\lambda}^{G / \Pi} \oplus \left(Y_{\lambda} \oplus Z_{\lambda}\right) \oplus U_{\lambda}.
$$
Accordingly $\dproj{Z_{\lambda}}( \pmb{R}_{C_2} ) = \pmb{0}$, where $\pmb{R}_{C_2}$ is the characteristic vector of the cell $C_2$, and thus $\dproj{Z_{\lambda}}( \pmb{w}_{1} ) = \pmb{0}$ for any $w_1$ in $C_1$. Consequently, the structure of the action of $\mathfrak{H}$ on $C_1$ is represented by $\bigoplus_{\lambda\hspace{0.4mm}\in\hspace{0.4mm}\drm{spec}{ \dAM{G} }} \hspace{0.4mm} \left( Y_{\lambda} \oplus U_{\lambda} \right)$, while the structure of the action on $C_2$ is represented by $\bigoplus_{\lambda\hspace{0.4mm}\in\hspace{0.4mm}\drm{spec}{ \dAM{G} }} \hspace{0.4mm} Z_{\lambda}$. Hence the action of $\mathfrak{H}$ on $C_1$ has no effect on $C_2$ and vice versa. As a result, we can simultaneously deal with those two cells in order to clarify symmetries among vertices in $C_1$ and in $C_2$.

In addition, if there is a cell $C$ of $\Pi$ containing only a few vertices, we can first determine the symmetries represented by $\bigoplus_{\lambda\hspace{0.4mm}\in\hspace{0.4mm}\drm{spec}{ \dAM{G} }} \hspace{0.4mm} \drm{span}{ \{ V_{\lambda} : C \} }$, and then refine $\Pi$ by the partition of $C$ determined by $\oplus \hspace{0.5mm} \drm{span}{ \{ V_{\lambda} : C \} }$. As a result, we always first deal with those cells not being singleton but of minimum order. 

As a matter of fact, one decomposition of $\R^n$ is often helpful in refining an equitable partition and vice versa. Now let us present the details of the algorithm and functions relevant.

\floatname{algorithm}{Function}

\begin{algorithm}[H]
 \caption{ Revising and refining an equitable partition by a group of subspaces }      
\begin{algorithmic}[1]
\REQUIRE an equitable partition $\Pi = \{ C_1,\ldots,C_r \}$ of $G$ such that $|C_i| \leq |C_{i+1}|$, $i=1,\ldots,r-1$, and a direct sum $\oplus U_{\lambda,k}$ of $\R^n$ in which each term is a subspace of an eigenspace $V_{\lambda}$ and any two subspaces involved are orthogonal to one another
\ENSURE an ordered EP or {\it ``$\Pi$ is not an appropriate EP''}
\STATE set $i := 1$
\WHILE { $i \le r$ }
 \IF{ $|C_i| > 1$}
  \STATE in accordance with (\ref{TheBinaryRelation}), work out a family of equitable partitions $\{ \dEP{v}{U_{\lambda,k}} : v \in C_i \}$
  \STATE in accordance with (\ref{TheBinaryRelation-2}), work out a group of subsets $\drm{Part}{(C_i)}$ of $C_i$
 \ELSE
  \STATE set $\drm{Part}{(C_i)} := C_i$
 \ENDIF
 \STATE set $i := i + 1$
\ENDWHILE
\STATE figure out the coarsest EP $\Pi'$ of the partition $\{ \drm{Part}{(C_i)} : i = 1,\ldots,r \}$
\IF{ $\Pi \preceq \Pi'$ }
\STATE order cells of $\Pi'$ according to their sizes
\RETURN the ordered EP $\Pi'$
\ELSE
\RETURN {\it ``$\Pi$ is not an appropriate EP''}
\ENDIF
\end{algorithmic}
\end{algorithm}

\newpage

\begin{algorithm}[H]
 \caption{\small Splitting subspaces by eigenspaces of $\dAM{ G / \Pi }$ }      
\begin{algorithmic}[1] \small
\REQUIRE an equitable partition $\Pi$ of $G$ and a direct sum $\bigoplus_{ 1\leq i \leq t \atop 1\leq k \leq M_i} U_{\lambda_i,k}$ of $\R^n$ such that each term is a subspace of an eigenspace $V_{\lambda}$ and any two subspaces involved are orthogonal to one another
\ENSURE a direct sum of $\R^n$ 
\STATE set $i := 1$
\WHILE { $i \le t$ }
  \STATE set $k := 1$
  \WHILE { $k \leq M_i$ }
    \IF{ $\{ \pmb{0} \} \subsetneq U_{\lambda_i,k} \cap \dChMat{\Pi} V_{\lambda}^{ G / \Pi}$}
    \STATE set $Y_{\lambda_i,k} := U_{\lambda_i,k} \cap \dChMat{\Pi} V_{\lambda}^{ G / \Pi}$ and $Z_{\lambda_i,k} := U_{\lambda_i,k} \ominus Y_{\lambda_i,k}$
    \ELSE
    \STATE set $Y_{\lambda_i,k} := U_{\lambda_i,k}$
    \ENDIF
  \STATE set $k := k + 1$
  \ENDWHILE
\STATE set $i := i + 1$
\ENDWHILE
\RETURN $\oplus \big( Y_{\lambda_i,k} \oplus Z_{\lambda_i,k} \big)$ 
\end{algorithmic}
\end{algorithm}

\begin{algorithm}[H]
 \caption{\small Splitting subspaces by cells of an equitable partition }      
\begin{algorithmic}[1] \small
\REQUIRE an equitable partition $\Pi = \{ C_1,\ldots,C_r \}$ of $G$ such that $|C_i| \leq |C_{i+1}|$, $i=1,\ldots,r-1$, and a direct sum $\bigoplus_{ 1\leq i \leq t \atop 1\leq k \leq M_i} U_{\lambda_i,k}$ of $\R^n$ such that each term is a subspace of the eigenspace $V_{\lambda}$ and any two subspaces involved are orthogonal to one another
\ENSURE a direct sum of $\R^n$ 
\STATE set $s := 1$ and $\bigoplus_{ 1\leq i \leq t \atop 1\leq k \leq M_i} X_{\lambda_i,k} := \bigoplus_{ 1\leq i \leq t \atop 1\leq k \leq M_i} U_{\lambda_i,k}$
\WHILE { $s \le r$ }
  \STATE set $i := 1$
  \WHILE { $i \leq t$ }
    \STATE set $k := 1$
    \WHILE { $k \leq M_i$ }
      \IF{ $\drm{span}{ \{ X_{\lambda_i,k} : C_s \} } \subsetneq X_{\lambda_i,k} $}
      \STATE set $Y_{\lambda_i,k} := \drm{span}{ \{ X_{\lambda_i,k} : C_s \} }$ and $Z_{\lambda_i,k} := X_{\lambda_i,k} \ominus Y_{\lambda_i,k}$
      \ELSE
      \STATE set $Y_{\lambda_i,k} := X_{\lambda_i,k}$
      \ENDIF
      \STATE set $k := k + 1$
    \ENDWHILE
    \STATE set $M_i^* :=$ the number of non-trivial subspaces in the sum $\bigoplus_{1 \leq k \leq M_i} \big( Y_{\lambda_i,k} + Z_{\lambda_i,k} \big)$
    \STATE set $i := i + 1$
  \ENDWHILE
  \STATE set $\bigoplus_{ 1\leq i \leq t \atop 1\leq k \leq M_i^*} X_{\lambda_i,k} := \bigoplus_{ 1\leq i \leq t \atop 1\leq k \leq M_i} \big( Y_{\lambda_i,k} \oplus Z_{\lambda_i,k} \big)$ 
  \STATE set $M_i := M_i^*$ and $s := s + 1$
\ENDWHILE
\RETURN $\bigoplus_{ 1\leq i \leq t \atop 1\leq k \leq M_i} X_{\lambda_i,k}$ 
\end{algorithmic}
\end{algorithm}

\newpage

\floatname{algorithm}{Algorithm}

\begin{algorithm}[H]    
 \caption{\small Analyzing a graph with those three functions }      
\begin{algorithmic}[1] \small
\REQUIRE a graph $G$ of order $n$ and an integer $K$
\ENSURE a set $\mathscr{P}$ of pairs, consisting of an EP of $G$ and a decomposition of $\R^n$, and an integer $d$, indicating the number of vertices we fix in the process of figuring out $\mathscr{P}$
\STATE set $\Pi := \{ V(G) \}$ and  $\oplus X_{\lambda,k} := \oplus V_{\lambda,k}$ 
\STATE set $\mathscr{P} := \{ (\Pi,\oplus X_{\lambda,k}) \}$ and $d := 0$
\WHILE { $\exists$ a cell $C$ of $\Pi$ and a subspace $X_{\lambda,k}$ {\it s.t.,} $|C| > K$ and $\dim X_{\lambda,k} >K$ }
  \STATE set $\widetilde{\Pi} := \drm{F1}{ (\oplus X_{\lambda,k}, \Pi) }$
  \IF { $\Pi \prec \widetilde{\Pi}$ }
    \STATE set $\Pi := \widetilde{\Pi}$ and $\oplus X_{\lambda,k} := \drm{F3}{ ( \drm{F2}{ (\oplus X_{\lambda,k},\Pi) }, \Pi) }$
  \ELSE
    \STATE set $\mathscr{P} := \mathscr{P} \cup \{ (\Pi,\oplus X_{\lambda,k}) \}$
    \STATE set $i := 1$ and $i^* = 0$
    \WHILE { $i \leq |\Pi|$ and $i \neq i^*$ }
      \IF { $|C_i| > 1$, where $C_i$ is a cell of $\Pi$, }
        \STATE set $i^* := i$
      \ELSE
        \STATE set $i := i + 1$
      \ENDIF
    \ENDWHILE
    \WHILE { $|C_{i^*}| \leq K$ }
      \STATE examine the {\it dist}. of OPSB in $\oplus \hspace{0.5mm} \drm{span}{ \{ X_{\lambda,k} : C_{i^*} \} }$ and then determine the structure of the action of the group relevant on $C_{i^*}$
      \STATE set $i^* := i^* +1$ 
    \ENDWHILE
    \STATE select one vertex $x$ from $C_{i^*}$ and set $d := d+1$
    \STATE set $\Pi := \dEP{v}{ X_{\lambda,k} }$ and $\oplus X_{\lambda,k} := \drm{F3}{ ( \drm{F2}{ (\oplus X_{\lambda,k},\Pi) }, \Pi) }$
    \STATE set $\mathscr{P} := \mathscr{P} \cup \{ (\Pi,\oplus X_{\lambda,k}) \}$
  \ENDIF
\ENDWHILE
\STATE order $\mathscr{P}$ according to EPs and decompositions of $\R^n$
\RETURN $\mathscr{P}$ and $d$
\end{algorithmic}
\end{algorithm}

In lines 4 and 6, F$\hspace{0.2mm}k$ stands for {\bf Function}$\hspace{0.3mm}k$, $k=1,2,3$. Line 4 shows how we use a decomposition of $\R^n$ to refine an equitable partition, while line 6 shows another direction using an equitable partition to split a decomposition. 

Let $\left( \Pi^A,\oplus X_{\lambda,k}^A \right)$ and $\left( \Pi^B,\oplus X_{\lambda,k}^B \right)$ be two pairs in $\mathscr{P}$. We say $\left( \Pi^A,\oplus X_{\lambda,k}^A \right)$ is a {\it predecessor} of $\left( \Pi^B,\oplus X_{\lambda,k}^B \right)$, denoted by $\left( \Pi^A,\oplus X_{\lambda,k}^A \right) \leq \left( \Pi^B,\oplus X_{\lambda,k}^B \right)$, if $\Pi^A \prec \Pi^B$, or $\Pi^A = \Pi^B$ and $\forall\hspace{0.4mm} X_{\lambda,k}^A$, $\exists$ a group of subspaces $X_{\lambda,i_1}^B,\ldots,X_{\lambda,i_l}^B$ {\it s.t.,} $X_{\lambda,k}^A = \oplus_{j=1}^l X_{\lambda,i_j}^B$. In line 26, we order the set $\mathscr{P}$ in this way.

Let $x$ be a vertex of $G$. We color the vertex $x$ red and denote the resulted graph by $G_x$. Obviously, the problem of determining whether or not two vertices $u$ and $v$ in $G$ are symmetric and of figuring out one automorphism from $u$ to $v$ is equivalent respectively to that of determining whether or not two graphs $G_u$ and $G_v$ are isomorphic and of figuring out one isomorphism from $G_u$ to $G_v$. Consequently, we only show here how to use the algorithm above to deal the problem of Graph Isomorphism and in the case being isomorphic to figure out one isomorphism. 

Suppose the elements of the set $\mathscr{P}^G$ associated with $G$ are listed as follows:
$$
\left( \{ V(G) \},\oplus V_{\lambda}^G \right) = \left( \Pi_1^G,\oplus X_{\lambda,1,k}^G \right)
\leq \left( \Pi_2^G,\oplus X_{\lambda,2,k}^G \right)
\leq \cdots \leq
\left( \Pi_t^G,\oplus X_{\lambda,t,k}^G \right).
$$
As shown in the section 1, $G \cong H$ if and only if by conducting the same operation on $H$, we can obtain a set 
$$
\mathscr{P}^H = \big\{ \hspace{0.6mm}
\left( \{ V(H) \},\oplus V_{\lambda}^H \right) = \left( \Pi_1^H,\oplus X_{\lambda,1,k}^H \right)
\leq \left( \Pi_2^H,\oplus X_{\lambda,2,k}^H \right)
\leq \cdots \leq
\left( \Pi_t^H,\oplus X_{\lambda,t,k}^H \right)
\hspace{0.6mm} \big\}
$$
such that $\Pi_p^G \asymp \Pi_p^H$ and $G/\Pi_p^G \cong H/\Pi_p^H$, $p=1,\ldots,t$, and the {\it dist}. of OPSB in $\oplus X_{\lambda,t,k}^G$ is the same as that in $\oplus X_{\lambda,t,k}^H$. As a result, one can determine whether or not $G$ is isomorphic to $H$ by {\bf Algorithm 4} and in the case being isomorphic figure out one isomorphism from $G$ to $H$.

\vspace{2mm}
At last, let us estimate the complexity of our algorithm. In working out the set $\mathscr{P}^G$, we only perform two kinds of operations: splitting a subspace by an equitable partition and refining an equitable partition by a decomposition of $\R^n$.

In the 1st kind, we need to calculate eigenvalues and eigenspaces of $\dAM{G/\Pi}$, the complexity of which (within a relative error bound $2^{-b}$) is bounded by $O(|\Pi|^3 + (|\Pi| \log^2 |\Pi|) \log b)$ (see \cite{PanChenZheng} for details). After that we use $\dChMat{\Pi} V_{\lambda}^{G/\Pi}$ and $\drm{span}{ \{ X_{\lambda,k} : C \} }$, where $C$ is a cell of $\Pi$, to decompose subspaces contained in $V_{\lambda}$. To accomplish those computations, one can employ the Schmidt orthogonalization. It is not difficult to check that there are at most $n^3$ objects needed to be dealt with.

In the 2nd kind, we need to figure out a family of EPs $\{ \dEP{ v }{ X_{\lambda,k} } : v \in V(G) \}$ according to the relation (\ref{TheBinaryRelation}), and then to determine a partition {\bf P} according to another relation (\ref{TheBinaryRelation-2}), and finally to figure out the coarsest EP of {\bf P}. To accomplish first two steps, we need to calculate norms of projections $\dproj{X_{\lambda,k}}( \pmb{e}_v )$, $\forall X_{\lambda,k}$ in the decomposition and $\forall \hspace{0.4mm} v \in V(G)$, angles between any two projections and projections of characteristic vectors $\dproj{X_{\lambda,k}}\left( \pmb{R}_{\dEP{ v }{ X_{\lambda,k} }} \right)$, $\forall \hspace{0.4mm} v \in V(G)$. The calculation involved here is the inner product. To accomplish the 3rd step, there are a number of efficient algorithm one can choose. It is readily to check that there are at most $n^3$ objects we need to deal with.

To sum up, we can work out $\mathscr{P}^G$ within time $n^C$ for some constant $C$.

Let $\left( \Pi^G,\oplus X_{\lambda,k} \right)$ be a member in $\mathscr{P}^G$. As shown in line 21 of the {\bf Algorithm 4}, if the EP $\Pi^G$ we have cannot split subspaces in $\oplus X_{\lambda,k}$ further but there are some subspaces of large dimension, we have to choose one vertex $x$ from a suitable cell $C_{i^*}^G$ to refine $\Pi^G$, that causes a big trouble in determining whether $G \cong H$, for despite the fact that we can find the EP $\Pi^H$ corresponding to $\Pi^G$ and identify the right cell $C_{i^*}^H$ of $\Pi^H$, we do not know in advance which vertex in $C_{i^*}^H$ is the right candidate. As a result, we have to try all vertices contained in $C_{i^*}^H$.

Now let us estimate the number of such vertices in the worst case. First of all, we can assume in the worst case that there is no EP $\Pi^G$ in $\mathscr{P}^G$ containing irrelevant cells, say, $C_1$ and $C_2$, otherwise we can simultaneously deal with those two cells and do not need to consider $C_2$ in dealing with $C_1$ and vice versa. Consequently, every cell of any $\Pi^G$ in $\mathscr{P}^G$ is split into at least two pieces with a new EP $\widetilde{\Pi}$ which is refiner than $\Pi$.

Secondly, note that at each step we always begin with one of cells not being singleton but of minimum order, so we should assume in the worst case that any two pieces of $\Pi$ obtained by a refinement $\widetilde{\Pi}$ are of orders different at most one, provided that none of them is a singleton.

The third observation is that in the worst case the equitable partition $\dEP{ x }{ X_{\lambda,k} }$ splits the cell $C_{i^*}^G$ into at least 3 pieces. It is clear that $\dEP{ x }{ X_{\lambda,k} }$ splits  $C_{i^*}^G$ into at least 2 pieces and one of them consists of the vertex $x$, so if there are only two cells in $\dEP{ x }{ X_{\lambda,k} }$, the union of which is equal to $C_{i^*}^G$, then one can easily see that $(\drm{Aut}{G})|_{C_{i^*}^G} \cong \drm{Sym}{C_{i^*}^G}$, where $(\drm{Aut}{G})|_{C_{i^*}^G}$ stands for the restriction of $\drm{Aut}{G}$ to $C_{i^*}^G$. Consequently, we can omit the cell $C_{i^*}^G$ in figuring out the set $\mathscr{P}^G$. 

According to three properties above, there are at most $(n/2)^{\log n}$ such vertices needed to be tried, and therefore by virtue of the algorithm we can determine whether or not $G \cong H$ and in the case of being isomorphic figure out one isomorphism from $G$ to $H$ within time $n^{C \log n}$. Combing two algorithms presented at the end of section 2, showing how to find out all block systems of $\mathfrak{G}$, and of section 4, showing how to decompose eigenspaces of $\dAM{G}$ into IRs of $\mathfrak{G}$, we finally establish the theorem \ref{MainTheorem}.

\vspace{2mm}
In dealing the problem of graph isomorphism, we only consider simple graphs up till now, but it is clear that our approach can also deal with graphs with some weight function on its vertices or edges. As matter of fact, so long as the adjacency matrix of the graph considered is symmetric our method is effective.

\vspace{6mm}
\noindent{\Large\bf Acknowledgments}

\vspace{3mm}
I would like to express my deep gratitude to Prof. Fu-Ji Zhang, Prof. Xue-Liang Li, Prof. Qiong-Xiang Huang, Prof. Sheng-Gui Zhang, Prof. Li-Gong Wang and Prof. Johannes Siemons for their valuable advice which significantly improves the quality of this paper. I also want to thank Prof. Yi-Zheng Fan and Prof. Xiang-Feng Pan for their encouragement and support. Last but not the least, I would like to thank Dr. You Lu, Dr. Yan-Dong Bai, Dr. Bin-Long Li and Dr. Xiao-Gang Liu for helping me verify many parts of this paper.


\begin{thebibliography}{99}

\bibitem{Axler} Sheldon Axler, {\it Linear algebra - done right}, Springer-Verlag New York, 1997.

\bibitem{Babai} L. Babai, Graph isomorphism in quasipolynomial time, arXiv:1512.03547v2.

\bibitem{BaGrMu} L. Babai, D.Yu. Grigoryev and D.M. Mount, Isomorphism of graphs with bounded eigenvalue multiplicity, {\it Proc. 14th ACM Symposium on Theory of Computing} (ACM, New York): 310-324, 1982.

\bibitem{BHZ} R. Boppana, J. Hastad, and S. Zachos, Does co-NP have short interactive proofs? Information Processing Letters, 25(2):27–32, 1987.


\bibitem{DM} John D. Dixon and Brian Mortimer, {\it Permutation groups} (GTM 163), Springer-Verlag New York, 1996.

\bibitem{GodRoy} Chris Godsil and Gordon Royle, {\it Algebraic Graph Theory} (GTM 207), Springer-Verlag New York, 2001.

\bibitem{GolubLoan} Gene H. Golub and Charles F. Van Loan, {\it Matrix Computations}, Johns Hopkins Univ. Press, Baltimore, Maryland, 1996 (3rd edition).

\bibitem{Luks} E. Luks, Isomorphism of bounded valence can be tested in polynomial time, Journal of Computer and System Sciences, 25:42–65, 1982.

\bibitem{McP} Brendan D. McKay and Adolfo Piperno: Practical Graph Isomoprhism, II. arXiv:1301.1493, 2013.

\bibitem{PanChenZheng} Victor Y. Pan, Zhao Q, Chen and Ailong Zheng, The complexity of the algebraic eigenproblem, STOC 1999: 507-516. 

\bibitem{Schonig} U. Sch$\mathrm{\ddot{o}}$ning, Graph isomorphism is in the low hierarchy, Journal of Computer and System Sciences, 37:312–323, 1988.

\bibitem{Serre} Jean-Pierre Serre, {\it Linear representations of finite groups}, Springer-Verlag New York, 1977.


\end{thebibliography}
\end{document}